\documentclass[a4paper,12pt,leqno]{article}
\usepackage[a4paper,centering,vscale=0.77,hscale=0.77]{geometry}
\usepackage{amsmath,amsthm,amssymb,bbm,bm}
\usepackage{xcolor}
\usepackage{verbatim}
\usepackage{todonotes,ulem}
\usepackage{mathrsfs}

\usepackage{thmtools}
\usepackage[pdfdisplaydoctitle,colorlinks,breaklinks,urlcolor=blue,linkcolor=blue,citecolor=blue]{hyperref}
\usepackage{enumitem}   
\numberwithin{equation}{section}

\newtheorem{theorem}{Theorem}[section]
\newtheorem{lemma}[theorem]{Lemma}
\newtheorem{proposition}[theorem]{Proposition}

\newtheorem{definition}[theorem]{Definition}
\newtheorem{remark}[theorem]{Remark}
\newtheorem{assumption}[theorem]{Assumption}

\newcommand{\df }{\mathrm{d}}
\newcommand{\im }{\mathrm{i}}
\newcommand{\Real}{\operatorname{Re}}
\newcommand{\dualityReal}[2]{\Real \langle #1, #2 \rangle}
\newcommand{\productReal}[2]{\Real \left( #1, #2 \right)_H}
\newcommand{\Prob}{\mathbb{P}}
\newcommand{\Filtration}{\mathbb{F}}
\newcommand{\E}{\mathbb{E}}
\newcommand{\F}{\mathscr{F}}
\newcommand{\EA}{{V}}
\newcommand{\EAdual}{{V^*}}

\title{Stationary solutions for the nonlinear Schr\"odinger~equation}
\author{Benedetta Ferrario\thanks{Dipartimento di Scienze Economiche e Aziendali, Universit\`a di Pavia, 27100 Pavia, Italy  \textit{E-mail address}: benedetta.ferrario@unipv.it }  
   	\and Margherita Zanella\thanks{\textit{Corresponding author.} Dipartimento di Matematica "Francesco Brioschi", Politecnico di Milano, Via Bonardi 13, 20133 Milano, Italy  \textit{E-mail address}: margherita.zanella@polimi.it	}		}
\date{\today}

\begin{document}

\maketitle

\tableofcontents

\begin{abstract}
We construct stationary statistical solutions of a  deterministic unforced nonlinear Schr\"o\-din\-ger equation, 
by perturbing it by a linear damping $\gamma u$ and a stochastic force whose intensity is proportional to $\sqrt \gamma$, and then letting  $\gamma\to 0^+$. 
We prove indeed that the family of stationary solutions $\{U_\gamma\}_{\gamma>0}$
of the  perturbed equation possesses an accumulation point for any vanishing sequence  $\gamma_j\to 0^+$ and this stationary limit solves 
the deterministic unforced  nonlinear Schr\"o\-din\-ger equation and  is not a   trivial process.
This technique has been introduced in  \cite{KuksinS}, using a different  dissipation. 
However, considering a  linear damping of zero order and weaker solutions, we can deal with larger  ranges of the nonlinearity and of the spatial dimension; moreover we  consider  the focusing equation and the defocusing 
 equation as well.
\end{abstract}


\bigskip

%%  INTRODUCTION  %%
\section{Introduction}

The  nonlinear Schr\"odinger (NLS) equation occurs as a basic model in many areas of physics: 
hydrodynamics, plasma physics, optics, molecular biology, chemical reaction, etc. 
It describes the propagation of waves in media with both nonlinear and dispersive responses (see, e.g., \cite{Cazenave,Sulem} and references therein). 
When the nonlinearity is of polynomial type it reads as 
\begin{equation} \label{eq0}
\partial_t u+ \im \Delta  u+\im \alpha |u|^{2\sigma} u = f
\end{equation}
where the unknown is $u=u(t,x):\mathbb R\times D \to\mathbb C$ with $D\subset\mathbb R^d$. 
$\Delta$ is the spatial Laplacian and 
the parameters are  $\sigma\ge 0$ and
$\alpha \in \{-1,1\}$;  for
$\alpha=1$ this is called the focusing equation and for $\alpha=-1$ this is  the defocusing one.
 
Fractional powers of the Laplacian have been introduced by Laskin \cite{La}.
Hence we consider the more general equation
\begin{equation}\label{eq-beta}
\partial_t u- \im (-\Delta)^\beta  u+\im \alpha |u|^{2\sigma} u= f
\end{equation}
for $\beta>0$.
When there is no forcing term ($f=0$), 
for smooth enough solutions there are two conserved quantities: the mass and the energy.

Attempts to define invariant Gibbs measures by means of these  conserved quantities for the NLS equation 
\eqref{eq0}
 have been done by 
Bourgain  in \cite{Bourgain96} for the cubic defocusing case and by 
Tzvetkov  in \cite{Tzvetkov06} for the sub-cubic (focusing or defocusing)  case  and 
in  \cite{Tzvetkov} for the 
 defocusing sub-quintic case.
 A recent preprint by Casteras and Monsaingeon \cite{CM} deals with the 
 fractional Schr\"odinger equation with a Moser-Trudinger type nonlinearity.
It is well known that this approach to define  stationary (in time) 
statistical solutions presents several difficult issues.

On the other side, one can look for stationary statistical solutions of  the unforced NLS equation
by means of stochastic analysis as done by Kuskin and collaborators. Their technique 
works for  infinite dimensional Hamiltonian systems with at least two invariant quantities;  in \cite{K04}
Kuksin started by considering  the equations of fluid dynamics developed 
in many papers  by Kuksin and collaborators  later on (see also \cite{Fer}
for the the equations of fluid dynamics with fractional dissipation), then  they dealt  with 
the KdV  equation (see \cite{KP-KdV,Kuksin-KdV}) 
as well as the NLS equation    (see \cite{Kuksin99,KuksinS,Kuksin,Sh11}).  
 
The technique is the following. First,  equation  \eqref{eqS} is perturbed 
by adding a dissipative term and a stochastic force. In this paper we consider a linear dissipative
term of zero  order, called  damping term from now on, i.e.
\begin{equation}\label{eqsS}
\partial_t u- \im (-\Delta)^\beta  u+\im \alpha |u|^{2\sigma} u +\gamma u
= \sqrt \gamma\Phi  \partial_t W,
\end{equation}
for $\gamma>0$. 
The right hand side  represents a random force which is white in time and smooth in 
space, with the spatial regularity  depending on the  operator $\Phi$. When $\gamma=0$ we recover 
the unforced NLS equation
\begin{equation}\label{eqS}
\partial_t u- \im (-\Delta)^\beta   u+\im \alpha |u|^{2\sigma} u = 0.
\end{equation}
Since the perturbing terms produce the stochastic  linear  equation
\begin{equation}
\partial_t  u+\gamma u
= \sqrt \gamma\Phi  \partial_t  W,
\end{equation}
for which an invariant measure is known to exists (a Gaussian measure, independent of $\gamma$), it is natural to 
 expect that there exist stationary solutions of the damped  and driven NLS equation \eqref{eqsS}. 
 This will indeed  be proved as well as  bounds of the moments, independent of $\gamma$.
Then, given a family $\{U_\gamma\}_{\gamma>0}$ of stationary solutions for the perturbed equation \eqref{eqsS}
one considers the limit  as $\gamma\to 0^+$.
The main result consists in proving that 
 the limit,  obtained considering  any vanishing (sub)sequence $\gamma_j\to 0$,  is a stationary solution obeying the 
deterministic unforced equation \eqref{eqS}, and this process is not a trivial solution.

\smallskip
The contents of the paper are structured as follows. In Section \ref{s-2} we introduce the mathematical setting, the assumptions and the basic results on the NLS. 
In Section \ref{sect-existence} we prove the existence  of stationary 
martingale solutions for the stochastic damped NLS equation \eqref{eqsS}. 
Let us point out that in the literature there are  stronger results, i.e.  on  invariant measures instead of 
stationary solutions. In particular when $D=\mathbb R^d$ the existence of invariant measures has been proven in  
\cite{Ekren_2017,Kim} and the uniqueness for large damping $\gamma$ in \cite{large-damping}. 
In a bounded domain $D$ there are less results;  for $d=1$ Debussche and Odasso \cite{DebO} have proven the 
existence of a unique invariant measure when the noise is additive, $\sigma=1$ and $\alpha=-1$; for $d=2$, $\alpha=1$ 
and in a more general setting for the noise there is our  existence result with Brze{\'{z}}niak 
 \cite{noi2d}. 
Therefore our results of Section \ref{sect-existence} on the existence of stationary  martingale solutions are new 
if $d\ge 3$ for the defocusing equation and if $d \ge 2$ for the focusing equation. Moreover we provide 
precise moment estimates for the mass and the energy.
In Section \ref{il} we consider the limit of vanishing damping. We prove that
any vanishing sequence has a subsequence converging, as $\gamma_j\to 0$, to a stochastic process $U$, 
which is a  non-trivial stationary solution of the deterministic unforced Schr\"odinger equation \eqref{eqS}.
 Indeed the mean of its mass is expressed in terms of the spatial covariance of the noise term.
Hence  noises with different covariances provide different limits.  Moreover, the random variable $U(t)$ (at fixed time $t$) is not a trivial random variable and  for a full noise its mass is absolutely continuous with respect to the Lebesgue measure on $[0,+\infty)$.
Section \ref{more_reg_sec} considers more regular solutions %in dimension $d \le 2$. 
under specific conditions on the dimension $d$ and the fractional power $\beta$. This covers in particular the cases $d=1,2$ and $\beta=1$, more common in the literature.
In Section \ref{final} we compare our results  on the NLS equation \eqref{eqS} 
with those of Kuksin and Shirikyan in  \cite{KuksinS}. % and \textcolor{blue}{those of Sy and Yu in \cite{S} and \cite{SY}}. 
Basic results on the Faedo-Galerkin approximation are collected in Appendix \ref{app-Gal}. In Appendix \ref{section-compactness} we recall a compactness criterium used throughout the paper.
Appendix \ref{domini} contains a characterization of the domain of the 
fractional Laplacian  and  provides references for the extension of our results either to  a 
$d$-dimensional bounded spatial  domain with  Neumann boundary conditions
or to a  $d$-dimensional bounded manifold without boundary. Indeed, for the sake of exposition, the results 
are presented in the case of a $d$-dimensional bounded domain with Dirichlet boundary conditions.
Appendix \ref{section-dense} presents the properties of the random variable $U(t)$ and its mass, being $U$ any stationary process obtained in the limit for vanishing $\gamma$.

%%%%%%%%%%%%%%%%%%%%%%%%%%%%%
\section{Mathematical setting and assumptions}  \label{s-2}

In this section we first introduce the spaces and the operators we work with, 
then we state the assumptions on the parameter $\sigma$ defining the nonlinearity 
and the assumptions on the stochastic forcing term $\Phi dW$.
Then in \S \ref{s-abstract-eq} we write the  NLS equations \eqref{eqsS} and \eqref{eqS} in  the abstract form and revise their basic properties. In \S \ref{section-m-e} we deal with the mass and the energy functionals.

\noindent{\bf The spaces}\\
Let $D$ be an open bounded subset of $\mathbb R^d$ with $C^\infty$-boundary. 
By ${\mathbb L}^q$, for $q\in [1,\infty]$, we denote the space of equivalence classes of $q$-Lebesgue integrable functions $u:D\to\mathbb C$.
 We  abbreviate ${H}:={\mathbb L}^2$.
This is a complex Hilbert space with inner product $(u,v)_H=\int_D u(x) \overline{v(x)} \,{\rm d}x$.
However,  we can interpret $H$ as  a real Hilbert space with the inner product $\Real(u,v)_H$.
They are different but  in one-to-one correspondence: $(u,v)_H=\Real (u,v)_H+i\Real (u,iv)_H$.
These inner products introduce the same norms and hence both spaces are topologically equivalent.
%}

For $m\in \mathbb N$, we denote   by $H^m(D)$   the Hilbert space  of functions  of $H$ whose
derivatives till order $m$   are in $H$.
For real  non integer $b>0$   we define 
the spaces $H^b(D)$ by interpolation
 (see, e.g., \cite[Chapter 1]{LionsMag}).
 We denote by  $C^{\infty}_0(D)$  the set of smooth functions defined
over $D$ with compact support, and by 
$H_0^{b}(D)$ the closure  of $C^{\infty}_0(D)$  in $H^{b}(D)$.
For $b<0$ we define by duality the space $\left(H^{b}(D)\right)^*:=H^{-b}_0(D)$.
For $b>0$ we have 
\[
H^{b}_0(D)\subseteq  H^{b}(D) \subset H \subset H^{-b}(D).
\]
Moreover, we recall the  continuous embeddings 
\begin{equation}
\label{emb}
\begin{split}
\text{ for }d>2b : &\;H^b(D) \subset {\mathbb L}^{2+2s}\quad \text{ if } s=\frac{2b}{d-2b}
\\
\text{ for }d=2b : &\;H^b(D) \subset {\mathbb L}^{p} \qquad  \forall \ p\in [2, \infty)
\\
\text{ for }d<2b : &\;H^b(D) \subset {\mathbb L}^{\infty}
\end{split}
\end{equation}
and the compact embedding $H^b(D) \subset H^{b-\epsilon}(D) $ for any $\epsilon >0$.

\noindent{\bf The linear operator}\\
We consider  the negative Laplacian operator  with the homogeneous Dirichlet boundary condition $-\Delta_{\text{Dir}}$. This is  
a linear   unbounded  positive self-adjoint operator in $H$ with domain $\mathcal{D}\left(-\Delta_{\text{Dir}}\right)
= H^2(D)\cap H^1_0(D)$. 
We denote by $\{\lambda_j\}_{j \in \mathbb{N}}$ the sequence of the eigenvalues of this operator. This is a nondecreasing  sequence  with $\lambda_j >0$ and $\lambda_j\rightarrow \infty$ as $j \rightarrow \infty$. By $\{e_j\}_{j \in \mathbb{N}}$ we denote the sequence of the corresponding eigenfunctions. This is a complete orthonormal basis in $H$.
For any real $b>0$, the fractional operator $(-\Delta_{\text{Dir}})^b$ is an unbounded positive self-adjoint operator in $H$, by the functional calculus. We can write 
\[
(-\Delta_{\text{Dir}})^b v = \sum_{j=1}^\infty \lambda_j^b (v, e_j)_H e_j
\]
of domain 
\[
\mathcal{D}\left((-\Delta_{\text{Dir}})^b\right) = \left\{ v \in H: \sum_{j=1}^\infty |\lambda_j|^{2b} |(v, e_j)_H|^2 < +\infty\right\}.
\]
For a characterization of the space $\mathcal{D}\left((-\Delta_{\text{Dir}})^b\right)$ see %\cite[Theorem 2]{Grubb} and the 
Appendix \ref{domini}. 

Given the parameter $\beta>0$ appearing in equation \eqref{eqsS}, we set
\[
A:=(-\Delta)^\beta \qquad \text{and} \qquad  \EA:=\mathcal D(A^{\frac 12}).
\]
The space $\EA$ is a Hilbert space when equipped with the inner product 
\[
( u,v)_{V}=(u,v)_H+(A^{\frac 12}u,A^{\frac 12}v)_H.
\]
We call $V$ the energy space. 
We denote the dual space of $V$ by $V^*$ and abbreviate by $\langle\cdot,\cdot\rangle$ the duality bracket between $V$ and $V^*$, where the complex conjugation is taken over the second variable of the duality.\footnote{We will use the same notation for the duality ${\mathbb L}^{p}- {\mathbb L}^{p^\prime}$ for $\frac 1p+\frac 1{p^\prime}=1$, $1<p<\infty$.}
Note that $(\EA,H,\EAdual)$  is a Gelfand triple, i.e.
\[
\EA\hookrightarrow H \cong H^* \hookrightarrow \EAdual.
\]
We introduce the operator  $\hat A$ as
\[
\langle \hat A u,v\rangle=(A^{\frac 12 }u,A^{\frac 12 }v)_ H, \qquad u,v \in \EA.
\]
As proved in  \cite[Lemma 2.3(a)]{BHW-2019}, this is  a positive self-adjoint operator on $\EAdual$  with domain 
  $\mathcal D(\hat A) = \EA$ and $\hat A = A$ on $\mathcal D(A)$. 
In what follows, in most cases, where does not cause ambiguity or confusion, we use the notation $A$ for $\hat A$.

 \noindent{\bf The  nonlinear operator}\\
 We write the nonlinearity as 
 \begin{equation}\label{F}
F(u):= |u|^{2\sigma}u.
\end{equation} 
We have that $F: {\mathbb L}^{2+2\sigma} \to {\mathbb L}^{\frac{2+2\sigma}{1+2\sigma}}$ and
\begin{equation}\label{stimaF}
\|F(u)\|_{{\mathbb L}^{\frac{2+2\sigma}{1+2\sigma}}}
\lesssim \|u\|^{1+2\sigma}_{{\mathbb L}^{2+2\sigma}} .
\end{equation} 
Since 
\begin{equation}
\label{stima_F_L_infty}
\left| |a|^{2\sigma}a-|b|^{2\sigma}b\right |\lesssim_\sigma (|a|^{2\sigma}+|b|^{2\sigma})|a-b|
\end{equation}
for any $a,b\in \mathbb C$, we get by means of the H\"older inequality that $F$ is locally 
Lipschitz, i.e. 
\begin{equation}\label{stimaF-lipschitz}
\|F(u)-F(v)\|_{{\mathbb L}^{\frac{2+2\sigma}{1+2\sigma}}}
\lesssim_\sigma  \left( \|u\|_{{\mathbb L}^{2+2\sigma}}^{2\sigma}+  \|v\|_{{\mathbb L}^{2+2\sigma}}^{2\sigma}\right)
                \|u-v\|_{{\mathbb L}^{2+2\sigma}} .
\end{equation} 
Let us explain the latter notation.  If  functions $a,b\ge 0$ satisfy the inequality $a \le C_A b$ with a constant $C_A>0$ depending on the expression $A$, we write $a \lesssim_A b$; for a generic constant we put no subscript.
 If we have $a \lesssim_A b$ and $b \lesssim_A a$, we write $a \simeq_A  b$.
Moreover
 we shall denote by $C$ a generic positive constant, which might vary from one line to the other, 
except $G$ which is the particular constant in  the next estimate \eqref{constanteG} 
coming from the Gagliardo-Niremberg inequality.   However, if a constant depends on some parameter $p$ and we want to highlight this dependence,  we write $C_p$.

We work under the following assumption on the nonlinearity.

\begin{assumption}[on the nonlinearity \eqref{F}]
\label{ass-sigma}
 \hfill
\begin{itemize}
\item If $\alpha=1$ (focusing), let $0< \sigma<\frac {2\beta}d$.
\item If $\alpha=-1$ (defocusing), let $\begin{cases}0< \sigma< \frac {2\beta}{d-2\beta},& \text{ for }d>2\beta\\\sigma> 0, &  \text{ for }d\le 2\beta\end{cases}$
\end{itemize}
\end{assumption}
For the values of $\sigma$ fulfilling Assumption \ref{ass-sigma}
there is the compact embedding $H^\beta(D) \subset {\mathbb L}^{2+2\sigma}$ (see \eqref{emb}); 
 a fortiori the embedding 
\begin{equation}\label{H1-e-Lsigma}
\EA \subset {\mathbb L}^{2+2\sigma}
\end{equation} 
is compact.

Bearing in mind the continuous embeddings 
\begin{equation}\label{embedd-V-L}
\EA\subset {\mathbb L}^{2+2\sigma}\subset
H \equiv H^* \subset  ({\mathbb L}^{2+2\sigma})^*= {\mathbb L}^{\frac{2+2\sigma}{1+2\sigma}} \subset \EAdual,
\end{equation} 
from \eqref{stimaF} we also get that  $F:\EA \to \EAdual$ with
\begin{equation}\label{F-grandi}
\|F(u)\|_{\EAdual} 
\lesssim 
\|F(u)\|_{{\mathbb L}^{\frac{2+2\sigma}{1+2\sigma}}}
\underset{\eqref{stimaF}}{\lesssim}
\|u\|^{1+2\sigma}_{{\mathbb L}^{2+2\sigma}}
\lesssim
\|u\|_\EA^{1+2\sigma}.
\end{equation}

\noindent{\bf The stochastic forcing term}\\
Given two  separable Banach  spaces $X_1$ and $X_2$, we denote by
$\mathcal L(X_1,X_2)$ 
the space of  linear operators from $X_1$ to $X_2$. Moreover, 
 $\gamma(X_1,X_2)$  is 
the space of  $\gamma$-radonifying operators from $X_1$ to $X_2$.
If  $X_1$ and $X_2$ are Hilbert spaces, this is the 
space $L_{HS}(X_1,X_2)$  of  Hilbert-Schmidt operators from $X_1$ to $X_2$.

As far as the stochastic term is concerned, we consider  a complete probability space
 $(\Omega,\mathcal{F},\mathbb{P})$ and a real Hilbert space $Y$
with an  orthonormal basis  $\{e^Y_j\}_{j\in\mathbb{Z}_0}$, where ${\mathbb Z}_0=\mathbb Z\setminus \{0\}$. We use the double series for ease of notation in the following expression of $\Phi W$.
Let $W$ be  a $Y$-canonical cylindrical Wiener process adapted to a filtration $\mathbb{F}$ satisfying the standard
 conditions. We can write it formally as a series
\[
W(t)= \sum_{j\in \mathbb Z_0}W_j(t) e^Y_j, 
\]
with $\{W_j\}_{j\in \mathbb Z_0}$  a sequence of i.i.d. real Wiener processes (see, e.g., \cite[Ch 4]{dpz}). Hence, 
for a given linear operator  $\Phi:Y\to \EA$ we have
\begin{equation}\label{def-PhiW}
\Phi W(t)= \sum_{j\in\mathbb{Z}_0}W_j(t) \Phi e^Y_j.
\end{equation}
In particular we consider $\Phi$ of the following type
\begin{equation}\label{espressione-del-noise}
\Phi e^Y_j=\begin{cases} \varphi_j e_j, & j >0\\ \im \varphi_{j} e_{-j}, & j<0\end{cases}
\end{equation}
with $\varphi_j\in \mathbb R$ for any $j\in\mathbb Z_0$;
so \eqref{def-PhiW} becomes 
\begin{equation}\label{laserie}
\Phi W(t)= \sum_{j=1}^\infty  \left( \varphi_j W_j(t)+\im \varphi_{-j} W_{-j}(t)\right) e_j.
\end{equation}

We work under the following assumptions on the noise. We denote by 
$\|\Phi\|_{L_{HS}(Y,\EA)} := \big(\sum_{j\in{\mathbb Z}_0} \|\Phi e^Y_j\|^2_\EA\big)^{1/2}$ the Hilbert-Schmidt norm of the operator  $\Phi:Y\to \EA$.
\begin{assumption}[on the noise] \label{ass_G}
 $\Phi:Y\to \EA$ is a  Hilbert-Schmidt operator, i.e.
\begin{equation}\label{noise-nell-assumption} 
\sum_{j=1}^\infty  (1+\lambda_j^\beta) (\varphi_j^2+\varphi_{-j}^2)<\infty.
\end{equation}
\end{assumption}
Of course the operator $\Phi$ is also a $\gamma$-radonifying operator from $Y$ to ${\mathbb L}^{2+2\sigma}$
and a 
 Hilbert-Schmidt operator from $Y$ to $H$, with
\[
\|\Phi\|_{L_{HS}(Y,H)}: = \big(\sum_{j\in{\mathbb Z}_0}  \|\Phi e^Y_j\|^2_H\big)^{1/2}=
 \big(\sum_{j=1}^\infty (\varphi_j^2+\varphi_{-j}^2)\big)^{1/2} < \infty.
\]
  
 \begin{remark}
In order to compare our setting with the more general one of our previous paper \cite{noi2d} 
for $d=2$ and $\beta=1$, we point out that in that case
 $\Phi$ is also a  $\gamma$-radonifying operator from $Y$ to  ${\mathbb L}^{p}$ for any finite $p$. 
\end{remark}

We collect the basic properties of the Wiener process \eqref{def-PhiW} in the following lemma, showing that it 
 takes values in $C^a([0,\infty);\EA)$, $\mathbb{P}$-a.s., where 
the H\"older space $C^a([0,T];\EA)$ is defined in \eqref{defHa}.
%LEMMA
\begin{lemma}
\label{lemma-PhiW}
Under Assumption \ref{ass_G} we have
\begin{equation}
\label{stima-classica-W}
\E \|\Phi W(t)\|_{\EA}^2=t \|\Phi\|_{L_{HS}(Y,\EA)}^2 \qquad \text{ for any }t\ge 0
\end{equation}
and for any arbitrary $a\in [0,\frac 12)$ and $T>0$
\begin{equation}\label{PhiW-Holder}
\E\|\Phi W\|_{C^a([0,T];\EA)}\lesssim_{a,T} \|\Phi\|_{L_{HS}(Y,\EA)} .
\end{equation}
 \end{lemma}
\begin{proof} These are classical results, see  e.g.,   \cite[Chapter  4]{dpz}. 
 However we need a particular estimate in the next sections. So we provide  some details for the $C^a([0,\infty);\EA)$-norm. 
 
Equation \eqref{stima-classica-W} can be proved by direct computation starting from the definition 
\eqref{def-PhiW}. Moreover, 
since the Wiener process has normal distribution, it easily follows that a similar relation holds for any power: 
 for any $m\in\mathbb N$ there exists a constant $C_m$ such that
 \[
 \E \|\Phi W(t)\|_{\EA}^{2m}\le C_m t^m  \|\Phi\|_{L_{HS}(Y,\EA)}^{2m}.
 \]
 Therefore 
 \begin{equation}\label{PhiW-in-L}
 \E\|\Phi W\|^{2m}_{L^{2m}(0,T;V)}=
 \E\int_0^T  \|\Phi W(t)\|_{\EA}^{2m}\  {\rm d}t \le  \frac{C_m}{m+1}   \|\Phi\|_{L_{HS}(Y,\EA)}^{2m} T ^{m+1}.
 \end{equation}

Now, for $\alpha\in(0,1)$ and $q\in (1,\infty)$
we define the space $W^{\alpha,q}(0,T;\EA)$ as the space of functions $f\in L^q(0,T;\EA)$ such that 
\[
\int_0^T \int_0^T \frac{\|f(t)-f(s)\|^q_\EA}{|t-s|^{1+\alpha q}} \ {\rm d}s \  {\rm d}t
<\infty.
\]
For any $q\in \mathbb N$ we have
\begin{equation}\label{PhiW-in-W}
\begin{split}
\E \int_0^T \int_0^T \frac{\|\Phi W(t)-\Phi W(s)\|^q_\EA}{|t-s|^{1+\alpha q}} {\rm d}s {\rm d}t
&= \int_0^T \int_0^T \frac{\E\|\Phi W(t)-\Phi W(s)\|^q_\EA}{|t-s|^{1+\alpha q}} {\rm d}s {\rm d}t
\\&
\lesssim_q   \|\Phi\|_{L_{HS}(Y,\EA)}^q  \int_0^T \int_0^T \frac{|t-s|^{\frac q2}}{|t-s|^{1+\alpha q}}
 \ {\rm d}s \  {\rm d}t
\end{split}
\end{equation}
which is finite if and only if $\frac q2-\alpha q>0$, i.e. $\alpha<\frac 12$.
Recalling that  for $\alpha q>1$ there is the continuous embedding
\[
W^{\alpha,q}(0,T;\EA)\subset C^{\alpha-\frac 1q}([0,T];\EA),
\]
from \eqref{PhiW-in-L} with $2m\ge q$  and  from \eqref{PhiW-in-W} we get
\begin{equation}\label{stima_PhiW-Ca}
\E \|\Phi W\|_{C^{\alpha-\frac 1q}([0,T];\EA)}
\lesssim_{T,\alpha,q}  \|\Phi\|_{L_{HS}(Y,\EA)},
\end{equation}
when 
\begin{equation}\label{alpha-q}
0< \alpha<\tfrac 12,\quad\alpha q>1.
\end{equation}
This condition requires $q>2$.
Therefore, if the parameters $\alpha$ and $q$ fulfil \eqref{alpha-q}, then  the process $ \{\Phi W(t)\}_{t\ge 0}$ takes values in $C^{\alpha-\frac 1q}([0,\infty);\EA)$, $\mathbb{P}$-a.s.. Finally, given any $a\in (0,\frac 12)$
there exist $\alpha$ and $q$ fulfilling the conditions \eqref{alpha-q} such that $a=\alpha-\frac 1q$.
\end{proof}

% % % % % % % % % %%%%%%%%%%%%%%%
\subsection{The equation in abstract form}  \label{s-abstract-eq}
We consider the stochastic NLS equation with additive noise
 in a more general setting , i.e.  without imposing that its intensity depends on the damping parameter
 $\gamma$ as in \eqref{eqsS}, and we write it  in the abstract form as 
\begin{equation}\label{EQSgen}
\begin{cases}
{\rm d} u(t)+\left[ -\im A u(t)+\im \alpha F( u(t)) +\gamma u(t) \right] \,{\rm d}t 
= \Phi  \,{\rm d} W(t),\qquad t>0
\\
u(0)=u^{0}
\end{cases}
\end{equation}

Let us denote by $C_w([0,T];\EA)$  the space
of all continuous functions from the interval $[0,T]$ to  the space $\EA$ endowed with the weak topology; 
it is equipped  with the locally convex topology induced by the family 
of seminorms given by $\displaystyle\sup_{0\le t\le T} |\langle u(t),v\rangle|$ for $v\in \EAdual$.
The Borel $\sigma$-algebra on $C_w([0,T];\EA)$ is the $\sigma$-algebra generated by the open sets in this locally convex topology. 

We now define what is a martingale solution and a stationary martingale solution for the NLS equation \eqref{EQSgen}.

%DEF martingale solution
\begin{definition}[martingale solution]\label{def-martingale solution}
Let $\mu$ be a probability measure on $\EA$.
A \emph{martingale solution} of the equation $\eqref{EQSgen}$ on the time interval $[0,T]$ 
with initial value of law $\mu$ is a system
\[
\bigl(\tilde{\Omega},\tilde{\F},\tilde{\mathbb P},\tilde{W},\tilde{\Filtration},u\bigr)
\]
 consisting of
	\begin{itemize}
		\item a filtered probability space $\bigl(\tilde{\Omega},\tilde{\F},\tilde{\Prob},\tilde{\Filtration}\bigr)$,   satisfying the usual conditions, i.e.
 the filtration $\tilde{\Filtration}=\bigl(\tilde{\F}_t\bigr)_{t\in [0,T]}$ is  right-continuous and such that all $\tilde{\Prob}$-negligible  sets of $\tilde{\F}$ are elements of $\tilde{\F}_0$;
		\item  a $Y$-cylindrical   Wiener  processes $\tilde{W}$   on  $\bigl(\tilde{\Omega},\tilde{\F},\tilde{\Prob}\bigr);$
		\item an $H$-valued  continuous and  $\tilde{\Filtration}$-adapted process  $u$ with
$\tilde{\Prob}$-almost all paths in $ C_w([0,T];\EA)$,
	such that $u(0)$ has law $\mu$
and for every $t\in [0,T]$
 the equality
	\begin{equation}\label{eqn-ItoFormSolution}
	u(t)=  u(0)+ \int_0^t \left[\im A u(s)-\im \alpha F(u(s))-\gamma  u(s) \right] \df s     
	+\Phi W(t)
	\end{equation}
	holds $\tilde{\mathbb{P}}$-almost surely in $\EAdual$.
\end{itemize}
\end{definition}

%REMARK
\begin{remark}\label{rem-continuita-nella-def}
Since $\tilde{\Prob}$-almost all paths $u \in C_w([0,T];\EA)$, then 
$\tilde{\mathbb{P}}$-a.s. we have  $u(t)\in \EA$ for all $t\in [0,T]$; 
in particular $u(0)\in \EA$  and the requirement that its law equals $\mu$  is meaningful.

Moreover the space $C_w([0,T];\EA)$ is a subspace of $ L^{\infty}(0,T;\EA)$. 
Then the integral in the r.h.s. of \eqref{eqn-ItoFormSolution} is well defined.
In particular
\begin{align*}
&\int_0^\cdot  A u(s) \df s   \in C([0,T];\EAdual)
\\
&\int_0^\cdot F(u(s)) \df s   \in C([0,T];  {\mathbb L}^{\frac{2+2\sigma}{1+2\sigma}})\subset C([0,T];\EAdual) \quad\text{ by \eqref{stimaF} and  \eqref{embedd-V-L}}
\\
& \int_0^\cdot u(s) \df s   \in C([0,T];\EA)\subset  C([0,T];\EAdual)
\end{align*}
\end{remark}

In the sequel we will deal with strong solutions as well; by this we mean a process  $u$
fulfilling \eqref{eqn-ItoFormSolution} and the regularity specified above but we assume that  the  filtered probability space and the Wiener process 
are  a priori assigned.

As far as stationary martingale solutions are concerned, we give the definition involving the time interval $\mathbb R_+=[0+\infty)$. 
%Now the space $C_w([0,+\infty);\EA)$  is  equipped it  with the locally convex topology induced by the family of seminorms given by $\displaystyle\sup_{0\le t\le k} |\langle u(t),v\rangle|$ for $k \in \mathbb N$, $v\in \EAdual$.
%DEF stationary martingale solution
\begin{definition}[stationary martingale solution]\label{def- stationary martingale solution}
A \emph{stationary martingale solution} of the equation $\eqref{EQSgen}$ on the time interval
$[0,+\infty)$ is a system
\[
\bigl(\tilde{\Omega},\tilde{\F},\tilde{\mathbb P},\tilde{W},\tilde{\Filtration},u\bigr)
\]
 consisting of
	\begin{itemize}
		\item a filtered probability space $\bigl(\tilde{\Omega},\tilde{\F},\tilde{\Prob},\tilde{\Filtration}\bigr)$,   satisfying the usual conditions, i.e.
 the filtration $\tilde{\Filtration}=\bigl(\tilde{\F}_t\bigr)_{t\in [0,\infty)}$ is  right-continuous and such that all $\tilde{\Prob}$-null, i.e. $\tilde{\Prob}$-negligible,  sets of $\mathscr{F}$ are elements of $\tilde{\F}_0$;
		\item  a $Y$-cylindrical   Wiener  processes $\tilde{W}$   on  $\bigl(\tilde{\Omega},\tilde{\F},\tilde{\Prob}\bigr);$
		\item an  $H$-valued  continuous and  $\tilde{\Filtration}$-adapted process  $u$ with
 $\tilde{\Prob}$-almost all paths in $ C_w([0,\infty);\EA)$,
	such that for any $t>0$ equation \eqref{eqn-ItoFormSolution} holds 
	$\tilde{\mathbb{P}}$-almost surely in $\EAdual$. This is a stationary process 
	in $\EA$. 
\end{itemize}
\end{definition}
%REMARK

%{\color{purple}  \begin{remark}According to Remark \ref{rem-continuita-nella-def}, $\tilde{\mathbb{P}}$-a.s. we have  $u(t)\in \EA$ for all $t\ge 0$. Hence the stationarity condition in $\EA$ is meaningfull. In particular  for any $t\ge 0$ the law of  $u(t)$ is {\color{purple}the  measure $\mu$} on $\EA$.\end{remark}}

%%%%% subsection
\subsection{Conserved quantities for the unforced equation \eqref{eqS} }
\label{section-m-e}

We define the mass
\begin{align}
\label{mass}
&\mathcal{M}(u)=\|u\|^2_H
\\
\intertext{and the energy} 
\label{energy}
&\mathcal{E}(u)=\frac 12 \|A^{\frac 1 2} u\|^2_H -\frac{\alpha}{2+2\sigma}\|u\|_{{\mathbb L}^{2+2\sigma}}^{2+2\sigma}.
\end{align}
Thanks to \eqref{H1-e-Lsigma} they are both well defined on $\EA$. 
Moreover, if $u$ is a martingale solution, then 
the mappings $t\mapsto \mathcal{M}(u(t))$ and $t\mapsto\mathcal{E}(u(t))$ are in   $L^\infty(0,T)$.
 
Using the evolution equation \eqref{eqS},
now written in abstract form as 
\[
\frac{d u(t)}{dt}+\im A u(t)+\im \alpha F( u(t))
=0,
\]
one can prove (see, e.g.,  \cite{Cazenave}) that 
\[
\frac{{\rm d}\;}{{\rm d}t}\mathcal{M}(u(t))=0, \qquad
\frac{{\rm d}\;}{{\rm d}t} \mathcal{E}(u(t))=0,
\]
for suitably regular solutions $u$, that is the mass and the energy are constant in time.

For the defocusing equation ($\alpha=-1$) the energy is not negative and it dominates both 
$ \|A^{\frac 1 2} u\|_H^2$ and $ \|u\|_{{\mathbb L}^{2+2\sigma}}^{2+2\sigma} $. 
In order to have similar properties in the focusing case ($\alpha=1$) 
we introduce the  modified energy. Before defining it, let us recall the Gagliardo-Niremberg inequality
(see \cite{brezis})
\begin{equation}\label{GN-ineq}
\|u\|_{{\mathbb L}^{2+2\sigma}}
\le C \|u\|_{{\mathbb L}^2}^{1-\frac{\sigma d}{2\beta(1+\sigma)}} \| u\|_{H^\beta(D)}^{\frac{\sigma d}{2\beta(1+\sigma)}}
\qquad \forall u \in H^\beta(D)
\end{equation}
when $\sigma d<2\beta(1+\sigma)$. This is equivalent to 
\begin{equation}
\|u\|^{2+2\sigma}_{{\mathbb L}^{2+2\sigma}}
\lesssim  \|u\|_{{\mathbb L}^2}^{2(1+\sigma)-\frac{\sigma d}{\beta}} \| u\|_{H^\beta(D)}^{\frac{\sigma d}{\beta}}.
\end{equation}
In particular this holds for the values of $\sigma$ specified in the Assumption \ref{ass-sigma} (for $\alpha=1$).
Therefore, in the focusing case, thanks to the 
Young inequality for any $\epsilon>0$ there exists $C_\epsilon >0$ such that
\begin{equation}\label{GN-somma}
\|u \|_{{\mathbb L}^{2+2\sigma}}^{2+2\sigma}
\le
\epsilon \| u \|_\EA^2 +C_\epsilon
 \|u \|_{H}^{2+\frac{4\beta\sigma}{2\beta-\sigma d}}
 \qquad \forall u \in V.
\end{equation}
Hence for $\alpha=1$  we define the modified energy ${\mathcal E}_1:\EA\to\mathbb R$ as
\begin{equation}\label{modifEnergy}
\begin{split}
{\mathcal E}_1(u)&= \frac 1 2 \| u\|_\EA^2-\frac1{2+2\sigma} \|u\|_{{\mathbb L}^{2+2\sigma}}^{2+2\sigma} 
+G\|u\|_H^{2+\frac{4\beta\sigma}{2\beta-\sigma d}}
\\
&\equiv {\mathcal E}(u)+   \frac 1 2 \| u\|_H^2 + G\|u\|_H^{2+\frac{4\beta\sigma}{2\beta-\sigma d}},
\end{split}
 \end{equation}
where $G=G_{\beta,\sigma,d}$ is the constant appearing in the following particular form of \eqref{GN-somma}
\begin{equation}\label{constanteG}
\frac1{2+2\sigma} \|u\|_{{\mathbb L}^{2+2\sigma}}^{2+2\sigma} 
\le \frac 1{4+4\sigma} \| u\|_\EA^2+G \|u\|_H^{2+\frac{4\beta\sigma}{2\beta -\sigma d}}.
\end{equation}
Using this bound in \eqref{modifEnergy} we obtain  that
\begin{equation}\label{nabla-Htilde}
 {\mathcal E}_1(u)\ge \frac{1+2\sigma}{4+4\sigma} \| u\|_\EA^2\ge 0 
 \qquad \forall u\in \EA
\end{equation}
so the modified energy is non negative.
Moreover \eqref{constanteG} and \eqref{nabla-Htilde} imply that 
\begin{equation}\label{tildeH-domina}
\|u\|_{{\mathbb L}^{2+2\sigma}}^{2+2\sigma} \le
\frac {2+2\sigma}{1+2\sigma} {{\mathcal E}}_1(u)+(2+2\sigma)G\|u\|_H^{2+\frac{4\beta\sigma}{2\beta-\sigma d}} \qquad \forall u\in \EA.
\end{equation}

%\sout{Summing up, the energy in the defocusing case and the modified energy in the focusing case are both non negative quantities.}
In the following we will use the notation
\begin{equation}\label{due-energie}
{\mathcal E}_\alpha(u) =  {\mathcal E}(u)+  1_{\alpha=1}\left( \frac 1 2 \| u\|_H^2 + G\|u\|_H^{2+\frac{4\beta\sigma}{2\beta-\sigma d}}\right)
\end{equation}
for $\alpha \in \{-1,+1\}$.

%%%%%%%%%%%%%%%%
\section{Existence of stationary solutions for the stochastic equation}
\label{sect-existence}

In this Section we prove  that there exist    stationary martingale solutions
for  the stochastic nonlinear Schr\"odinger equation \eqref{EQSgen}. Moreover, for any  damping coefficient
$\gamma>0$,  we prove some mean  estimates, which are uniform with respect to the time variable 
 $t\in[0,+\infty)$. They quantify the effect of the damping term.  
 In particular, two estimates involving the mass and the energy will be crucial in the analysis of the limit as $\gamma\to 0$.
We do not deal with the uniqueness problem but some comments are given  in Remark \ref{oss-generale-esistenza}. 

We introduce a finite-dimensional approximation of the  stochastic nonlinear  Schr\"odinger  equation \eqref{EQSgen} in the space $ H_n=span(e_j:j\le n)$.

We define the projector operator $P_n:H\to H_n$ as
\[
P_n  u = \sum_{j=1}^n (u,e_j)_H \ e_j.
\]
To denote the scalar product in $H_n$ we use the same notation as for $H$.

The $P_n$'s do not destroy the conservation of mass and energy;
moreover, $P_nA=AP_n$. Thus we have
\begin{align}
\label{bound_H_Pn}
&\|P_n u\|_{H}
\le \|u\|_{H}, \quad u \in H%\qquad \qquad \text{and} \qquad\|P_n u\|_{\EAdual}\le \|u\|_\EAdual, \quad   u\in \EAdual,
\\
\label{bound_V_Pn}
&\|P_n u\|_{\EA}^2=\|P_n u\|_{H}^2+\|A^\frac12 P_n u\|_{H}^2
=\|P_n u\|_{H}^2+\|P_n A^\frac 12 u\|_{H}^2
\le \|u\|_{\EA}^2, \quad u \in \EA
\\\label{bound_Phi_Pn}
&\|P_n \Phi\|_{L_{HS}(Y,H)}\le \|\Phi\|_{L_{HS}(Y,H)}, \qquad
\|P_n \Phi\|_{L_{HS}(Y,\EA)}\le \|\Phi\|_{L_{HS}(Y,\EA)}.
\end{align}
 However, there is no uniform estimate in the Lebesgue space ${\mathbb L}^{2+2\sigma}$; 
to get indeed uniform estimates we use the continuous embedding $\EA\subset {\mathbb L}^{2+2\sigma}$, so

\begin{equation}
\label{bound_L_V}
\|P_n u\|_{{\mathbb L}^{2+2\sigma}} \lesssim  \|P_n u\|_\EA \le \| u\|_\EA, \qquad u \in \EA.
\end{equation} 
By density, we can extend $P_n$ to an operator $P_n :\EAdual\to H_n$ with $\|P_n\|_{\EAdual\to\EAdual}\le 1$ and 
\[
\langle v,P_n v\rangle \in \mathbb R, \qquad \langle v,P_n w\rangle=(P_n v,w)_H, \quad v\in\EAdual, w\in \EA.
\]

 We also have
\[\begin{split}
&\lim_{n\to\infty} \|P_nx-x\|_H=0, \qquad v\in H
\\
&\lim_{n\to\infty} \|P_nx-x\|_\EA=0, \qquad v\in \EA
\\
&\lim_{n\to\infty} \|P_nx-x\|_{{\mathbb L}^{2+2\sigma}}\le C \lim_{n\to\infty} \|P_nx-x\|_\EA= 0, , \qquad v\in \EA.
\end{split}\]

The Faedo-Galerkin approximation is obtained 
by projecting  the Schr\"odinger equation  \eqref{EQSgen}
onto the finite-dimensional space $H_n$, that is 
\begin{equation}\label{eq-Gal}
{\rm d} u_n(t)+\left[ -\im A  u_n(t)+\im  \alpha P_n F(u_n(t)) +\gamma u_n(t) \right] \,{\rm d}t 
=P_n \Phi  \,{\rm d} W(t).
\end{equation}
Also the initial condition $u_n(0)=u_n^0$ is obtained by projection onto $H_n$, i.e.
$u_n^0= P_n u^0 $.

It is a classical result to show that the Faedo-Galerkin equation has a unique solution. This is a strong solution in the probabilistic sense.
The main results are given and proven in Appendix \ref{app-Gal}.

Now we deal with the stationary solutions of the Faedo-Galerkin equation \eqref{eq-Gal} and provide moment estimates.

 We introduce the function
 \begin{equation}\label{phi_alpha}
\phi_{\alpha}(d,\beta, \sigma,\gamma, \Phi)
=
\|\Phi\|_{L_{HS}(Y,\EA)}^{2}
+ \|\Phi\|_{L_{HS}(Y,\EA)}^{2+2\sigma}\gamma^{-\sigma}
+\mathbbm{1}_{\alpha=1} \|\Phi\|^{2+\frac{4\beta\sigma}{2\beta-\sigma d}}_{L_{HS}(Y,H)} \gamma^{ -\frac{2\beta\sigma}{2\beta-\sigma d}},
\end{equation}
for $\alpha\in \{-1,+1\}$.
 
% PROPOSITION
\begin{proposition}
\label{prop-u-staz-Galerkin}
Under the Assumptions  \ref{ass-sigma} and \ref{ass_G}, 
 for any $n\in \mathbb N$
  there exists a stationary solution $u^{st}_n$ of  the Galerkin equation \eqref{eq-Gal}
  such that 
\begin{equation}
  \label{Gal_uniform-mass}
\E [{\mathcal M}(u^{st}_n(t))] =\frac1{2\gamma} \|P_n \Phi\|_{L_{HS}(Y,H)}^2 \qquad\text{ for all }t\ge 0.
\end{equation}
Moreover, for every $t\ge 0$ and $p\ge 1$ 
  \begin{align}
  \label{Gal_uniform-mass-p}
\E  [ \mathcal{M}(u^{st}_n(t))^{p}]
 &\lesssim_p \| P_n \Phi \|^{2p}_{L_{HS}(Y,H)} \gamma^{-p}
 \\
 \E  [{ \mathcal E}_\alpha(u^{st}_n(t))^p]
 &\lesssim_{d,\beta,\sigma}\phi_{\alpha}(d,\beta, \sigma,\gamma,P_n \Phi)^p \gamma^{-p} \qquad
 \text{ for } \alpha\in\{-1,1\},
 \label{Gal_uniform-energy}
 \end{align}
 and for every finite $T>0$  and $p\ge 1$, 
\begin{multline}
\label{Gal_sup-est-V}
\mathbb{E} \left[\sup_{0\le t\le T}\|u_n^{st}(t)\|^{2p}_{\EA} \right] \lesssim_{d, \sigma, p,T} 
%(1+ \mathbbm{1}_{\alpha=1}e^{\gamma T})
(\mathbbm{1}_{\alpha=-1}+ \mathbbm{1}_{\alpha=1} e^{C p \gamma T})
\\\times
\left(1+ \phi_{\alpha}(d,\beta, \sigma,\gamma, P_n \Phi)^p \gamma^{-p}
+ \|P_n \Phi\|_{L_{HS}(Y,\EA)}^{2p(1+\sigma)} 
+ \mathbbm{1}_{\alpha=1} \|P_n  \Phi\|_{L_{HS}(Y,H)}^{2p(1+\frac{2\beta\sigma}{2\beta-\sigma d})}
\right).
\end{multline}
\end{proposition}
 % PROOF
 \begin{proof}
We construct a stationary solution by means of the 
Krylov-Bogoliubov's technique, which provides an invariant measure $\pi_n $ for the Galerkin  dynamics.
Let us consider the Galerkin equation  \eqref{eq-Gal} with vanishing initial data $u_n^{0}=0$. 
Denote by  $\mathcal L(u_n(t;0))$ the law of this solution process  evaluated at  time $t>0$, 
living in  the finite-dimensional space $H_n$.
We construct the sequence of the time averaged measures
\begin{equation}\label{medie-temporali-KB}
\frac 1k\int_0^k  \mathcal L(u_n(t;0)) {\rm d}t, \qquad k\in \mathbb N.
\end{equation}
From  \eqref{p_est_mass} %Proposition  \ref{bound_lemma_mass}  in the Appendix 
we obtain 
\[
\sup_{t\ge 0} \mathbb E \left[\|u_n(t;0)\|_H^2\right]
 \le \|P_n \Phi \|^{2}_{L_{HS}(Y,H)}C \gamma^{-1}.
\]
Thus, by Chebychev's inequality we find that  the sequence of the time averaged measures
\eqref{medie-temporali-KB}
is tight in  $H_n$. Hence, there exists a  converging subsequence  (as $k_j\to\infty$) and it is a classical result to show that 
the limit is an invariant measure for the finite-dimensional Galerkin dynamics. Denote it by $\pi_n$.
With this construction, by  the definition \eqref{medie-temporali-KB}  of the  time averaged measures and 
by Propositions \ref{bound_lemma_mass} and \ref{bound_lemma_energy}  we obtain the following
%uniform 
bounds for the moments with respect to $\pi_n$ of the powers of the mass,  the energy and the modified energy:
\begin{equation}\label{massa-p-n-staz}
\int_{H_n}{\mathcal M}(v)^p {\rm d} \pi_n(v)\le \|P_n\Phi \|^{2p}_{L_{HS}(Y,H)}C \gamma^{-p},
\end{equation}
\begin{equation}\label{energy-p-n-staz}
\int_{H_n}  {\mathcal E}_\alpha(v)^p{\rm d} \pi_n(v) 
\le 
C \phi_{\alpha}(d, \beta,\sigma,\gamma,P_n\Phi)^{p}\gamma^{-p},
\end{equation}
for some  constant $C$ independent of $n$. 

Now we consider the  solution  of the Galerkin equation with initial data of law $\pi_n$; 
it is a stationary solution and  thanks to the two latter estimates on $\pi_n$, 
by Propositions \ref{bound_lemma_mass}  and \ref{bound_lemma_energy} we obtain  the estimates \eqref{Gal_uniform-mass-p} and \eqref{Gal_uniform-energy}, respectively. 
The identity  \eqref{Gal_uniform-mass} comes from \eqref{ito-base}, due to stationarity.

The bound  \eqref{Gal_sup-est-V} comes from \eqref{sup_est_V_def},  
by estimating the initial data by means of \eqref{massa-p-n-staz} and \eqref{energy-p-n-staz} and bearing in mind  the definition of the function \eqref{phi_alpha}.
\end{proof}

 Thanks to the estimates of the previous Proposition we construct a  stationary martingale solution to the 
  stochastic NLS equation \eqref{EQSgen}
  and we provide  power moments estimates (see, e.g,  \cite{FG,ChowK}
  for similar results in the case of the Navier-Stokes equations).

% PROPOSITION
 \begin{proposition}\label{prop-station}
  Under the Assumptions  \ref{ass-sigma} and \ref{ass_G}, for any $\gamma>0$
  there exists a stationary martingale solution 
  $\bigl({\tilde\Omega}_\gamma, {\tilde\F}_\gamma, {\tilde{\mathbb P}}_\gamma, {\tilde W}_\gamma,\tilde{\Filtration}_\gamma, U_\gamma \bigr)$ 
   of equation \eqref{EQSgen}  such that  for every $t\ge 0$ and $p\ge 1$ 
\begin{align}
{\tilde \E}_\gamma {\mathcal M}(U_\gamma(t)) &=\frac1{2\gamma} \|\Phi\|_{L_{HS}(Y,H)}^2
 \label{st-massa=}
\\   
 {\tilde \E}_\gamma [ \mathcal{M}(U_\gamma(t))^{p}]
 &\le \|\Phi \|^{2p}_{L_{HS}(Y,H)}C_p \gamma^{-p} 
 \label{st-massa-p}
 \\
 {\tilde \E}_\gamma [{ \mathcal E}_\alpha(U_\gamma(t))^p]&\lesssim_{d,\beta,\sigma}\phi_{\alpha}(d,\beta, \sigma,\gamma, \Phi)^p\gamma^{-p} \qquad
 \text{ for } \alpha\in\{-1,1\}.
        \label{st-energy-p}
 \end{align}
Moreover, for every finite $T>0$, 
\begin{multline}\label{U_gamma_sup-est-V}
 \tilde{\mathbb{E}}_\gamma \left[\sup_{0\le t\le T}\|U_\gamma(t)\|^{2p}_{\EA} \right]
\lesssim_{d, \sigma, p,T} 
\\
(\mathbbm{1}_{\alpha=-1}+ \mathbbm{1}_{\alpha=1} e^{C p \gamma  T})
\left(1+ \phi_{\alpha}(d,\beta, \sigma,\gamma,  \Phi)^p \gamma^{-p}+
 \|\Phi\|_{L_{HS}(Y,\EA)}^{2p(1+\sigma)}
 + \mathbbm{1}_{\alpha=1} \|  \Phi\|_{L_{HS}(Y,H)}^{2p(1+\frac{2\beta\sigma}{2\beta-\sigma d})}\right).
\end{multline} 
\end{proposition}
 %PROOF
\begin{proof} The damping parameter $\gamma>0$ is fixed. 
We consider the stationary solutions $\{u_n^{st}\}_{n\in\mathbb N}$ of the Galerkin equation given in Proposition \ref{prop-u-staz-Galerkin} and show that
they are tight in order to consider the limit as $n\to \infty$.  
\\
\underline{Step 1} (tightness)\\
First, we notice that, thanks to 
\eqref{bound_Phi_Pn},  the right-hand sides of  \eqref{Gal_uniform-mass-p}-\eqref{Gal_sup-est-V} can be bounded by quantities independent of $n$.

We consider the following locally convex topological spaces:
\begin{itemize}
 \item
$C([0,+\infty);\EAdual)$  with metric
$d_1(u,v)=\displaystyle\sum_{k=1}^\infty
     \frac
     1{2^k}\frac{\|u-v\|_{C([0,k];\EAdual)}}{1+\|u-v\|_{C([0,k];\EAdual)}}$;
\item 
$L^{2+2\sigma}_{\mathrm{loc}}(0,+\infty; {\mathbb L}^{2+2\sigma} )$ 
with metric 
$d_2(u,v)=\displaystyle\sum_{k=1}^\infty
     \frac 1{2^k}\frac{\|u-v\|_{L^2(0,k;L^{2+2\sigma}(D))}}{1+\|u-v\|_{L^2(0,k;{\mathbb L}^{2+2\sigma})}}$;
\item
$C_w([0,+\infty);\EA)$
 with the topology generated by the family of semi-norms\\
 $\|u\|_{k,v}=\displaystyle\sup_{0\le t\le k}|\langle u(t), v\rangle |$, $k \in
 \mathbb N, v \in \EAdual$.
\end{itemize}
We define the  space
\begin{equation}
\label{Z_space}
\mathcal Z= C([0,+\infty);\EAdual) \cap L^{2+2\sigma}_{\mathrm{loc}}(0,+\infty;{\mathbb L}^{2+2\sigma}) \cap C_w([0,+\infty);\EA).
\end{equation}
It is a locally convex topological space with the topology $\mathcal T$
given by  the supremum of the corresponding topologies.

The tightness in $\mathcal Z$ of the laws of the processes  $u^{st}_n$ defined on the time interval 
$[0,+\infty)$ is equivalent  to the  tightness in 
\begin{equation}
\label{Z_space_k}
\mathcal Z_k=C([0,k];\EAdual) \cap L^{2+2\sigma}(0,k; {\mathbb L}^{2+2\sigma}) \cap C_w([0,k];\EA)
\end{equation}
(see the Appendix \ref{section-compactness} for more details on $\mathcal Z_k$)
for any $k \in\mathbb N$ of the laws of the processes  $u^{st}_n$ 
defined on the time interval $[0,k]$.
Therefore  for every $k\in\mathbb N$ we work on the time interval  $[0,k]$ and prove that 
the family of the laws of   $\{u^{st}_n\}_n$ is tight in $\mathcal Z_k$. 

Using
the compactness result \eqref{compact-embedding2} of  Appendix \ref{section-compactness}  it is enough to show that there exists  $a\in(0,\frac 12) $ such that the following holds: 
for any $\varepsilon>0$ there exists $R_\varepsilon$ such that
\begin{equation}\label{tightness-Gal-Zk}
\inf_n \Prob\left(\|u^{st}_n\|_{L^\infty(0,k;\EA)}+\|u^{st}_n\|_{C^{a}([0,k];\EAdual)}\le R_\varepsilon \right)\ge 1-\varepsilon.
\end{equation}
If
\begin{equation}\label{stima-tight-infty}
\sup_n \E \left(\sup_{0\le t\le k}\|u^{st}_n(t)\|_{\EA}\right)<\infty
\end{equation}
and
\begin{equation}\label{stima-tight-holder}
\sup_n \E\|u^{st}_n\|_{C^a ([0,k];\EAdual)}<\infty,
\end{equation}
 then by Chebyshev's inequality we obtain \eqref{tightness-Gal-Zk}.
 
The first bound \eqref{stima-tight-infty} comes from \eqref{Gal_sup-est-V} and \eqref{bound_Phi_Pn}.
To get the other bound \eqref{stima-tight-holder}  we write the equation in the integral form
 \begin{equation}\label{eq-integrale-Gal-st}
 u^{st}_n(t)=u^{st}_n(0) + \im \int_0^t  A  u^{st}_n(r) {\rm d}r - \im  \alpha\int_0^t P_n F(u^{st}_n(r)) {\rm d}r -\gamma \int_0^t u^{st}_n(r) {\rm d}r+ 
P_n \Phi  W(t).
 \end{equation}
Then we  estimate the H\"older norm for the  three integrals as follows.
Set $J_{n,1}(t)=\int_0^t A u^{st}_n(r) {\rm d}r$; we have
\[
\E \sup_{0\le s<t\le k}\frac{ \|J_{n,1}(t)-J_{n,1}(s)\|_\EAdual }{(t-s)^{a}}
\le
\left(\E \sup_{0\le s<t\le k}\frac{ \|J_{n,1}(t)-J_{n,1}(s)\|^2_\EAdual }{(t-s)^{2a}} \right)^{\frac 12}
\]
and
\[
\frac{ \|J_{n,1}(t)-J_{n,1}(s)\|^2_\EAdual }{(t-s)^{2a}}
\le 
\frac{ \left(\int_s^t  \|A u^{st}_n(r)\|_\EAdual  {\rm d}r \right)^2 }{(t-s)^{2a}}
\le
\frac{\left(\int_s^t   \|A u^{st}_n(r)\|^2_\EAdual \rm d r\right)(t-s)}{(t-s)^{2a}}.
\]
So
\begin{equation}\label{Jn1}
\E \sup_{0\le s<t\le k}\frac{ \|J_{n,1}(t)-J_{n,1}(s)\|_\EAdual }{(t-s)^{a}}
\le k^{\frac 12 -a}
\left(\E \int_0^k   \|A u^{st}_n(r)\|^2_\EAdual \rm d r \right)^{\frac 12}
=  k^{1 -a}  \left(  \int \| v\|^2_\EA{ \rm d}\pi_n(v)\right)^{\frac 12}
\end{equation}
where $\pi_n$ is the marginal distribution at any fixed time, i.e. the law of $u^{st}_n(t)$.

Set $J_{n,2}(t)=\int_0^t P_n F( u^{st}_n(r)) {\rm d}r$.
We have
\[
 \int_s^t \|P_n F(u^{st}_n(r))\|^2_{\EAdual} {\rm d}r
 \le
  \int_s^t \|F(u^{st}_n(r))\|^2_{\EAdual} {\rm d}r
\underset{\text{by }\eqref{F-grandi}}{\le} \int_s^t  \E \| u^{st}_n(r)\|^{2(1+2\sigma)}_{\EA} {\rm d}r.
\]
Proceeding as before, we get
\begin{equation}\label{Jn2}
\E \sup_{0\le s<t\le k}\frac{ \|J_{n,2}(t)-J_{n,2}(s)\|_\EAdual }{(t-s)^{a}}
\le k^{1 -a}  \left(  \int \| v\|^{2(1+2\sigma)}_\EA{ \rm d}\pi_n(v)\right)^{\frac 12}.
\end{equation}
Set $J_{n,3}(t)=\int_0^t u^{st}_n(r) {\rm d}r$; then the estimate
\begin{equation}\label{Jn3}
 \E \sup_{0\le s<t\le k}\frac{ \|J_{n,3}(t)-J_{n,3}(s)\|_\EAdual }{(t-s)^{a}}
\le k^{1 -a}  \left(  \int \| v\|^{2}_\EA{ \rm d}\pi_n(v)\right)^{\frac 12}
\end{equation}
is obtained exactly  as for the first integral.
\\
Set $J_{n,4}(t)=P_n\Phi W(t)$; then from \eqref{PhiW-Holder}  we know that
for any $a\in (0,\frac 12)$ we have
\begin{equation}\label{Jn4}
 \E \|J_{n,4}\|_{C^a([0,k];\EA)}\lesssim_{a,k} \|P_n \Phi\|_{L_{HS}(Y,\EA)}.
\end{equation}

Now we  collect the previous estimates. From  Proposition \ref{prop-u-staz-Galerkin} 
we know that  the right hand sides of \eqref{Jn1}-\eqref{Jn3} 
are estimated by an expression which depends  on $n$ throught 
$ \|P_n \Phi \|_{L_{HS}(Y,V)}$; but this quantity is uniformly bounded in $n$, thanks to  \eqref{bound_Phi_Pn}. 
The same applies to \eqref{Jn4}.
Hence
 we obtain  \eqref{stima-tight-holder}.
% .
\\
\underline{Step 2} (convergence)\\
 We can apply the Prokhorov  Theorem and the  Jakubowski
generalization of the Skorohod Theorem to non-metric spaces (see \cite{Jak86}, \cite{Jak98}) to deduce convergence from tightness. 
We infer the existence of a subsequence $\left(u_{n_j}^{st}\right)_{j\in\mathbb{N}}$, a probability space $\left(\tilde{\Omega},\tilde{\F},\tilde{\Prob}\right)$ and random variables  $v_j, v:\tilde{\Omega} \rightarrow \mathcal Z$ with the law of $v_j$ equal to the law of $u^{st}_{n_j}$ and 
\begin{equation}
\label{convergencev_n}
v_j\to v  \ \quad \tilde{\Prob}-\text{a.s.} \quad  \text{in $\mathcal Z\quad $ as } j\to \infty.
\end{equation}
Moreover, since $V$ is compactly embedded in $H$, $v_j\to v$ $\tilde{\Prob}$-a.s. in $C([0, \infty);H)$ as  $j\to \infty$ and therefore $v\in \mathcal Z\cap C([0,+\infty);H)$.

Since each $v_j$ has the same law as $u^{st}_{n_j}$, it is a martingale solution to the Galerkin problem \eqref{eq-Gal}. Thus, each process 
\begin{equation*}
M_j(t):= v_j(t)-v_j(0)+\int_0^t \left[ -\im A  v_j(s)+\im  \alpha P_{n_j} F(v_j(s)) +\gamma v_j (s) \right] \,{\rm d}s
\end{equation*}
is a $H$-valued continuous martingale w.r.t. the filtration $\tilde{\mathbb{F}}_j:=\{\tilde{\F}_{j,t}\}_{t \in [0, \infty)}$, where $\tilde{\F}_{j,t}:=\sigma\left(v_j(s): s \le t\right)$. Its quadratic variation 
 is $\langle \langle M_j\rangle \rangle(t)=t \|P_{n_j}\Phi\|^2_{L_{HS}(Y,H)}$. 
Proceeding as in \cite[Lemma 6.2]{BHW-2019} one can prove that 
\begin{equation*}
\lim_{j\to\infty}\langle M_j(t)-M(t), \psi\rangle= 0, \qquad \tilde{\mathbb{P}}-a.s.,
\end{equation*}
for any $\psi \in \EA$ and $t \in [0,T]$, where 
\begin{equation}
M(t):= v(t)-v(0)+\int_0^t \left[ -\im A  v(s)+\im  \alpha F(v(s)) +\gamma v(s) \right] \,{\rm d}s. 
\end{equation}
It is easy to prove 
that $M$ is an $H$-valued continuous martingale w.r.t. the filtration $\tilde{\mathbb{F}}:=\{\tilde{\F}_{t}\}_{t \in [0, \infty)}$, where $\tilde{\F}_{t}:=\sigma\left(v(s): s \le t\right)$. 
Its quadratic variation is $\langle \langle M\rangle \rangle(t)=t \|\Phi\|^2_{L_{HS}(Y,H)}$.
Therefore, the Martingale Representation Theorem, see  \cite{dpz}, yields the 
existence of a cylindrical Wiener processes $\tilde{W}$ on $Y$ with respect to a standard extension of the original filtration, still denoted by $\tilde{\mathbb{F}}$ for simplicity, 
%KALLENBERG 2021, P.359: For filtrations $\mathcal{F}$, $\mathcal{G}$ on a common probability space, we say that $\mathcal{G}$ is a standard extension of $\mathcal{F}$ if $\mathcal{F}_t \subset \mathcal{G}_t \underset{\mathcal{F}_t}{\bot} \mathcal{F}, t \ge 0$. These are the minimal conditions ensuring that $\mathcal{G}$ will preserve the conditioning and adaptedness properties of $\mathcal{F}$. //CONCRETAMENTE COME SI COSTRUISCE LA STANDARD EXTENSION: vedi proof of Thm 19.13 Kallenberg: sia $M$ una martingale w.r.t. the filtration $\F$. Considero un Wiener process $W$ w.r.t. the generated filtration $\mathcal{X}$ e indipendente da $\F$. Definisco $\mathcal{G}:\{\mathcal{G}_t\}_t$ come $\mathcal{G}_t: \F_t \lor \mathcal{X}_t$. Questa \'e una standard extension of $\F$ and $\mathcal{X}$.
such that
\begin{align}
M(t)=\int_0^t \Phi \df \tilde{W}(s)
\end{align}
for $t\in [0,T].$
Therefore the system $\left(\tilde{\Omega},\tilde{\mathcal{F}},\tilde{\Prob},\tilde{W},\tilde{\Filtration},v\right)$ is a martingale solution to \eqref{EQSgen} whose paths belong
$\tilde{\mathbb{P}}$-a.s. to the space $\mathcal Z$, as already shown at the beginning of the proof.
%Since $\EA$ is compactly embedded in $H$,  from $v \in C_w([0,+\infty);\EA)$ is follows that $v\in C([0,+\infty);H)$ too.

The limit is  a stationary process in $\EA$. Indeed the weak and the strong Borel subsets of $\EA$ coincide. 
Therefore stationarity in $\EA$ is a consequence of the stationarity of the Galerkin sequence $v_j$  and  its 
convergence  in $C_w([0,+\infty);\EA)$.
\\
\underline{Step 3} (relationships)
\\
The estimates \eqref{st-massa-p}-\eqref{st-energy-p} are inherited from the same estimate for the Galerkin sequence, given in 
Proposition \ref{prop-u-staz-Galerkin}, because of the uniform bound \eqref{bound_Phi_Pn} 
on the Hilbert-Schmidt norm of $P_{n_j} \Phi$.

As far as the mass is concerned,  we consider the   equality \eqref{Gal_uniform-mass}, now written as 
\begin{equation}\label{=energia-vj}
\tilde{\E} [{\mathcal M}(v_j(t))] =\frac1{2\gamma} \|P_{n_j} \Phi\|_{L_{HS}(Y,H)}^2.
\end{equation}
Since %$\tilde{\mathbb{P}}$-a.s. $v_j \rightarrow v$ in $C_w([0,+\infty);\EA)$ and $\EA$ is compactly embedded in $H$, then   
$\tilde{\mathbb{P}}$-a.s. $v_j \rightarrow v$ in $C([0,+\infty);H)$, in particular for fixed $t$ we have 
 \[
 {\mathcal M}(v_j(t)) \to {\mathcal M}(v(t))
 \]
pathwise, as $j\to\infty$.
From \eqref{Gal_uniform-mass-p} and \eqref{bound_Phi_Pn} we know that
\[
\sup_j \tilde \E  [ \mathcal{M}(v_j(t))^{2}]
<\infty,
\]
so the sequence $ \mathcal{M}(v_j(t))$ is uniformly integrable (with respect to the probability measure $\tilde\Prob$).
Since we already know the pathwise convergence , we get the convergence in the mean
\[
\tilde{\E} [{\mathcal M}(v_j(t))]\to \tilde{\E} [{\mathcal M}(v(t))].
\]
The convergence of the right-hand side of
\eqref{=energia-vj}  towards $\frac1{2\gamma} \|\Phi\|_{L_{HS}(Y,H)}^2$  is trivial. This proves identity \eqref{st-massa=}.

Estimate \eqref{U_gamma_sup-est-V} is inherited from the same estimate for the Galerkin sequence \eqref{Gal_sup-est-V}.

Finally, to  point out that the  stationary solution depends on $\gamma$ we will denote it 
by $(\tilde{\Omega}_\gamma,\tilde{\F}_\gamma,\tilde{\mathbb P}_\gamma,\tilde{W}_\gamma,\tilde{\Filtration}_\gamma,U_\gamma)$ from now on.
 \end{proof}

%REMARK
\begin{remark}\label{oss-generale-esistenza}
 The proof of Proposition \ref{prop-station} relies on two properties.
 On the one hand, 
we used the a priori estimates of the Galerkin approximating sequence, coming from the 
 Hamiltonian structure
 of the NLS equations (see the proof of Proposition \ref{prop-Galerkin-esistenza_bis}); 
 on the other hand, we used the compact embeddings that are the basis of the tightness argument. 
 Unlike the majority of the papers in the literature, we do not use the Strichartz estimates 
 to infer the existence of solutions; this technique  first appeared in \cite{BHW-2019} in the stochastic setting.
Hence  we do not require strong restrictions on the power $\sigma$ and the spatial dimension $d$.

 With this procedure the results are obtained in any space dimension 
 and we impose on the domain only the assumptions needed to have compact embeddings: 
 e.g. we can consider  bounded domains in $\mathbb{R}^d$, $d \ge 1$,  
 as well as  $d$-dimensional compact Riemannian manifolds (see Appendix \ref{domini}). 
 The drawbacks of this approach are that we construct a martingale solution, 
 that is a weak solution in the probabilistic sense, and that we obtain solutions 
 with a.e. path in $C_w([0,+\infty);\EA)$ and not in $C([0,+\infty);\EA)$. 
 
 For $d\le 2$ the pathwise uniqueness can be proved  (see some results for $d=2$ in \cite{noi2d});  
 the strong existence and the existence of the Markov transition semigroup follow. The details are given in Section \ref{more_reg_sec}.
 In that case the stationary martingale solution yields the existence of an invariant measure.
 
 Anyway the existence of stationary solutions is sufficient for our pourposes.
\end{remark}

%%%%%%%%%%%%%%%%%%%%%%%%
\section{Inviscid limit}
\label{il}

From now on we set $\Phi_\gamma=\sqrt \gamma \Phi$, that is 
we deal with the stochastic NLS equation
\begin{equation}
\label{eqS-sqrt-gamma}
\,{\rm d}u(t) +[ -\im A u(t)  + \im \alpha F( u(t)) +\gamma   u(t)] \,{\rm d}t=\sqrt\gamma \Phi \, {\rm d}W(t).
\end{equation}
With this choice of the noise intensity we have mean  estimates independent of $\gamma$.
% PROPOSITION
\begin{proposition}\label{Prop-no-gamma}
 Under the Assumptions  \ref{ass-sigma} and \ref{ass_G}, for any $\gamma>0$ 
  there exists a stationary martingale solution 
$\bigl({\tilde\Omega}_\gamma, {\tilde\F}_\gamma, {\tilde{\mathbb P}}_\gamma, {\tilde W}_\gamma,\tilde{\Filtration}_\gamma, U_\gamma \bigr)$ 
of equation \eqref{eqS-sqrt-gamma}  such that  for every $t\ge 0$  and $p\ge 1$ 
\begin{align}\label{uniform-mass}
{\tilde \E}_\gamma [{\mathcal M}(U_\gamma(t))] &=\frac1{2} \|\Phi\|_{L_{HS}(Y,H)}^2
\\   \label{uniform-mass-p}
 {\tilde \E}_\gamma [ \mathcal{M}(U_\gamma(t))^{p}]
 &\lesssim_p \|\Phi \|^{2p}_{L_{HS}(Y,H)}
\\ \label{uniform-energy}
 {\tilde \E}_\gamma [{ \mathcal E}_\alpha(U_\gamma(t))^p]&\lesssim_{d,\beta,\sigma}\phi_{\alpha}(d,\beta, \sigma,1, \Phi)^p, \qquad
 \text{ for } \alpha\in\{-1,1\}.
 \end{align}

Moreover, for every finite $T>0$ and any $\gamma\in (0,1]$,
\begin{equation}\label{sti_tight-gamma}
 \tilde{\mathbb{E}}_\gamma \left[\sup_{0\le t\le T}\|U_\gamma(t)\|^{2p}_{\EA} \right]
\lesssim_{d, \sigma, p,T} 
1+ \phi_{1}(d,\beta, \sigma,1, \Phi)^p.
\end{equation} 

\end{proposition}
%PROOF
\begin{proof}
This is nothing but  Proposition \ref{prop-station} now rewritten with the covariance of the noise 
as in equation \eqref{eqS-sqrt-gamma}. Indeed, 
since the noise has intensity proportional to $\sqrt \gamma$, we have that
the right hand sides of \eqref{st-massa=}-\eqref{st-energy-p} are   independent of $\gamma$.
This is easily checked for 
the equality \eqref{st-massa=} and of the bound \eqref{st-massa-p}. Moreover, by 
the definition \eqref{phi_alpha} of $\phi_\alpha$ we have
\[
\phi_{\alpha}(d,\beta, \sigma,\gamma, \sqrt \gamma\Phi)=\gamma \phi_{\alpha}(d, \beta,\sigma,1,\Phi).
\]
In a  similar  way from \eqref{U_gamma_sup-est-V} we get  \eqref{sti_tight-gamma}, when $T$ and $\gamma$ are bounded.
\end{proof}

These estimates  provide a tighness result for the family $\{U_\gamma\}_{0<\gamma\le 1}$, which is the starting point to analyze the limit as $\gamma\to 0$.
Recall \eqref{Z_space} and \eqref{Z_space_k}; we have the following tightness result for the family $\{U_\gamma\}_{0<\gamma\le 1}$ of stationary solutions.
% PROPOSITION
\begin{proposition}\label{tight-U_gamma}
Under the Assumptions  \ref{ass-sigma} and \ref{ass_G}, there exists a  family $\{U_\gamma\}_{0<\gamma\le 1}$ of
 stationary martingale solutions of equation \eqref{eqS-sqrt-gamma}, whose laws are  tight in $\mathcal Z$.
\end{proposition}
% PROOF
\begin{proof}
From Proposition \ref{Prop-no-gamma} we know that  for any $\gamma>0$ the  equation \eqref{eqS-sqrt-gamma}
has a stationary martingale solution 
$\bigl({\tilde\Omega}_\gamma, {\tilde\F}_\gamma, {\tilde{\mathbb P}}_\gamma, {\tilde W}_\gamma,\tilde{\Filtration}_\gamma, U_\gamma \bigr)$.
We prove the tightness in   $\mathcal Z$ along the same lines as in the first step of the proof of Proposition \ref{prop-station}.

First, it is enough to prove that for    every $k \in\mathbb N$
the laws of the processes  $\{U_\gamma\}_{0< \gamma \le 1}$, when 
defined on the time interval $[0,k]$, are tight  in $\mathcal Z_k$. And this holds when 
 for some $a\in (0,\frac 12 )$ we have
\[
\sup_{0< \gamma \le 1} \tilde{\E}_\gamma 
\left(\sup_{0\le t\le k}\|U_\gamma(t)\|_{\EA}+\|U_\gamma\|_{C^a ([0,k];\EAdual)}\right)<\infty.
\]
From  estimate \eqref{sti_tight-gamma}, for $p \ge 1$ we have
\begin{equation}
\label{stima-Lebesgue-tight}
\sup_{0< \gamma \le 1}{\tilde \E}_\gamma \left[ \sup_{0 \le t \le k} \|U_\gamma(t)\|_{\EA}^{p} \right]\le M_{p},
\end{equation}
where $M_p$  is a positive constant depending on   $d,\beta, \sigma,\Phi$.

Writing the equation  \eqref{eqS-sqrt-gamma} in integral form 
\begin{equation}\label{U-somma}
\begin{split}
U_\gamma(t)&=U_\gamma(0)
+ \im \int_0^t A U_\gamma(r) {\rm d}r - \im \alpha \int_0^tF( U_\gamma(r))  {\rm d}r  -\gamma  \int_0^t U_\gamma(r) {\rm d}r+\sqrt\gamma \Phi W(t)
\\&:= U_\gamma(0)+ \im J_1(t)- \im  \alpha J_2(t)-\gamma  J_3(t)+\sqrt\gamma \Phi W(t),
\end{split}
\end{equation}
 in a similar way to \eqref{Jn1}, \eqref{Jn2},  \eqref{Jn3}
 we get, respectively
\begin{equation}
{\tilde \E}_\gamma \sup_{0\le s<t\le k}\frac{ \|J_{1}(t)-J_{1}(s)\|_\EAdual }{(t-s)^{a}}
\le k^{1 -a}  \left( \tilde\E_\gamma \| U_\gamma(t)\|^2_\EA \right)^{\frac 12}
\end{equation}
\begin{equation}
{\tilde \E}_\gamma \sup_{0\le s<t\le k}\frac{ \|J_{2}(t)-J_{2}(s)\|_\EAdual }{(t-s)^{a}}
\le k^{1 -a}  \left(  \tilde\E_\gamma \| U_\gamma(t)\|^{2(1+2\sigma)}_\EA  \right)^{\frac 12}
\end{equation}
\begin{equation}
 {\tilde \E}_\gamma \sup_{0\le s<t\le k}\frac{ \|J_{3}(t)-J_{3}(s)\|_\EAdual }{(t-s)^{a}}
\le k^{1 -a}  \left( \tilde\E_\gamma \| U_\gamma(t)\|^2_\EA \right)^{\frac 12}.
\end{equation}
Merging these estimates with \eqref{stima_PhiW-Ca}, from equation \eqref{U-somma}  and the estimate 
\eqref{stima-Lebesgue-tight} we get that there exists $a\in (0,\frac 12)$ such that 
\[
\sup_{0< \gamma \le 1}{\tilde \E}_\gamma \|U_\gamma\|_{C^a ([0,k];\EAdual)}<\infty.
\]
This concludes the proof of the tightness in  $\mathcal Z_k$.
\end{proof}

 Now we consider the vanishing limit.
%TEO
\begin{theorem}  \label{teoUstaz}
Let  the Assumptions  \ref{ass-sigma} and \ref{ass_G} hold  and consider
the family $\{U_\gamma\}_{0<\gamma\le 1}$ of stationary solutions 
to equation \eqref{eqS-sqrt-gamma}, given in Proposition \ref{Prop-no-gamma}.
Then 
any vanishing sequence has a subsequence $\gamma_j\to 0$ as $j\to\infty$ such that
there exist 
a new filtered probability space 
$\bigl(\hat{\Omega},\hat{\F},\hat{\Prob},\hat{\Filtration}\bigr)$, 
a sequence of $\mathcal Z$-valued stochastic processes ${\hat U}_{j}$ and a 
$\mathcal Z$-valued stochastic process ${\hat U}$ enjoying the following properties:
\begin{enumerate}
\item[$P_1)$]\label{i1}
 ${\hat U}_j$ and ${U}_{\gamma_j}$ have the same law;
\item[$P_2)$]\label{i2}
$\hat{\Prob}$-a.s. the sequence ${\hat U}_{j}$ converges to $\hat U$ in 
$\mathcal Z$ as $j\to \infty$;
\item[$P_3)$]\label{i3}
$\hat U$ fulfills the equation \eqref{eqS};
\item[$P_4)$]\label{i4}
$\hat U\in C([0,\infty);H)\cap C_w([0, \infty);V)$, 
$\hat{\Prob}$-a.s., 
and
\begin{align}
\label{nonbanale}
 &\hat{\mathbb E}{\mathcal M}(\hat U(t)) =\frac1{2} \|\Phi\|_{L_{HS}(Y,H)}^2, \qquad \forall t>0
  \intertext{ and for any $ t>0$  and $p\ge 1$}
  \label{finaleMp}
 &\hat{\mathbb E}[{\mathcal M}(\hat U(t))^p]  \lesssim_p \|\Phi\|^{2p}_{L_{HS}(Y,H)}
 \\ \label{finale-energy}
&{\hat \E} \left[ \mathcal{E}_\alpha(\hat U(t))^p \right]  \lesssim_{d,\beta, \sigma,p}\phi_{\alpha}(d,\beta,\sigma, 1, \Phi)^p
 \qquad \text{ for }\alpha\in\{-1,1\}.
\end{align}
\item[$P_5)$]\label{i5}
$\hat U$ is a stationary process in $V$.
\item[$P_6)$]\label{i6}
The mass  $\mathcal M(\hat U(t))$ is a random variable, constant in time.
\end{enumerate}
\end{theorem}
% DIMO
\begin{proof}
%1
\underline{Step 1 (convergence)}\\
From Proposition \ref{tight-U_gamma} we know that the family   $\{U_\gamma\}_{0<\gamma\le 1}$ of
 stationary martingale solutions of equation \eqref{eqS-sqrt-gamma} is  tight in $\mathcal Z$.
Then, given any vanishing sequence, by Prokhorov's theorem 
 there exists a (sub)sequence weakly convergent; we denote it by  $\{\gamma_{j}\}_{j=1}^\infty$. 
 And   by the  Jakubowski
generalization of the Skorohod theorem to non-metric spaces  (see \cite{Jak86}, \cite{Jak98}) there exist
a probability basis $(\hat \Omega, \{\hat {\mathcal F}_t\}_t, \hat \Prob)$, random variables
 ${\hat U}_j,\hat U:\hat \Omega\to \mathcal Z$ such that
${\hat U}_j$ and $U_{\gamma_j}$ have the same law and
\[
\lim_{j\to \infty}{\hat U}_j=\hat U \qquad \hat {\Prob}-a.s. \text{ in } \mathcal Z.
\]
Moreover the limit process $\hat U$ is stationary  in $\EA$.
These prove $P_1)$, $P_2)$ and  $P_5)$.
%2
\\
\underline{Step 2 (the limit equation)}
\\
Now we look for the equation fulfilled by the limit process $\hat U$.
Fix $t>0$ arbitrarily. We have
\[
{\hat U}_j(t)-{\hat U}_j(0) -\im \int_0^t  A {\hat U}_j(s) \,{\rm d}s + \im \alpha \int_0^t F( {\hat U}_j(s)) \,{\rm d}s
=-\gamma_j \int_0^t   {\hat U}_j(s) \,{\rm d}s+\sqrt{\gamma_j} \Phi \hat W_j(t)
\]
$\hat\Prob$-a.s. Also the following results hold pathwise and we do not specify it any more.
The r.h.s. vanishes as $j\to \infty$, in the ${\mathbb L}^{2+2\sigma}$-norm.
For the l.h.s. we have
\begin{align*}
{\hat U}_j(t)&\to {\hat U}(t)  \quad \text{ in }\EAdual
\\
\int_0^t  \langle A {\hat U}_j(s), v \rangle \,{\rm d}s
& \to \int_0^t  \langle A {\hat U}(s), v \rangle \,{\rm d}s     \qquad    \text{ for any  } v \in \EA
\\
 \int_0^t F( {\hat U}_j(s)) \,{\rm d}s & \to \int_0^t F( {\hat U}(s)) \,{\rm d}s  \quad  \text{ in } {\mathbb L}^{\frac{2+2\sigma}{1+2\sigma}}.
\end{align*}
Let us check the two latter convergences, since the first one is trivial. We have
\[
\left|\int_0^t  \langle A {\hat U}_j(s) - A {\hat U}(s), v \rangle \,{\rm d}s\right|
=
\left| \int_0^t  \langle  {\hat U}_j(s) -  {\hat U}(s),A v \rangle \,{\rm d}s\right|
\le
t \sup_{0\le s\le t} |\langle  {\hat U}_j(s) -  {\hat U}(s),A v \rangle|
\]
We conclude, since $A v\in \EAdual$ and there is convergence in $C_w([0,t];\EA)$.

As far as the nonlinear term is concerned, from  \eqref{stimaF-lipschitz}
\[
\|F({\hat U}_j(s))-F({\hat U}(s))\|_{{\mathbb L}^{\frac{2+2\sigma}{1+2\sigma}}}
\lesssim
  \left( \|{\hat U}_j(s)\|_{{\mathbb L}^{2+2\sigma}}^{2\sigma}+\|{\hat U}(s)\|_{{\mathbb L}^{2+2\sigma}}^{2\sigma} \right) 
  \|{\hat U}_j(s)-{\hat U}(s)\|_{{\mathbb L}^{2+2\sigma}}.
\]
Now, we apply H\"older inequality in time with 
$\frac{1}{2+2\sigma}+\frac{2\sigma}{2+2\sigma}+\frac{1}{2+2\sigma}=1$:
\[\begin{split}
&\|F({\hat U}_j)-F({\hat U})\|_{L^1(0,t;{\mathbb L}^{\frac{2+2\sigma}{1+2\sigma}})}
\\&
\lesssim_\sigma t^{\frac1{2+2\sigma}}
  \left( \|{\hat U}_j\|_{L^{2+2\sigma}(0,t;{\mathbb L}^{2+2\sigma})}
       +\|{\hat U}\|_{L^{2+2\sigma}(0,t;{\mathbb L}^{2+2\sigma})} \right)^{2\sigma} 
  \|{\hat U}_j-{\hat U}\|_{L^{2+2\sigma}(0,t;{\mathbb L}^{2+2\sigma})}.
\end{split}
\]
Thus, for any $t$ we have that  $\hat U$ fulfills the deterministic unforced Schr\"odinger 
equation
\[
\hat U(t)-\hat U(0)-\im \int_0^t A \hat U(s) \,{\rm d}s+\im \alpha \int_0^t F( \hat U(s)) \,{\rm d}s=0
\]
$\hat{\Prob}$-a.s., 
understood as an identity in $\EAdual$.
This proves $P_3)$.
%3
\\
\underline{Step 3 (regularity results)}\\
Since $\EA$ is compactly embedded in $H$,  from $\hat U\in  C_w([0, \infty);V)$ we also get 
$\hat U\in C([0,\infty);H)$.
Moreover, the estimates \eqref{finaleMp}-\eqref{finale-energy} hold true since the bounds \eqref{uniform-mass-p}-\eqref{uniform-energy} still hold for  the limit. 
Indeed, any bounded sequence has a 
subsequence converging in a very weak sense;  
we denote in general  the norm  by $\hat\E\|\cdot\|_X$,  and we know that 
\[
\hat\E\|\hat U\|_X \le \liminf_{j\to\infty} \hat\E\|\hat U_j\|_X.
\] 
%\\
\underline{Step 4 (the identity \eqref{nonbanale})} \\
Fix $t>0$ arbitrarily.
Since $\hat U_j(t)$ converges to $\hat U(t)$ weakly in $\EA$, then it converges strongly in $H$, 
because of the compact embedding $\EA\subset H$. Therefore
\[ \lim_{j\to\infty} {\mathcal M}(\hat U_j(t))
=  {\mathcal M}(\hat U(t)) .
\]
$\hat\Prob$-a.s.
Moreover, by  \eqref{uniform-mass-p} we have
\[
\sup_j \hat \E {\mathcal M}(\hat U_j(t))^2\lesssim
\|\Phi\|^4_{L_{HS}(Y,H)}.
\]
This gives uniform integrability;
 hence from the convergence $\hat\Prob$-a.s. we get  the  convergence in the mean, namely
\[
\hat  \E {\mathcal M}(\hat U_j(t))
\to
\hat  \E {\mathcal M}(\hat U(t)).
\]
Bearing in mind \eqref{uniform-mass},
we get \eqref{nonbanale}.
This show that the limit is not the zero process. Moreover the limit process $\hat U$
depends on the covariance of the noise appearing  in the perturbed  equation \eqref{eqS-sqrt-gamma}.
\\ % 5 
\underline{Step 5 (the mass)}
\\
 The mass is a random variable, constant in time. Indeed
 $\hat \Prob$-a.s. we have
 \[\begin{split}
\mathcal M(\hat U(t))
 &=\mathcal M(\hat U(0))+2\int_0^t  Re\langle \hat U(s),\frac{d\hat U}{ds} (s)\rangle \,{\rm d}s
\\& =\mathcal M(\hat U(0))+2 \int_0^t Re\langle \hat U(s), \im A \hat U(s)\rangle \,{\rm d}s
 -2 \alpha \int_0^t  Re\langle \hat U(s), \im  F( \hat U(s))\rangle \,{\rm d}s.
\end{split} \]
 Both integrals in the r.h.s. are well defined. The first because $\hat U\in L^2(0,t;\EA)$ 
 so $A\hat U\in  L^2(0,t;\EAdual)$ and therefore $\langle \hat U, A \hat U\rangle\in L^1(0,t)$;
the second because 
 $\hat U\in L^{2+2\sigma}(0,t;{\mathbb L}^{2+2\sigma})$ and, by \eqref{F-grandi},  
 $F(\hat U)\in L^{\frac{2+2\sigma}{1+2\sigma}}(0,t; {\mathbb L}^{\frac{2+2\sigma}{1+2\sigma}}) $ so
 $\langle \hat U,   F( \hat U)\rangle\in  L^1(0,t)$. 
 Moreover, they both vanish, since they  involve real parts of imaginary quantities. 
 Hence  the mass is a random variable, constant in time.
 \end{proof}

 Any stationary solution  $\hat U$ constructed in the previous theorem is not a trivial process, i.e. it is not a deterministic 
 function of time, due to the presence of the noise in the approximating equation. 
 In fact we can prove  two results, as in \cite{Sh11}: if there is noise, then the law of the random variable $\hat U(t)$  has no atom in $0$; and if the noise is full, then
  the mass $\frac 12 \|\hat U(t)\|_H^2$ is a random variable whose law 
   admits a density w.r.t. the Lebesgue  measure.

 Now we denote by $l$ the Lebesgue measure on the real line and by $\mu$ the  marginal law of the stationary process  $\hat U$ of  the previous theorem, i.e. 
$\mu$ is the law of the random variable $\hat U(t)$.

 % TEO
\begin{theorem}
\label{mass_density}
Let Assumptions \ref{ass-sigma} and  \ref{ass_G} be in force.
\\
If $\varphi_j \neq 0$ for some  $j\in {\mathbb Z}_0$, then $\mu$ has no atoms in $0$.
\\
If  $\varphi_j \ne 0$ for all $j \in \mathbb{Z}_0$,  then there exists a continuous increasing function $p(r)$ going to zero with $r$ such that 
\begin{equation*}
\mu\left( \mathcal{M}(u) \in \Gamma\right) \le p(l(\Gamma))
\end{equation*}
for any Borel subset $\Gamma$ of $[0,+\infty)$.
\end{theorem}

 The proof follows the lines of  \cite{Sh11}; we postpone the details in Appendix \ref{section-dense}.

\section{More regular solutions}
\label{more_reg_sec}
In this Section we study the regularity of solutions to equation \eqref{eqS} in the cases specified by  the following 
\begin{assumption}
\label{ass_dim}
Let $d \ge 1$. We assume $\beta>\frac d2$ if $d \ne 2$ and $\beta \ge \frac d2$ if $d=2$.
\end{assumption}
Under the Assumption \ref{ass_dim} we can prove the uniqueness of solutions  for the  equation \eqref{EQSgen} 
as well as  \eqref{eqS}
 and we get the conservation of energy for the deterministic equation  \eqref{eqS}.

We start by proving the pathwise uniqueness for solutions to \eqref{eqS}.  When $\gamma=0$ and $\Phi=0$ this gives the  uniqueness for the deterministic equation  \eqref{eqS} as well. 
%PROP
\begin{proposition}
\label{uniq_prop}
Let Assumptions \ref{ass-sigma}, \ref{ass_G}  and  \ref{ass_dim} be in force. 
For any $\gamma\ge 0$ let 
$\bigl({\Omega}_\gamma, {\F}_\gamma, {\mathbb P}_\gamma, {W}_\gamma,\Filtration_\gamma, u^i_\gamma \bigr)$, $i=1, 2$ be two  martingale solutions of equation  \eqref{eqS-sqrt-gamma} with the same  initial data $u^0\in \EA$.
Then 
\begin{equation}\label{pathwise-uniq}
\mathbb{P}_\gamma\left(u^1_\gamma(t)=u^2_\gamma(t) \ \text{for all} \ t \in[0, +\infty) \right)=1.
\end{equation}
\end{proposition}
%PROOF
\begin{proof}
\underline{$d\ge1$ and $\beta>\frac d2$.} We know the existence of martingale solutions, which for any $T>0$ enjoy 
\begin{equation}
\label{reg_uniq}
u_\gamma \in L^{2p}(\Omega_\gamma;L^\infty(0,T;V)).
\end{equation} 
Now we suppose they are defined on the same probability space and have the same initial data.
Define $v_\gamma:=u^1_\gamma-u^2_\gamma$. This difference satisfies
\begin{equation*}
\begin{cases}
\frac{{\rm d}v_\gamma(t)}{{\rm d}t}= \im Av_\gamma(t)- \im\alpha \left(F(u^1_\gamma(t))-F(u^2_\gamma(t))\right)-\gamma v_\gamma(t)
\\
v_\gamma(0)=0.
\end{cases}
\end{equation*}
We have 
\begin{align}
\frac{{\rm d}}{ {\rm d}t} \|v_\gamma(t)\|_H^2
&=2\text{Re}\langle v_\gamma(t), -\im \alpha F(u^1_\gamma(t))+ \im \alpha F(u^2_\gamma(t))\rangle -2 \gamma \|v_\gamma(t)\|^2_H.
\end{align}
From \eqref{stima_F_L_infty} we get
\begin{align}
\label{u4}
\text{Re}\langle v_\gamma(t), -\im \alpha F(u^1_\gamma(t))+ \im \alpha F(u^2_\gamma(t))\rangle
&\lesssim \|v_\gamma(t)\|^2_H \left[\|u^1_\gamma(t)\|_{L^{\infty}}^{2\sigma} +\|u^2_\gamma(t)\|^{2\sigma}_{L^{\infty}}\right],
\end{align}
so that 
\begin{align*}
\frac{{\rm d}}{{\rm d}t} \|v_\gamma(t)\|_H^2
&\lesssim
-2 \|v_\gamma(t)\|^2_H\left( \gamma- \left[\|u^1_\gamma(t)\|_{L^{\infty}}^{2\sigma} +\|u^2_\gamma(t)\|^{2\sigma}_{L^{\infty}}\right] \right)
\\
&\lesssim
-2 \|v_\gamma(t)\|^2_H\left( \gamma- C\left[\|u^1_\gamma(t)\|_{\EA}^{2\sigma} +\|u^2_\gamma(t)\|^{2\sigma}_{\EA}\right] \right),
\end{align*}
where we used the  Sobolev embedding $V \subset L^\infty$, which holds for $\beta>\frac d2$.

By the Gronwall lemma we get, for any $t \in \left[0,T\right]$,
\begin{equation*}
e^{-2\gamma t + C\int_0^t\left[\|u^1_\gamma(s)\|_{\EA}^{2\sigma} +\|u^2_\gamma(s)\|^{2\sigma}_{\EA}\right]\, {\rm d}s}\|v_\gamma(t)\|^2_H\le 0, \qquad \mathbb{P}_\gamma-a.s,
\end{equation*}
where we notice that the term $\int_0^t\left[\|u_1(s)\|_{\EA}^{2\sigma}+\|u_2(s)\|^{2\sigma}_{\EA}\right] {\rm d}s$ is well defined thanks to \eqref{reg_uniq}.
Therefore there exists a sequence of times $(t_n)_n$, dense in $[0,T]$, such that
\[
\mathbb{P}_\gamma\left(u^1_\gamma(t_n)=u^2_\gamma(t_n) \ \text{for all} \ n \in \mathbb N \right)=1.
\]
By  the continuity of the trajectories in $H$,  we conclude that
\begin{equation*}
\mathbb{P}_\gamma\left(u^1_\gamma(t)=u^2_\gamma(t) \ \text{for all} \ t \in[0, T] \right)=1.
\end{equation*}
Taking the sequence $T=N\in\mathbb N$ we get \eqref{pathwise-uniq}.
\\
\underline{$d=2$ and $\beta=1$.} For the defocusing case we refer to \cite[Theorem  6.5]{noi2d} in the simplified setting of an additive noise. The same proof works in the focusing case, since the estimates in \cite[Theorem  6.5]{noi2d} and in all the results in \cite[Section 6]{noi2d} are independent of $\alpha$.
\end{proof}

The pathwise uniqueness and the existence of martingale solutions
imply the existence of  strong solutions, see e.g. \cite[Theorem 2]{Ondrejat_2004_Uniqueness} and \cite[Theorem 5.3 and Corollary 5.4]{Kunze_2013_Yamada}.
The following result is thus a direct consequence of Propositions \ref{Prop-no-gamma} and \ref{uniq_prop}.

\begin{proposition}
\label{sta_uni_str}
Let Assumptions \ref{ass-sigma}, \ref{ass_G}  and  \ref{ass_dim} be in force. 
 Then for any $\gamma>0$   there exists a unique strong (in the probabilistic sense) stationary solution 
$U_\gamma$ of equation \eqref{eqS-sqrt-gamma}  fulfilling the estimates 
\eqref{uniform-mass}-\eqref{sti_tight-gamma}.
\end{proposition}
Therefore one can define the Markov semigroup associated to the stochastic equation \eqref{eqS-sqrt-gamma} and any stationary solution provides an invariant measure (see, e.g., \cite{dpz}). Thus in this setting our results are closer to those of Kuksin and Shirikyan \cite{KuksinS}.

Reasoning as in Proposition \ref{tight-U_gamma} one gets that the family $\{U_\gamma\}_{0<\gamma \le 1}$ of stationary strong solutions of equation \eqref{eqS-sqrt-gamma} is tight in $\mathcal{Z}$. Considering the vanishing limits we obtain the following result.
%TEO
\begin{theorem}
\label{hatHdimle2}
Let Assumptions \ref{ass-sigma}, \ref{ass_G}  and  \ref{ass_dim} be in force. 
Consider the family $\{U_\gamma\}_{0<\gamma \le 1}$ of strong stationary solutions
of equation \eqref{eqS-sqrt-gamma}. Then 
any vanishing sequence has a subsequence $\gamma_j\to 0$ as $j\to\infty$ such that
there exist 
a new filtered probability space 
$\bigl(\hat{\Omega},\hat{\F},\hat{\Prob},\hat{\Filtration}\bigr)$, 
a sequence of $\mathcal Z$-valued stochastic processes ${\hat U}_{j}$ and a 
$\mathcal Z$-valued stochastic process ${\hat U}$ enjoying properties $P_1$)-$P_3$) and $P_5$)-$P_6$) of Theorem \ref{teoUstaz} and
\begin{itemize}
\item[$P^\prime_4)$]\label{i4_uniq}
$\hat U\in C([0,\infty);V)$, $\hat{\Prob}$-a.s., 
such that \eqref{nonbanale} - \eqref{finale-energy} hold.
\item[$P_7)$]\label{i7_uniq}The energy is a constant in time  random variable.
\end{itemize}
\end{theorem}

\begin{proof}
Referring back to Theorem \ref{teoUstaz} we only need to prove $P_7)$ and the additional regularity $C([0, \infty);V)$ in $P^\prime_4)$. 

To deal with the energy in $P_7)$ we follow \cite[Section 3.3 (p. 591)]{Burq+G+T_2004}. We work pathwise.
We consider a Galerkin approximation $\{\hat U_n\}_{n \in \mathbb{N}}$ of equation \eqref{eqS}
 with initial data $\hat U^0_n$ converging strongly in $V$ (hence in ${\mathbb L}^{2+2\sigma}$)   to $\hat U(0)$.
  For any $n \in \mathbb{N}$ the energy conservation law
\begin{equation}
\label{cons_energy_Gal_bis}
\mathcal{E}(\hat U_n(t)) = \mathcal{E}(\hat U^0_n), \quad t >0
\end{equation}
is satisfied.
We obtain  uniform estimates in $\EA$, as usual. Therefore
$\hat U_n(t)$
weakly converges in $\EA$, hence strongly in ${\mathbb L}^{2+2\sigma}$.
Passing to the limit as $n \rightarrow \infty$ in \eqref{cons_energy_Gal_bis} gives for any $t>0$,
\begin{equation*}
%\label{ineq_energy}
\frac 12 \|A^{\frac 12}\hat U(t)\|^2_H - \frac{\alpha}{2\sigma +2}\|\hat U(t)\|_{{\mathbb L}^{2\sigma+2}}^{2\sigma +2} \le \frac 12 \|A^{\frac 12} \hat U(0)\|^2_{H}- \frac{\alpha}{2\sigma +2}\|\hat U(0)\|_{{\mathbb L}^{2\sigma +2}}^{2\sigma +2} .
\end{equation*}
By uniqueness of the solution, by reversing time, we get the reverse inequality, hence the  equality holds
\begin{equation}
\label{eq_energy}
\mathcal{E}(\hat U(t)) = \mathcal{E}(\hat U(0)), \quad t >0,
\end{equation}
which proves that the energy is a random variable constant in time.
\\
As a consequence, we infer $\|\hat U_n(t)\|_V \rightarrow \|\hat U(t)\|_V$ for all $t \ge 0$. Together with $\hat U_n \rightarrow \hat U$ in $C_w([0, \infty);V)$ this yields $\hat U_n \rightarrow \hat U$ in $C([0, \infty);V)$ which proves the additional regularity $\hat U \in C([0, \infty);V)$ in $P^\prime_4)$.
\end{proof}

%%%%%%%%%%%%%%%%%%%%%%%%%%%
\section{Final remarks}
\label{final}

The stationary solution of the unforced NLS equation \eqref{eqS}
given by Theorem \ref{hatHdimle2} enjoys the conservation of mass and energy. On the other hand, if we drop Assumption \ref{ass_dim},
the stationary solution of \eqref{eqS}
given by Theorem \ref{teoUstaz} enjoys the conservation of mass but we are not able to deal with the conservation of energy. 
Indeed the solution is not regular enough to repeat for the energy what we have done for the mass. 
However in \cite{KuksinS} Kuksin and Shirikyan consider a more regular (parabolic) stochastic approximation, i.e. they deal with the complex Ginzburg-Landau equation
\begin{equation}
  {\rm d} u +[ \im \Delta  u-\im  |u|^{2} u -\gamma \Delta u]\,{\rm d}t
= \sqrt \gamma\Phi  \,{\rm d}W
\end{equation}
as a perturbation of the cubic defocusing Schr\"odinger equation
\begin{equation}
  \partial_t u  +\im \Delta  u-\im  |u|^{2} u =0.
  \end{equation}
  Therefore they get more regularity  %in space and in time 
and they can prove that in the limit both the mass and the energy are conserved. Indeed the limit invariant measure has support in $D(A)$.

In addition let us point out that our solution is less regular in space and in time, as well. Indeed we get that
$\hat U\in C([0,\infty);H)\cap C_w([0,\infty);V)$ and not the $C([0,\infty);\EA)$.  
Only under Assumption \ref{ass_dim} we can infer such kind of regularity exploiting the uniqueness of the solution.

Anyway, the conservation of the mass and the expression  \eqref{nonbanale} are the important results which show that the limit is not the 
trivial zero solution, and noises with different $\|\Phi\|_{L_{HS}(Y,H)}$ provide 
different stationary solutions to the unforced deterministic NLS  equation \eqref{eqS}.
Moreover the noise given in \eqref{espressione-del-noise} is a bit more general than that in \cite{Sh11}, where it was assumed that $\varphi_j=\varphi_{-j}$ for all $j$.

To conclude the comparison with the results by Kuksin and collaborators, we point out that  the stronger results by  Kuksin and Shirikyan in \cite{KuksinS} hold under stronger assumptions; they 
 work in a bounded domain of $\mathbb R^d$  with the Dirichlet boundary condition 
or on the box $[0,L]^d$ with  the odd-periodic boundary conditions, considering the defocusing case for $d\le 4$.
Here we deal with the focusing equation as well  and there are no restrictions  on the dimension $d$.
Moreover, we  consider more general polynomial nonlinearities and power $\beta$ of the negative Laplacian.  Also other boundary conditions or spatial domains can be considered, as explained in Appendix \ref{domini}.

Finally,  we comment on the recent works \cite{S} and \cite{SY} by Sy and collaborators, where the authors prove the existence of an invariant probability measure for the deterministic defocusing energy-{\it supercritical} NLS equation. 
This measure $\mu$ is supported on smooth Sobolev spaces; they show that the NLS equation is  globally well-posed for $\mu$-a.e. initial data in the support of this measure. Even if their technique is related to  that of Kuksin, the setting and the results are very different. The aim there is to deal with the  supercritical NLS equation; however, 
our results and those by Kuksin and collaborators hold for the  energy-subcritical case corresponding to Assumption \ref{F}.

\appendix

%%%%%%%%%%%%%%%%
\section{Estimates for the Galerkin approximation}
\label{app-Gal}

We collect the basic results on the Faedo-Galerkin approximation \eqref{eq-Gal}. 

First, we point out that
 in this section we  often refer to our previous paper  \cite{noi2d}, where the noise is more general (a conservative noise in the form of a multiplicative linear Stratonovich noise plus a nonlinear It\^o noise). 
There we were mainly interested in the  defocusing case for $d=2$ and $\beta=1$, because we wanted to get existence and uniqueness of solutions. 
In that more general setting  it was necessary to 
introduce another sequence of smoothing operators $S_n$ such that $\sup_{n\in \mathbb{N}}\|S_n\|_{\mathcal{L}({\mathbb L}^{2+2\sigma};{\mathbb L}^{2+2\sigma})}< \infty$, since the operators $P_n$, $n \in\mathbb{N}$, are not uniformly bounded from ${\mathbb L}^{2+2\sigma}$ to ${\mathbb L}^{2+2\sigma}$. 
However this smoothing is not needed when dealing with an additive noise. So our finite-dimensional 
equation \eqref{eq-Gal} is easier than that in \cite{noi2d}.

Anyway the results on the existence of martingale solutions 
for the Galerkin system proved in \cite{noi2d} hold also in the setting of the present paper. Therefore even if there are some differences in the setting, 
when  possible  we  refer to the computations done there;  otherwise we provide details of the 
proofs or at least we point out the main differences with respect to \cite{noi2d}.

We start by recalling the existence result  of a solution $u_n$ of the Faedo-Galerkin approximation \eqref{eq-Gal}. 
The proof is classical, see \cite[Proposition 5.6]{noi2d}.

%PROP
\begin{proposition}\label{prop-un-H}
Suppose Assumptions $\ref{ass-sigma}$ and $\ref{ass_G}$ hold and arbitrarily fix  $\gamma\ge 0$.
If $u^0_n$ is an $\mathcal{F}_0$-measurable random variable  such that
\[
\mathbb E\|u_n^0\|_H^2<\infty,
\]
 then   there exists a unique (strong in the probabilistic sense) global solution $u_n$ of equation $\eqref{eq-Gal}$,
with continuous paths in $H_n$. 
%Moreover\begin{equation}\label{stima-base-u_n-H}\mathbb E \left[\sup_{0\le t\le T }\|u_n(t)\|_H ^{2}\right]\lesssim_T\mathbb{E}\left[\|u_n^{0}\|_H ^{2}\right]+ \|P_n \Phi\|_{L_{HS}(Y,H)}^2.\end{equation}
\end{proposition}

Now we consider the evolution of the mass and of the energy.
%PROP
\begin{proposition}
Suppose Assumptions $\ref{ass-sigma}$ and $\ref{ass_G}$ hold and arbitrarily fix  $\gamma\ge 0$.
For the  solution $u_n$ to the finite-dimensional problem \eqref{eq-Gal},
the   mass, defined in   \eqref{mass}, evolves according to
\begin{equation}\label{mass-n}
{\rm d}\mathcal{M}(u_n(t))+ 2\gamma \mathcal{M}(u_n(t))\, {\rm d}t
=\|P_n \Phi \|^2_{L_{HS}(Y,H)}\, {\rm d}t+2\dualityReal{u_n(t)}{P_n \Phi \,{\rm d}W(t)}, \quad t>0,
\end{equation}
whereas the energy, defined in \eqref{energy},  evolves according to 
\begin{multline}\label{Ito_energy}
{\rm d}\mathcal{E}(u_n(t))+2 \gamma \mathcal{E}(u_n(t)) \, {\rm d}t=\alpha\gamma \frac \sigma{\sigma+1} \|u_n(t)\|_{{\mathbb L}^{2\sigma+2}}^{2\sigma+2} \, {\rm d}t
\\
-\sum_{j \in \mathbb{Z}_0}\Real \langle -A u_n(t)+\alpha |u_n(t)|^{2\sigma}u_n(t),P_n \Phi e_j^Y \rangle {\rm d}W_j(t)
+\frac 12
 \| A^{\frac 12 } P_n \Phi\|^2_{L_{HS}(Y,H)} \, {\rm d}t
 \\ - \frac \alpha 2 \||u_n(t)|^{\sigma}P_n \Phi\|^2_{L_{HS}(Y,H)}\, {\rm d}t
-\alpha \sigma \sum_{j \in \mathbb{Z}_0}
\langle |u_n(t)|^{2\sigma-2}, [\Real (\overline u_n(t)P_n \Phi e_j^Y)]^2\rangle \, {\rm d}t, \qquad t>0.
\end{multline}
\end{proposition}
%PROOF
\begin{proof}
\eqref{mass-n} and \eqref{Ito_energy} are obtained by the It\^o formula  for the mass and  the energy, respectively. 
For the mass one easily gets 
\begin{equation}
\begin{split}
{\rm d}\|u_n(t)\|_H ^{2}&=2 \productReal{u_n(t)}{{\rm d}u_n(t)}+\|P_n \Phi\|^2_{L_{HS}(Y,H)}{\rm d}t
\\&
=2 \productReal{u_n(t)}{\im A  u_n(t)-\im  \alpha P_n F(u_n(t)) -\gamma u_n(t)}{\rm d}t
\\&\; +2 \productReal{u_n(t)}{P_n \Phi  \,{\rm d} W(t)}
+\|P_n \Phi\|^2_{L_{HS}(Y,H)}{\rm d}t.
\end{split}
\end{equation}
One 
checks that 
\begin{equation}
\productReal{u_n}{\im A  u_n}=0, \qquad 
\productReal{u_n}{\im   P_n F(u_n)}=0
\end{equation}
and \eqref{mass-n} immediately follows.
We refer  to \cite[Section 5]{noi2d} for the energy.
\end{proof}
Proposition \ref{prop-Galerkin-esistenza_bis} below provides the estimate of the power moments of the $H$- and $V$-norms. For its proof we need the following technical lemma; see \cite[Lemma 5.6]{BHW-2019} for a proof.
\begin{lemma}
\label{lem_new}
Let $r \in [1, +\infty)$, $\varepsilon>0$, $T>0$ and $f \in L^r(\Omega, L^\infty(0,T))$. Then, 
\[
\|f\|_{L^r(\Omega, L^2(0,t))} \le \varepsilon \|f\|_{L^r(\Omega, L^\infty(0,t))}+ \frac{1}{4 \varepsilon} \int_0^t \|f\|_{L^r(\Omega,L^\infty(0,s))}\, {\rm d}s, \qquad t \in [0,T].
\]
\end{lemma}
Then, we have
\begin{proposition}\label{prop-Galerkin-esistenza_bis}
Suppose Assumptions $\ref{ass-sigma}$ and $\ref{ass_G}$ hold and arbitrarily fix  $\gamma\ge 0$, 
$p \in [1, \infty)$ and  $T>0$. For the solution $u_n$ to the finite-dimensional problem \eqref{eq-Gal},  we have
\begin{itemize}
\item [(i)] if $u^0_n$ is an $\mathcal{F}_0$-measurable random variable 
such that
\begin{equation}
\E \|u_n^0\|^{2p}_H<\infty,
\end{equation}
then
\begin{equation}
\label{sup_est_H}
\mathbb{E} \left[\sup_{0\le t\le T}\|u_n(t)\|^{2p}_{H} \right] 
\lesssim_{T,p} \mathbb{E}\left[ \|u_n^0\|^{2p}_{H}\right]+ \|P_n \Phi\|_{L_{HS}(Y,H)}^{2p};
\end{equation}
\item [(ii)]
if $u^0_n$ is an $\mathcal{F}_0$-measurable random variable 
such that
\begin{equation}\label{ipo_un0-V}
\E \|u_n^0\|^{2p(1+\sigma)}_V<\infty 
\end{equation}
and in addition when $ \alpha=1$
\begin{equation}\label{ipo_un0-H}
\E \|u_n^0\|^{2p(1+ \frac{2\beta\sigma}{2\beta-\sigma d})}_H<\infty,
\end{equation}
then
\begin{multline}
\label{sup_est_V_def}
\mathbb{E} \left[\sup_{0\le t\le T}\|u_n(t)\|^{2p}_{\EA} \right] 
\lesssim_{\sigma,T,p}
(\mathbbm{1}_{\alpha=-1}+ \mathbbm{1}_{\alpha=1} e^{C\gamma  pT})
\Big(1+ \mathbb{E} \|u_n^0\|^{2p}_{H}+  \mathbb{E} |\mathcal{E}(u_n^0)|^p
\\
+ \|P_n \Phi\|_{L_{HS}(Y,\EA)}^{2p(1+\sigma)} + \mathbbm{1}_{\alpha=1} \left[
 \mathbb{E}\|u_n^0\|_H^{2p(1+\frac{2\beta\sigma}{2\beta-\sigma d})} + \|P_n  \Phi\|_{L_{HS}(Y,H)}^{2p(1+\frac{2\beta\sigma}{2\beta-\sigma d})}\right]
\Big).
\end{multline}
\end{itemize}
\end{proposition}
%PROOF
\begin{proof}
The two estimates \eqref{sup_est_H} and \eqref{sup_est_V_def}
are proved in  \cite{noi2d}[\S 5] but in the statements the dependence on $\gamma$ and $\Phi$ is not explicitly given. 
For this reason we provide the proof. 
\\
(i) 
One starts with \eqref{mass-n} and, neglecting the damping term, gets
\begin{align}
\label{sti_new_H}
\|u_n(t)\|^2_H
\le \|u_n^0\|^2_H+\|P_n \Phi \|^2_{L_{HS}(Y,H)}t+2\int_0^t\dualityReal{u_n(s)}{P_n \Phi \,{\rm d}W(s)}.
\end{align}
Denote by $L^p_\Omega( L^q_T)$ the  norm given by 
$[\mathbb E (\|\cdot\|_{L^q(0,T)}^p)]^{\frac 1p}$ and 
let us  estimate the various terms that appear in the right-hand side of
\eqref{sti_new_H} in this norm. Set $I_{0,n}(u)(t):=2\int_0^t\dualityReal{u(s)}{P_n \Phi \,{\rm d}W(s)}$. 
By the Burkholder-Davis-Gundy inequality we get
\[
\|I_{0,n}(u_n)\|_{L^p_\Omega( L^\infty_T)} 
\lesssim
 \left \Vert\left( \sum_{j \in \mathbb{Z}_0} \left| \text{Re}\langle u_n, P_n \Phi e_j^Y\rangle\right|^2\right)^{\frac 12} \right\Vert_{L^p_\Omega(L^2_T)}.
\]
Since
\[
\sum_{j \in \mathbb{Z}_0}\left| \text{Re}\langle u_n, P_n \Phi e_j^Y\rangle\right|^2
\le
\|u_n\|_H^2  \|P_n \Phi\|_{L_{HS}(Y,H)}^{2},
\]
then, by Lemma \ref{lem_new}, we get
\begin{align*}
\|I_{0,n}(u_n)\|_{L^p_\Omega( L^\infty_T)} 
&\lesssim \left \Vert \|u_n\|^2_H\right\Vert_{L^p_\Omega(L^2_T)} 
 + \sqrt T \|P_n \Phi\|_{L_{HS}(Y,H)}^{2}
\\\notag
&\lesssim  \varepsilon \left\Vert\|u_n\|^2_H\right\Vert_{L^p_\Omega( L^\infty_T)} + \frac{1}{4\varepsilon}\int_0^T \left\Vert\|u_n\|_H^2\right\Vert_{L^p_\Omega( L^\infty_s)}\, {\rm d}s 
+ \sqrt T \|P_n \Phi\|_{L_{HS}(Y,H)}^{2},
\end{align*}
for  any $\varepsilon>0$.
Coming back to \eqref{sti_new_H} and inserting the above estimate, by choosing $\varepsilon$ sufficiently small, we get 
\begin{align}
\label{sti_intermedia_mass}
\left\Vert\|u_n\|^2_H\right\Vert_{L^p_\Omega( L^\infty_T)} 
&\lesssim
\left\Vert\|u_n^0\|^2_H\right\Vert_{L^p_\Omega}+ (1+T) \|P_n \Phi\|_{L_{HS}(Y,H)}^{2}
+\int_0^T \left\Vert\|u_n\|_H^2\right\Vert_{L^p_\Omega( L^\infty_s)}\, {\rm d}s.
\end{align}
From the Gronwall lemma  we infer \eqref{sup_est_H}.
\\
(ii)
 First, notice that the assumption \eqref{ipo_un0-V}  implies that 
$ \mathbb{E}( |\mathcal{E}(u_n^0)|^p)$ is finite; this comes from the definition of the energy \eqref{energy} 
and the embedding $\EA\subset {\mathbb L}^{2+2\sigma}$.
\\
 As far as \eqref{sup_est_V_def} is concerned, since $\|u\|_\EA^2=\|u\|_H^2+\|A^{\frac 12} u\|_H^2$ and we know the bound \eqref{sup_est_H} for the $H$-norm, it is enough to estimate the moments of
 $\|A^\frac 12 u_n\|^2_{L^\infty(0,T;H)}$. We proceed differently in the defocusing and  in the focusing case.
 
We start dealing with the  defocusing case where, by the very definition of the energy $\mathcal E$ ($\equiv \mathcal E_{-1}$), we have
\begin{equation}
\label{ineq_def}
\|A^\frac 12u\|^2_H\lesssim {\mathcal E}(u).
\end{equation}
To conclude is thus enough to get a bound on the moments of $\mathcal{E}(u_n)$.

Let us consider the integral form of \eqref{Ito_energy} and estimate the $L^p_\Omega(L^\infty_T)$-norm of the terms that appear in the r.h.s.. Since $ \gamma>0$ and $\alpha \gamma<0$, we neglect the  terms depending on $\gamma$ and obtain
\begin{multline}\label{Ito-energy-integral}
\mathcal{E}(u_n(t))
\le\mathcal{E}(u_n^0)+ \sum_{j \in \mathbb{Z}_0}\int_0^t  \Real \langle A u_n(s)+ |u_n(s)|^{2\sigma}u_n(s),P_n \Phi e_j^Y \rangle {\rm d}W_j(s)
+\frac 12
 \| A^{\frac 12 } P_n \Phi\|^2_{L_{HS}(Y,H)} t
 \\ +\int_0^t \left( \frac1 2 \||u_n(s)|^{\sigma}P_n \Phi\|^2_{L_{HS}(Y,H)}
+\sigma   \sum_{j \in \mathbb{Z}_0}\langle |u_n(s)|^{2\sigma-2}, [\Real (\overline u_n(s)P_n \Phi e_j^Y)]^2\rangle\right) \, {\rm d}s, \qquad t>0.
\end{multline}
Denote by  $I_{1,n}(u_n)(t)$ the deterministic integral. 
Using  \eqref{H1-e-Lsigma} we estimate
\begin{align*}
\frac12 \||u_n(t)|^{\sigma}&P_n \Phi\|^2_{L_{HS}(Y,H)}+\sigma  
\sum_{j \in \mathbb{Z}_0}\langle |u_n(t)|^{2\sigma-2}, [\Real (\overline u_n(t)P_n \Phi e_j^Y)]^2\rangle 
\\
&\le\left(1+\sigma \right)\|u_n(t)\|^{2\sigma}_{{\mathbb L}^{2+2\sigma}}
\sum_{j \in \mathbb{Z}_0}\|P_n\Phi e_j^Y\|^2_{{\mathbb L}^{2+2\sigma}}
\le \left(1+\sigma \right)\|u_n(t)\|^{2\sigma}_{{\mathbb L}^{2+2\sigma}} \|P_n\Phi\|^2_{L_{HS}(Y,V)}
\\
&\le \frac{1}{2+2\sigma}\|u_n(t)\|^{2+2\sigma}_{{\mathbb L}^{2+2\sigma}}+ C_\sigma \|P_n\Phi\|^{2\sigma +2}_{L_{HS}(Y,V)}
\le \mathcal E(u_n(t))+ C_\sigma \|P_n\Phi\|^{2\sigma +2}_{L_{HS}(Y,V)}.
\end{align*}
The Minkowski inequality thus yields 
\begin{equation*}
\left \Vert I_{1,n}(u_n) \right \Vert_{L^p_\Omega( L^\infty_T)}
\lesssim 
C_\sigma T \|P_n\Phi\|^{2+2\sigma}_{L_{HS}(Y,V)}+ \int_0^T \|\mathcal E(u_n)\|_{L^p_\Omega( L^\infty_s)}\, {\rm d}s.
\end{equation*}
\\
Denote by  $I_{2,n}(u_n)(t)$ the stochastic  integral in \eqref{Ito-energy-integral}. 
Exploiting the Burkholder-Davis-Gundy inequality we infer 
\begin{equation*}
\|I_{2,n}(u_n)\|_{L^p_\Omega( L^\infty_T)} 
\lesssim
 \left \Vert\left( \sum_{j \in \mathbb{Z}_0}\left| \text{Re}\langle Au_n+|u_n|^{2\sigma}u_n, P_n \Phi e_j^Y\rangle\right|^2\right)^{\frac 12} \right\Vert_{L^p_\Omega(L^2_T)}.
\end{equation*}
We have
\begin{equation}\label{stimaAphi}
\left( \sum_{j \in \mathbb{Z}_0}\left| \text{Re}\langle Au_n, P_n \Phi e_j^Y\rangle\right|^2\right)^{\frac 12}
\le \|A^\frac 12u_n\|_H  \| P_n \Phi\|_{L_{HS}(Y,\EA)} 
\le \tfrac 12 \|A^\frac 12u_n\|_H^2 +\tfrac 12   \| P_n \Phi\|^2_{L_{HS}(Y,\EA)} 
\end{equation}
and, thanks to the H\"older and Young inequalities and
the continuous Sobolev embedding $\EA \subset {\mathbb L}^{2+2\sigma}$,
\begin{align}\label{stimaFphi}
\left( \sum_{j \in \mathbb{Z}_0}\left| \text{Re}\langle  |u_n|^{2\sigma}u_n, P_n \Phi e_j^Y\rangle\right|^2\right)^{\frac 12} 
&\le \| |u_n|^{2\sigma}u_n\|_{{\mathbb L}^{\frac{2+2\sigma}{1+2\sigma}}}\|P_n\Phi\|_{\gamma(Y,{\mathbb L}^{2+2\sigma})}
\notag\\
&\lesssim  \| u_n\|^{1+2\sigma}_{{\mathbb L}^{2+2\sigma}}\|P_n \Phi\|_{L_{HS}(Y,\EA)}
\\ \notag
&\lesssim
\frac{1}{2+2\sigma }\| u_n\|^{2+2\sigma}_{{\mathbb L}^{2+2\sigma}} + C_\sigma \|P_n \Phi\|_{L_{HS}(Y,\EA)}^{2+2\sigma}.
\end{align}
Thus, thanks to Lemma \ref{lem_new} we infer for $\varepsilon>0$,
\begin{align}
\label{sti4}
\|I_{2,n}(u_n)\|_{L^p_\Omega( L^\infty_T)} 
&\lesssim \left \Vert \mathcal E(u_n)\right\Vert_{L^p_\Omega(L^2_T)} 
 + \sqrt T (\|P_n \Phi\|_{L_{HS}(Y,\EA)}^{2}+C_\sigma\|P_n \Phi\|_{L_{HS}(Y,\EA)}^{2+2\sigma})
\notag \\
&\lesssim  \varepsilon \| \mathcal E(u_n)\|_{L^p_\Omega( L^\infty_T)} 
+ \frac{1}{4\varepsilon}\int_0^T \| \mathcal E(u_n)\|_{L^p_\Omega( L^\infty_s)}\, {\rm d}s 
\\ \notag
&\quad + \sqrt T (\|P_n \Phi\|_{L_{HS}(Y,\EA)}^{2}+C_\sigma\|P_n \Phi\|_{L_{HS}(Y,\EA)}^{2+2\sigma}).
\end{align}
Inserting the latter estimates in \eqref{Ito-energy-integral}  we  get 
\begin{align}
\label{sti_intermedia_energy}
\|\mathcal{E}(u_n)\|_{L^p_\Omega( L^\infty_T)}
& \lesssim_\sigma \|\mathcal{E}(u_n^0)\|_{L^p_\Omega} + (1+ T)\left( \|P_n \Phi\|_{L_{HS}(Y,\EA)}^{2}+\|P_n \Phi\|_{L_{HS}(Y,\EA)}^{2+2\sigma}\right)
\\\notag
&\qquad + \varepsilon \|\mathcal{E}(u_n)\|_{L^p_\Omega( L^\infty_T)} + \left(\frac{1}{4\varepsilon}+1\right)\int_0^T \|\mathcal{E}(u_n)\|_{L^p_\Omega( L^\infty_s)}\, {\rm d}s.
\end{align}
Choosing $\varepsilon$ sufficiently small,  we infer 
\begin{align*}
\|\mathcal{E}(u_n)\|_{L^p_\Omega( L^\infty_T)}
& \lesssim_\sigma \|\mathcal{E}(u_n^0)\|_{L^p_\Omega}+ \int_0^T \|\mathcal{E}(u_n)\|_{L^p_\Omega( L^\infty_s)}\, {\rm d}s
\\
&
+(1+T)\left( \|P_n \Phi\|_{L_{HS}(Y,H)}^{2}
+\|P_n \Phi\|_{L_{HS}(Y,\EA)}^{2+2\sigma}\right).
\end{align*}
From the Gronwall lemma we get
\[
\|\mathcal{E}(u_n)\|_{L^p_\Omega( L^\infty_T)}
\le e^{CT}
\left( \|\mathcal{E}(u_n^0)\|_{L^p_\Omega}
+(1+T)(\|P_n \Phi\|_{L_{HS}(Y,H)}^{2}
+\|P_n \Phi\|_{L_{HS}(Y,\EA)}^{2+2\sigma}) \right).
\]
Bearing in mind \eqref{sup_est_H}, from the above estimate and \eqref{ineq_def} we infer \eqref{sup_est_V_def}.

Now we consider   \eqref{sup_est_V_def} in the focusing case; this is the only estimate depending on $\gamma$. Since in this case $\mathcal{E}(\cdot)$ could be negative, the bound \eqref{ineq_def} does not hold and we cannot argue as in the defocusing case. We proceed with a direct estimate of the moments of  $\|A^\frac 12 u_n\|^2_{L^\infty(0,T;H)}$, as follows.
Recalling the definition \eqref{energy} of energy,  we have  for $\alpha=1$
\[
\frac 12\|A^\frac 12 u\|^2_H
=  \frac{1}{2+2\sigma}\|u\|_{{\mathbb L}^{2+2\sigma}}^{2+2\sigma} +\mathcal{E}(u) .
\]
According to the evolution  \eqref{Ito_energy} of the energy, we get
\begin{multline}
\label{sti1}
\frac 12\|A^\frac 12 u_n(t)\|^2_H
= 
  \frac{1}{2+2\sigma}\|u_n(t)\|_{{\mathbb L}^{2+2\sigma}}^{2+2\sigma} +\mathcal{E}(u_n(0))
  \\
-2 \gamma \int_0^t\mathcal{E}(u_n(s)) \, {\rm d}s + \gamma \frac \sigma{\sigma+1}\int_0^t \|u_n(s)\|_{{\mathbb L}^{2+2\sigma}}^{2+2\sigma} \, {\rm d}s
\\
-\int_0^t \sum_{j \in \mathbb{Z}_0}\Real \langle -A u_n(s)+ |u_n(s)|^{2\sigma}u_n(s),P_n \Phi e^Y_j \rangle
 {\rm d}W_j(s)+\frac 12
 \|P_n A^{\frac 12} \Phi\|^2_{L_{HS}(Y,H)} \,t
 \\ - \frac 1 2 \int_0^t\||u_n(s)|^{\sigma}P_n \Phi\|^2_{L_{HS}(Y,H)}\, {\rm d}s
- \sigma \int_0^t \sum_{j \in \mathbb{Z}_0}\langle |u_n(s)|^{2\sigma-2}, [\Real (\overline u_n(s)P_n \Phi e^Y_j)]^2\rangle \, {\rm d}s, 
\end{multline}
almost surely and  for all $t \in [0,T]$.

The last two terms in the right-hand side  can be neglected being negative. 
 We will estimate the other  terms  in the $L^p_\Omega( L^\infty_T)$-norm.
 
 We will use the short notation
 \begin{equation}
\label{Z}
Z(u):= \|A^{\frac 12} u\|^2_H + \|u\|^{2+2\sigma}_{{\mathbb L}^{2+2\sigma}}.
\end{equation}
Using \eqref{GN-somma} and the Young inequality we have  
\begin{align}
\label{sti2} \notag
\left \Vert\|u_n^{2+2\sigma}\|_{{\mathbb L}^{2+2\sigma}}\right \Vert_{L^p_\Omega( L^\infty_T)}
&\le 
\varepsilon \left \Vert\|u_n\|_\EA^2\right\Vert_{L^p_\Omega( L^\infty_T)}
+ C_\varepsilon \left \Vert \|u_n \|_{H}^{2+\frac{4\beta\sigma}{2\beta-\sigma d}}\right\Vert_{L^p_\Omega( L^\infty_T)}
\\
&
\le 
\varepsilon \left \Vert\|A^{\frac 12} u_n\|_H^2\right\Vert_{L^p_\Omega( L^\infty_T)}
+\varepsilon \left \Vert\|u_n\|_H^2\right\Vert_{L^p_\Omega( L^\infty_T)}
\\  \notag
&
\quad + C_\varepsilon \left \Vert \|u_n \|_{H}^{2+\frac{4\beta\sigma}{2\beta-\sigma d}}
\right\Vert_{L^p_\Omega( L^\infty_T)}.
\end{align}
We set $I_{3}(u)(t):=-2 \gamma \int_0^t\mathcal{E}(u(s)) \, {\rm d}s + \gamma \frac \sigma{\sigma+1}\int_0^t \|u(s)\|_{{\mathbb L}^{2+2\sigma}}^{2+2\sigma} \, {\rm d}s$.
Recalling \eqref{energy} we have 
\begin{equation*}
I_{3}(u_n)(t)\le \gamma \int_0^t \|u_n(s)\|_{{\mathbb L}^{2+2\sigma}}^{2+2\sigma} \, {\rm d}s
\end{equation*}
and proceeding similarly to the previous estimate we get 
\begin{multline}
\label{sti3}
\|I_{3}(u_n)\|_{L^p_\Omega( L^\infty_T)}
\le \gamma \int_0^T\left\Vert\|u_n\|_{{\mathbb L}^{2+2\sigma}}^{2+2\sigma}\right\Vert_{L^p_\Omega( L^\infty_s)}\, {\rm d} s
\\
\le  \varepsilon \gamma\int_0^T \left \Vert\|A^{\frac 12} u_n\|^2_H \right\Vert_{L^p_\Omega( L^\infty_s)}\,{\rm d}s
+ \varepsilon \gamma\int_0^T \left \Vert\|u_n\|^2_H \right\Vert_{L^p_\Omega( L^\infty_s)}\,{\rm d}s
\\
+C_\varepsilon \gamma\int_0^T\left \Vert \|u_n\|_{H}^{2+\frac{4\beta\sigma}{2\beta-\sigma d}}\right\Vert_{L^p_\Omega( L^\infty_s)}\, {\rm d}s.
\end{multline}
We set $I_{4,n}(u)(t):=-\sum_{j \in \mathbb{Z}_0}\int_0^t \Real \langle -A u(s)+ |u(s)|^{2\sigma}u(s),P_n \Phi e_j^Y \rangle {\rm d}W_j(s)$.  

Proceeding similarly as done in  \eqref{sti4} for $I_{2,n}$  (for  the defocusing case), we obtain 
\begin{align}
%\label{sti4}
\|I_{4,n}(u_n)\|_{L^p_\Omega( L^\infty_T)} 
&\lesssim  \varepsilon \|Z(u_n)\|_{L^p_\Omega( L^\infty_T)} + \frac{1}{4\varepsilon}\int_0^T \|Z(u_n)\|_{L^p_\Omega( L^\infty_s)}\, {\rm d}s 
\\ \notag
&\quad + \sqrt T (\|P_n \Phi\|_{L_{HS}(Y,\EA)}^{2}+\|P_n \Phi\|_{L_{HS}(Y,\EA)}^{2+2\sigma}).
\end{align}
By the very definition \eqref{Z} of $Z$ and  by \eqref{GN-somma} 
we have
\[
Z(u)
\le(1+\varepsilon) \|A^\frac 12 u\|^2_H + \varepsilon \| u\|^2_H
+C_\epsilon \|u \|_{H}^{2+\frac{4\beta\sigma}{2\beta-\sigma d}}
\]
and by Young inequality we  get 
\[
Z(u)
\le(1+\varepsilon) \|A^\frac 12 u\|^2_H + C(1+
 \|u \|_{H}^{2+\frac{4\beta\sigma}{2\beta-\sigma d}}).
\]
We come back to  \eqref{sti1} and insert the latter estimates; 
using again  Young inequality we  get 
\begin{multline}
\label{sti_6}
\tfrac 12 \left\Vert\|A^\frac 12 u_n\|^2_H\right\Vert_{L^p_\Omega( L^\infty_T)} 
\le
  \|\mathcal{E}(u_n^0)\|_{L^p_\Omega}+ \varepsilon(2+\varepsilon) \left \Vert\|A^\frac 12u_n\|^2_H\right\Vert_{L^p_\Omega( L^\infty_T)}
\\
 +
(\gamma\varepsilon+\tfrac {1+\varepsilon}{4\varepsilon}) \int_0^T \left \Vert\|A^\frac 12u_n\|^2_H\right\Vert_{L^p_\Omega( L^\infty_s)}\,{\rm d}s
+ C_\varepsilon (1+\left \Vert \|u_n \|_{H}^{2+\frac{4\beta\sigma}{2\beta-\sigma d}}\right\Vert_{L^p_\Omega( L^\infty_T)})
\\
+ C_\varepsilon (1 + \gamma)\int_0^T(1+\left \Vert \|u_n\|_{H}^{2+\frac{4\beta\sigma}{2\beta-\sigma d}}\right\Vert_{L^p_\Omega(L^\infty_s)})\, {\rm d}s
+ C_T (1+\|P_n \Phi\|_{L_{HS}(Y,\EA)}^{2+2\sigma}).
\end{multline}
Choosing $\varepsilon>0$ small enough and using the estimate \eqref{sup_est_H} for the powers of the mass, 
we get
\begin{multline*}
\left\Vert\|A^\frac 12 u_n\|^2_H\right\Vert_{L^p_\Omega(L^\infty_T)}
\le
C  \|\mathcal{E}(u_n^0)\|_{L^p_\Omega}
+
C(\gamma+1)\int_0^T \left \Vert\|A^\frac 12u_n\|^2_H\right\Vert_{L^p_\Omega(L^\infty_s)}\,{\rm d}s
\\+C_T (\gamma+1)\left(1+    \left\Vert\|u_n^0\|^2_H\right\Vert^{1+\frac{2\beta\sigma}{2\beta-d\sigma}}_{L_\Omega^{p(1+\frac{2\beta\sigma}{2\beta-d\sigma})}}
+ \|P_n \Phi\|_{L_{HS}(Y,H)}^{2+\frac{4\beta\sigma}{2\beta-\sigma d}}
+ \|P_n \Phi\|_{L_{HS}(Y,\EA)}^{2+2\sigma}\right)
.
\end{multline*}
By the Gronwall Lemma and bearing in mind \eqref{sup_est_H} we get the estimate 
\begin{align*}
\left\Vert\|A^\frac 12 u_n\|^2_H\right\Vert_{L^p_\Omega(L^\infty_T)}
\lesssim_{p,T} %C_T(1 + \gamma)(1+ e^{\gamma T})
%(1+ e^{\gamma T})
e^{C\gamma   T}\left(1+   \|\mathcal{E}(u_n^0)\|_{L^p_\Omega}+  \left\Vert\|u_n^0\|^2_H\right\Vert^{1+\frac{2\beta\sigma}{2\beta-d\sigma}}_{L_\Omega^{(1+\frac{2\beta\sigma}{2\beta-d\sigma})p}} + \|P_n \Phi\|_{L_{HS}(Y,H)}^{2+\frac{4\beta\sigma}{2\beta-\sigma d}}
+ \|P_n \Phi\|_{L_{HS}(Y,\EA)}^{2+2\sigma}\right).
\end{align*} 
Bearing in mind the previous estimate for the power moments of the $H$-norm and recalling  \eqref{bound_Phi_Pn} we get 
\eqref{sup_est_V_def}.
\end{proof}

Now we  provide power moments estimates of the mass and the energy
for the Galerkin equation \eqref{eq-Gal}, which are  uniform in time on the interval $[0,\infty)$ and 
with an explicit dependence on the parameter $\gamma$.

For the mass we have the following result.
% PROP
\begin{proposition}
\label{bound_lemma_mass} 
Let Assumptions $\ref{ass-sigma}$ and $\ref{ass_G}$ hold and  fix $\gamma>0$. 
If  $u^0_n$ is an $\mathcal{F}_0$-measurable random variable 
such that
\[
\E \|u_n^0\|^{2}_H<\infty,
\]
then the solution $u_n$ of the Galerkin equation \eqref{eq-Gal} fulfils
\begin{equation}\label{ito-base}
\frac{{\rm d}}{{\rm d}t}\mathbb E {\mathcal M}(u_n(t))+2\gamma \mathbb E{\mathcal M}(u_n(t))
=
\|P_n\Phi\|_{L_{HS}(Y,H)}^2
 \end{equation}
 for every $ t>0$. 
 
 Moreover,  fix $p\in[1, \infty)$.
 If
\[
\E \|u_n^0\|^{2p}_H<\infty,
\]
then
 \begin{equation}\label{p_est_mass}
 \mathbb E [{\mathcal M}(u_n(t))^p]
 \le
 e^{-\gamma p t }\mathbb{E}\left[{\mathcal M}(u_n^{0})^p\right]+
 \|P_n \Phi \|^{2p}_{L_{HS}(Y,H)}C_p \gamma^{-p},\qquad t\ge 0
  \end{equation}
where the constant $C_p$ does not depend on $n$.
\end{proposition}
% PROOF
\begin{proof}
 Let us start by proving the first equality.
 Consider the identity \eqref{mass-n}. Taking the expected value and using the fact that the stochastic integral is a martingale by Proposition \ref{prop-Galerkin-esistenza_bis}, we obtain \eqref{ito-base}.
Next we follow   \cite[Proposition 2.7]{large-damping}; indeed, even if   \cite{large-damping}  deal with 
the Schr\"odinger equation  in $\mathbb{R}^d$, $d \ge 1$, let us  notice that 
 those computations rely on the Hamiltonian structure of the equation which is independent of the underlying geometry of the domain. Therefore, we do not repeat those computations but quote the basic ones:
 for any $\epsilon>0$ and $p>1$  by the It\^o formula for $\mathcal{M}(u_n(t))^{p}$ we get
 \begin{equation}\label{massa-epsilon}
  \begin{split}
\mathcal{M}&(u_n(t))^{p} 
\le 
 \mathcal{M}( u_n(0))^{p}-(2-\epsilon )\lambda p \int_0^t \mathcal{M}(u_n(s))^{p}\, {\rm d}s
\\
&+C_{\epsilon,p} \|P_n \Phi \|^{2p}_{L_{HS}(Y,H)} \lambda^{1-p} t
+2p \int_0^t \mathcal{M}(u_n(s))^{p-1} \text{Re}\langle u_n(s),\Phi {\rm d}W(s)\rangle.
\end{split}\end{equation}
 Choosing $\epsilon=1$ and  taking the mathematical expectation to get rid of the stochastic integral, 
 by means of Gronwall lemma we get \eqref{p_est_mass}.
\end{proof}

For the energy and  the modified energy we have the following result.
% PROPOSITION energy
\begin{proposition}
\label{bound_lemma_energy}
Let Assumptions $\ref{ass-sigma}$ and $\ref{ass_G}$ hold; fix $\gamma>0$ and  $p\in[1,+ \infty)$. For
the solution $u_n$ of the Galerkin equation \eqref{eq-Gal} we have:
\begin{itemize}
\item[(i)] in the defocusing case ($\alpha=-1$): if $u^0_n$ is an $\mathcal{F}_0$-measurable random variable such that \eqref{ipo_un0-V} holds true, then
 \begin{equation}\label{energia-p-n}
 \mathbb E [{\mathcal E}(u_n(t))^p]\le
   e^{-\gamma p t } \mathbb{E}\left[{\mathcal E}( u_n^{0})^p\right]
   + C\phi_{-1}(\sigma,\gamma,P_n\Phi)^{p}\gamma^{-p},
 \end{equation} 
 for any $t>0$, where the constant $C=C(\beta,\sigma,p)$  is independent of  $n$ and $\gamma$, and 
$\phi_{-1}$ is defined in \eqref{phi_alpha};
 \item [(ii)] in the focusing case ($\alpha=1$), if $u^0$ is an $\mathcal{F}_0$-measurable random variable such that \eqref{ipo_un0-V} and \eqref{ipo_un0-H} hold true, then
 \begin{multline}\label{energia-modif-p-n}
\mathbb E [{\mathcal E}_1(u_n(t))^p]  
\le 
e^{-p \frac{1+\sigma}{1+2\sigma} \gamma t}  
\left[  \mathbb{E}\left[{\mathcal E}_1( u_n^{0})^p\right]+C_1 \| P_n\Phi\|^p_{L_{HS}(Y,H)} \gamma^{-\frac{p}2} \mathbb{E}\left[\mathcal M(u_n^{0})^p\right]\right]
  \\
  +C_2 e^{-p a \gamma t}  
\mathbb{E}\left[\mathcal M(u_n^{0})^{p(\frac12+\frac{2\beta\sigma}{2\beta-\sigma d})}\right]
\| P_n\Phi\|^{p}_{L_{HS}(Y,H)} \gamma^{-\frac{p}2}
+C_3\phi_{1}(d,\beta,\sigma,\gamma, P_n\Phi)^p \gamma^{-p},
 \end{multline}
for any $t>0$, where the positive constants 
$a, C_1,C_2$ and $C_3$ depend on the parameters $d, \beta,\sigma,p$ but neither on $n$ nor on  
$\gamma$, and $\phi_{1}$ is defined in \eqref{phi_alpha}.
\end{itemize}
\end{proposition}
%PROOF
\begin{proof} 
 First, notice that the assumption \eqref{ipo_un0-V}  implies that  in the defocusing case the energy moment
$ \mathbb{E}( \mathcal{E}_1(u_n^0)^p)$ is finite; by adding 
 assumption\eqref{ipo_un0-H}   we also get that  in the focusing case the modified energy moment
$ \mathbb{E}( \mathcal{E}_1(u_n^0)^p)$ is finite; this comes from the definitions of 
the energy \eqref{energy}  and 
the modified energy \eqref{modifEnergy},
thanks to  the embedding $\EA\subset {\mathbb L}^{2+2\sigma}$.

For $p=1$ 
the bounds for the energy and for the modified energy are obtained from 
the identities \eqref{Ito_energy}  and  \eqref{mass-n} by estimating the right hand sides.  For $p\ge 2$ 
we consider the It\^o formula for the power $p$ and proceed in a similar way. 
For $1<p<2$ we proceed  by means of the H\"older inequality as before, using the
estimate for $p = 2$.

In this way we proved a similar result in \cite{large-damping}. However
there we considered $\beta=1$ and  the spatial domain $D$ was $\mathbb R^d$. This implies that 
 the Gagliardo-Niremberg inequality is now different;
 indeed in $\mathbb R^d$  the Gagliardo-Niremberg inequality is
\begin{equation}
\|u\|_{L^{2+2\sigma}(\mathbb R^d)}
\le C \|u\|_{L^2(\mathbb R^d)}^{1-\frac{\sigma d}{2\beta(1+\sigma)}} \|A^{\frac 12} u\|_{L^2(\mathbb R^d)}^{\frac{\sigma d}{2\beta(1+\sigma)}}
\end{equation}
whereas here in the Gagliardo-Niremberg inequality \eqref{GN-ineq} the last term involves the $\EA$-norm.
Thus  since there are differences with respect to \cite{large-damping} we provide the details to obtain the estimate \eqref{energia-modif-p-n} in the focusing case, 
whereas the defocusing case can be dealt  with as in  \cite{large-damping}. 

We write the It\^o formula for the modified energy 
$\mathcal{E}_1(u)= {\mathcal{E}}(u)+\frac 12 \mathcal M(u)
+G\mathcal M(u)^{1+\frac{2\beta\sigma}{2\beta-\sigma d}}$. 
Recalling \eqref{mass-n} for the mass,   \eqref{massa-epsilon} for the power 
$p=1+\frac{2\beta\sigma}{2\beta-\sigma d}$ of the  mass and \eqref{Ito_energy} for the energy 
when $\alpha=1$ (we neglect the last two terms of \eqref{Ito_energy}), we have
\begin{multline}\label{stima-d-tildeH}
{\rm d}{\mathcal{E}_1}(u_n(t))+2 \gamma {\mathcal{E}_1}(u_n(t)) \,{\rm d}t
\\
\le\gamma \tfrac \sigma{\sigma+1} \|u_n(t)\|_{{\mathbb L}^{2\sigma+2}}^{2\sigma+2} \,{\rm d}t
+ \gamma \left(\epsilon(1+\tfrac {2\beta\sigma}{2\beta-\sigma d})- \tfrac{4\beta\sigma}{2\beta-\sigma d} \right) G \mathcal M(u_n(t))^{1+\frac{2\beta\sigma}{2\beta-\sigma d}} \,{\rm d}t\;
\\
\qquad +C_\epsilon \|P_n \Phi \|^{2+\frac{4\beta\sigma}{2\beta-\sigma d}}_{L_{HS}(Y,H)}  \gamma^{-\frac{2\beta\sigma}{2\beta-\sigma d}}\,{\rm d}t
 +\tfrac 12
 \|P_n \Phi\|^2_{L_{HS}(Y,\EA)} \,{\rm d}t
\\
-\sum_{j \in \mathbb{Z}_0}\Real \langle -A u_n(t)+ |u_n(t)|^{2\sigma}u_n(t),P_n \Phi e^Y_j\rangle \,{\rm d}W_j(t)
\\
+2\left[(1+\tfrac{2\beta\sigma}{2\beta-\sigma d}) G \mathcal{M}(u_n(s))^{\frac{2\beta\sigma}{2\beta-\sigma d}} +1\right]\text{Re}\langle u_n(t),P_n \Phi {\rm d}W(t)\rangle.
\end{multline}
We have $(1-\frac{2\beta}{2\beta-\sigma d})< 0$  by assumption; then choosing  $\epsilon$ small so to have 
$\epsilon(1+\frac {2\beta\sigma}{2\beta-\sigma d})+2\sigma (1-\frac{2\beta}{2\beta-\sigma d})<0 $, we 
estimate the second line in \eqref{stima-d-tildeH} as follows
\[\begin{split}
 \tfrac \sigma{\sigma+1} \|u_n\|_{{\mathbb L}^{2\sigma+2}}^{2\sigma+2} 
&+ \left(\epsilon(1+\tfrac {2\beta\sigma}{2\beta-\sigma d})- \tfrac{4\beta\sigma}{2\beta-\sigma d} \right) G \mathcal M(u_n)^{1+\frac{2\beta\sigma}{2\beta-\sigma d}} 
\\&\underset{\eqref{constanteG}}{\le}
\tfrac \sigma {2+2\sigma} \| u_n\|_\EA^2+\left( \epsilon(1+\tfrac {2\beta\sigma}{2\beta-\sigma d})+2\sigma (1-\tfrac{2\beta}{2\beta-\sigma d})\right) G \mathcal M(u_n)^{1+\frac{2\beta\sigma}{2\beta-\sigma d}}
\\&
\le \tfrac \sigma {2+2\sigma}  \| u_n\|_\EA^2 
\underset{\eqref{nabla-Htilde}}{\le} \tfrac{2\sigma}{1+2\sigma} {\mathcal E}_1(u_n).
\end{split}
\]
Then, we insert this bound in \eqref{stima-d-tildeH} and obtain
\begin{multline}\label{stima-tildeH-con-martingala}
{\rm d}{\mathcal{E}_1}(u_n(t))+2 \tfrac{1+\sigma}{1+2\sigma}\gamma {\mathcal{E}_1}(u_n(t)) \,{\rm d}t
\\
\le \Big(
 C \|P_n\Phi \|^{2+\frac{4\beta\sigma}{2\beta-\sigma d}}_{L_{HS}(Y,H)}  \gamma^{-\frac{2\beta\sigma}{2\beta-\sigma d}} +\tfrac 12
 \| P_n\Phi\|^2_{L_{HS}(Y,\EA)}\Big) \,{\rm d}t
\\
-\sum_{j \in \mathbb{Z}_0}\Real \langle -A u_n(t)+ |u_n(t)|^{2\sigma}u_n(t),P_n \Phi e^Y_j\rangle \,{\rm d}W_j(t)
\\
+2\left[(1+\tfrac{2\beta\sigma}{2\beta-\sigma d}) G \mathcal{M}(u_n(s))^{\frac{2\beta\sigma}{2\beta-\sigma d}} +1\right]
\Real \langle u_n(t),P_n \Phi {\rm d}W(t)\rangle.
\end{multline}
So
considering the mathematical expectation to get rid of the stochastic integrals, we obtain
\begin{equation}
\frac{{\rm d}}{{\rm d}t}\mathbb E \left[{\mathcal{E}_1}(u_n(t))\right]
+2\tfrac{1+\sigma}{1+2\sigma}\gamma \mathbb E\left[{\mathcal{E}_1}(u_n(t))\right]
\le C \|P_n\Phi \|^{2+\frac{4\beta\sigma}{2\beta-\sigma d}}_{L_{HS}(Y,H)}  \gamma^{-\frac{2\beta\sigma}{2\beta-\sigma d}} +\frac 12
 \| P_n\Phi\|^2_{L_{HS}(Y,\EA)}.
\end{equation}
By means of the Gronwall lemma,   we get
  \[
\mathbb E \left[{\mathcal{E}_1}(u_n(t))\right]
\le
e^{-2\frac{1+\sigma}{1+2\sigma}\gamma t} \mathbb{E}\left[{\mathcal{E}_1}(u_n^0)\right]
+C\left(  \|P_n\Phi \|^{2+\frac{4\beta\sigma}{2\beta-\sigma d}}_{L_{HS}(Y,H)}  \gamma^{-\frac{2\beta\sigma}{2\beta-\sigma d}} +
 \| P_n\Phi\|^2_{L_{HS}(Y,\EA)}\right) 
   \gamma^{-1}.
   \]
From the above estimate we infer \eqref{energia-modif-p-n} for $p=1$. 

For $p\ge 2$, we have by It\^o formula 
\begin{equation}\label{Ito-tildeH-alla-m}
{\rm d} {\mathcal{E}_1}(u_n(t))^p\le p {\mathcal{E}_1}(u_n(t))^{p-1} {\rm d} {\mathcal{E}_1}(u_n(t))
+\frac{p(p-1)}2 
{\mathcal{E}_1}(u_n(t))^{p-2}
2r(u_n(t))\ {\rm d}t,
\end{equation}
where we  estimate the quadratic variation of the stochastic integral in 
\eqref{stima-d-tildeH} so to get 
\begin{multline*}
  r(u_n(t))
  \le
\sum_{j \in \mathbb{Z}_0}[\Real \langle A u_n(t)+|u_n(t)|^{2\sigma}u_n(t),P_n \Phi e^Y_j\rangle ]^2 
  \\  +  4\left[G^2(1+\tfrac{2\beta\sigma}{2\beta-\sigma d})^2
    \mathcal M(u_n(t))^{\frac{4\beta\sigma}{2\beta-\sigma d}}+1\right]\sum_{j \in \mathbb{Z}_0}[\text{Re}\langle u_n(t),P_n \Phi e^Y_j \rangle]^2 .
  \end{multline*}
Proceeding as in \eqref{stimaAphi} and  \eqref{stimaFphi},  we get 
\[\begin{split}
 r(u_n(t))\lesssim & \|A^{\frac 12} u_n(t)\|_H^2 \|P_n\Phi\|^2_{L_{HS}(Y,\EA)}
+ \|u_n(t)\|^{2(2\sigma+1)}_{{\mathbb L}^{2\sigma+2}} \|P_n\Phi\|_{L_{HS}(Y,\EA)}^2
\\
&
+ 4\left[ G^2(1+\tfrac{2\beta\sigma}{2\beta-\sigma d})^2
    \mathcal M(u_n(t))^{1+\frac{4\beta\sigma}{2\beta-\sigma d}}  +\mathcal M(u_n(t))\right]
     \|P_n\Phi\|_{L_{HS}(Y,H)}^2.
\end{split}\]
Now to estimate the first term in the r.h.s. we use \eqref{nabla-Htilde}, so
$\|A^{\frac 12} u\|_H^2\le \|u\|_\EA^2 \le 4 \frac{1+\sigma}{1+2\sigma} {\mathcal{E}_1}(u)$. For the second term,
 by means of \eqref{GN-somma} we get
\[\begin{split}
 \|u_n\|^{2(2\sigma+1)}_{{\mathbb L}^{2\sigma+2}} 
 &\le
 \epsilon  \|u_n\|_\EA^{2\frac{2\sigma+1}{\sigma+1}}
  +C_{\epsilon,\sigma}\mathcal M(u_n)^{\frac{2\sigma+1}{\sigma+1}(1+\frac{2\beta\sigma}{2\beta-\sigma d})}        \\
 &\le  \epsilon C_\sigma
        {\mathcal{E}_1}(u_n)^{\frac{2\sigma+1}{\sigma+1}}+C_{\epsilon,\sigma} \mathcal M(u_n)^{\frac{2\sigma+1}{\sigma+1}(1+\frac{2\beta\sigma}{2\beta-\sigma d})}
\end{split}\]
for any $\epsilon>0$. 
Thus we estimate the latter term in \eqref{Ito-tildeH-alla-m} as follows
\[\begin{split}
{\mathcal{E}_1}(u_n)^{p-2} r(u_n)
&\lesssim_{\beta,\sigma,d}
{\mathcal{E}_1}(u_n)^{p-1} \|P_n\Phi\|^2_{L_{HS}(Y,\EA)}
+
{\mathcal{E}_1}(u_n)^{p-\frac1{\sigma+1}} \|P_n\Phi\|^2_{L_{HS}(Y,\EA)}
\\&\;+
{\mathcal{E}_1}(u_n)^{p-2} \mathcal M(u_n)^{\frac{2\sigma+1}{\sigma+1}(1+\frac{2\beta\sigma}{2\beta-\sigma d})}
\|P_n\Phi\|^2_{L_{HS}(Y,\EA)}
\\&\;+
{\mathcal{E}_1}(u_n)^{p-2} [ \mathcal M(u_n)^{1+\frac{4\beta\sigma}{2\beta-\sigma d}} 
+\mathcal M(u_n)]  \|P_n\Phi\|_{L_{HS}(Y,H)}^2
\\
\intertext{   and by Young inequality, used three times with the three different powers $\mathcal E_1$,}
&\le \gamma \epsilon {\mathcal{E}_1}(u_n)^p
  + C \|P_n\Phi\|^{2p}_{L_{HS}(Y,\EA)} \gamma^{1-p}
  + C  \|P_n\Phi\|^{2p(1+\sigma)}_{L_{HS}(Y,\EA)} \gamma^{1-p(1+\sigma)}
\\
& \;+C  [ \mathcal M(u_n)^{\frac{2\sigma+1}{\sigma+1}(1+\frac{2\beta\sigma}{2\beta-\sigma d})}
  +\mathcal M(u_n)^{1+\frac{4\beta\sigma}{2\beta-\sigma d}} +\mathcal M(u_n)]^{\frac p2}
 \|P_n\Phi\|^p_{L_{HS}(Y,H)} \gamma^{1-\frac p2}.
 \end{split}\]
In \eqref{Ito-tildeH-alla-m} we insert this estimate and the previous estimate
\eqref{stima-tildeH-con-martingala} for $d {\mathcal{E}_1}(u_n(t))$,  integrate in time, take the 
mathematical expectation to get rid of the stochastic integrals; 
hence using again the Young inequality for $\epsilon$ small enough
 we obtain
\[\begin{split}
\frac d{dt} \mathbb E &[{\mathcal E}_1(u_n(t))^p]
+p\tfrac{1+\sigma}{1+2\sigma}\gamma 
\mathbb E [{\mathcal E}_1(u_n(t))^p]
\\&\le
C  \left(  \mathbb E\mathcal [M(u_n(t))^{\frac{2\sigma+1}{\sigma+1}(1+\frac{2\beta\sigma}{2\beta-\sigma d})}]
  + \mathbb E[\mathcal M(u_n(t))^{1+\frac{4\beta\sigma}{2\beta-\sigma d}} ]
  + \mathbb E\mathcal M(u_n(t))\right)^{\frac p2}       \|P_n\Phi\|^p_{L_{HS}(Y,H)} \gamma^{1-\frac p2}
\\&\; 
+C \|P_n\Phi\|^{2p}_{L_{HS}(Y,\EA)} \gamma^{1-p}
  + C  \|P_n\Phi\|^{2p(1+\sigma)}_{L_{HS}(Y,\EA)} \gamma^{1-p(1+\sigma)}
%\\&
%+ C \gamma^{1-\frac p2}\left(\|\Phi\|_{L_{HS}(Y,\EA)}^p 
%\mathbb E[ \mathcal M(u(t))^{\frac p2 \frac {2\sigma+1}{\sigma+1}(1+\frac{2\sigma}{2-\sigma d})} ]
%+\|\Phi\|_{L_{HS}(Y,H)}^p  \mathbb E[ \mathcal M(u(t))^{\frac p2 (1+\frac{4\sigma}{2-\sigma d})} ] \right)
\\&\; 
+C \left(  \|P_n\Phi\|^{2}_{L_{HS}(Y,\EA)}
   +\|P_n\Phi\|^{2+\frac{4\beta\sigma}{2\beta-\sigma d}}_{L_{HS}(Y,H)} \gamma^{-\frac{2\beta\sigma}{2\beta-\sigma d}}\right)^p\gamma^{1-p}
%+\|\Phi\|^{2p(\sigma+1)}_{L_{HS}(Y,\EA)} \gamma^{1-p(\sigma+1)} .
%+C \|\Phi\|^{2p\frac{2\sigma+2}{2-\sigma d}}_{L_{HS}(Y,\EA)} \gamma^{-p\sigma+1} 
\end{split}\]
Since $1<\frac{2\sigma+1}{\sigma+1}(1+\frac{2\beta\sigma}{2\beta-\sigma d})<1+\frac{4\beta\sigma}{2\beta-\sigma d}$, the three-masses's  sum can be estimated by 
\[
C\left( \mathbb E[\mathcal M(u_n(t))^{1+\frac{4\beta\sigma}{2\beta-\sigma d}} ]
  + \mathbb E\mathcal M(u_n(t))\right)^{\frac p2} 
\]
and now we bound them by means of \eqref{p_est_mass}. We obtain an inequality which, thanks to Gronwall lemma and after computing the time integrals appearing there, 
 with  some elementary calculations  gives 
\[\begin{split}
\mathbb E [{\mathcal E}_1(u_n(t))^p]
\le &
e^{-p\frac{1+\sigma}{1+2\sigma} \gamma t} \left[{\mathcal E}_1( u_n ^{0})^p
+C \mathcal M( u_n^{0})^p \right]
+C \phi_{1}^p \gamma^{-p}
\\&+C
e^{-p a \gamma t}  
\mathcal M(u_n^{0})^{p(\frac12+\frac{2\beta\sigma}{2\beta-\sigma d})}  
\|P_n\Phi\|^{p}_{L_{HS}(Y,V)}\gamma^{-\frac{p}2}
\end{split}\]
for any $t\ge 0$,  where
\[
a=\min\left(   \tfrac{1+\sigma}{1+2\sigma}, \tfrac 12 +\tfrac{2\beta\sigma}{2\beta-\sigma d} \right).
\]
Recalling \eqref{bound_H_Pn} this gives \eqref{energia-modif-p-n}. \end{proof}

\begin{remark}
Notice that any deterministic initial data $u^{0}\in \EA$ fulfills the assumptions of Propositions \ref{prop-Galerkin-esistenza_bis}-\ref{bound_lemma_energy} for all $p\ge 1$. 
Hence for deterministic initial data the solutions to \eqref{eq-Gal}
enjoy the properties listed in Propositions \ref{prop-Galerkin-esistenza_bis}-\ref{bound_lemma_energy}  for every  finite $p$.
\end{remark}

%%%%%%%%%%%%%%%%%%%%%%%%%%%%%%%%%%%%
\section{Compactness  criteria}
\label{section-compactness}

We present a compactness result, used in the previous sections for the  convergence.

For $0<a<1$ we denote by $C^a([0,T];\EAdual)$  the space of $a$-H\"older continuous functions from the interval $[0,T]$ to  the space $\EAdual$, equipped with the norm
\begin{equation}\label{defHa}
\|u\|_{C^a([0,T];\EAdual)}= \sup_{0\le t\le T} \|u(t)\|_\EAdual +
 \sup_{0\le s<t\le T}\frac{\|u(t)-u(s)\|_\EAdual}{|t-s|^a}.
 \end{equation}
We define the locally convex space
\[
\mathcal Z_T=C([0,T];\EAdual) \cap L^{2+2\sigma}(0,T;{\mathbb L}^{2+2\sigma}) \cap  C_w([0,T];\EA)
\]
with the topology given by the supremum of the corresponding topologies for the spaces in the right-hand side. 
 We recall the compact embedding $\EA\subset {\mathbb L}^{2+2\sigma}$ and the continuous embeddings given in \eqref{embedd-V-L}.
%PROP
\begin{proposition}\label{Dub} 
Let $T>0$ be any finite time and $0<a<1$.
Then the  embedding
\begin{equation}\label{compact-embedding2}
L^{\infty}(0,T;\EA)\cap C^a([0,T];\EAdual)
\subset 
\mathcal Z_T
\end{equation}
is compact.
\end{proposition}
% proof
\begin{proof}  
We recall \cite[Proposition 4.2]{BHW-2019} which involves 
equicontinuous functions\footnote{A set $S$  of functions is equicontinuous in $C([0,T];\EAdual) $ when 
		\begin{align*}
		\lim_{\delta \to 0} \sup_{u\in S} \sup_{\vert t-s\vert\le \delta} \Vert u(t)-u(s)\Vert_{\EAdual}=0.
		\end{align*}}. It states that a bounded set $K\subseteq L^{\infty}(0,T;\EA)$  consisting of functions equicontinuous in $C([0,T];\EAdual)$
 is relatively compact in $C_w([0,T];\EA) \cap L^{2+2\sigma}(0,T;   {\mathbb L}^{2+2\sigma})\cap C([0,T];\EAdual)$. 
\\
Now consider a set $K$ such that
\[
\sup_{u\in K}\left( \|u\|_{L^\infty(0,T;\EA)}+\|u\|_{C^a([0,T];\EAdual)}\right)\le M
\]
for some positive $M$.
Then, we have
\[
\|u(t)-u(s)\|_{\EAdual}
\le
|t-s|^a \|u\|_{C^a([0,T];\EAdual)}
\le
M |t-s|^a.
\]
This implies the equicontinuity in $C([0,T];\EAdual)$.  We thus conclude by applying \cite[Proposition 4.2]{BHW-2019}. 
\end{proof}

%%%%%%%%%%%%%%%%
\section{Geometry of the domain}
\label{domini}
The proof of our main results, Theorems \ref{teoUstaz}, \ref{mass_density} and \ref{hatHdimle2}, relies on the following basic facts:
\begin{enumerate}[label=\textbf{(\arabic*)}] \itemsep 0.3em
    \item \label{fact1} the operator $A$ is a self-adjoint non-negative operator with a compact resolvent in the space $H$,
    \item \label{fact2} the embedding $V:=\mathcal{D}(A^\frac12) \subset {\mathbb L}^{2+2\sigma}$ is compact,
    \item \label{fact3} we obtain apriori bounds for the solution by exploiting the Hamiltonian structure of the equation, that is working with the mass and the energy functionals.
\end{enumerate}
Throughout the paper we considered, as an example of the operator $A$, the $\beta$-fractional power of the negative Laplacian, on an open bounded subset $D$ of $\mathbb R^d$ with $C^\infty$-boundary\begin{footnote}{
We work under these assumptions on the domain for simplicity. Generalizations including less regular domains $D$ can be found in \cite{Grubb}.
}\end{footnote}, with homogeneous Dirichlet boundary conditions. Assumption \ref{F} ensures that property \ref{fact2} above holds true.

Other possible examples for the operator $A$ are the following:
\begin{itemize}
\item the $\beta$-fractional power of the negative Laplacian, on an open bounded subset $D$ of $\mathbb R^d$ with $C^\infty$-boundary, with Neumann boundary conditions, that is $A:=(-\Delta_{\text{Neu}})^\beta$, % and $V=\mathcal{D}(A^{\frac 12})$,
\item the $\beta$-fractional power of the negative Laplace–Beltrami operator on a compact $d$-dimensional Riemannian
manifold $(M,g)$ without boundary, equipped with a Lipschitz metric $g$, that is $A:= (-\Delta_g)^\beta$.
\end{itemize}
These operators satisfy property \ref{fact1} and property \ref{fact2} when Assumption \ref{F} is in force; we refer to \cite[Sections 2 and 3]{BHW-2019} for a proof.

Therefore, although all the results of the paper are stated considering Dirichlet boundary conditions, the same results continue to hold in the case of homogeneous Neumann boundary conditions. The case of a $d$-dimensional compact Riemannian manifold without boundary is another possible setting. We emphasize that we never used the Poincar\'e inequality, which is available in $\mathcal D((-\Delta_{\text{Dir}})^{\frac b 2})$ for $b>\frac 12$, and is  in fact  largely used in the analysis of  the  NLS equation with $b=1$. This choice has been made 
precisely with the purpose  to generalize the results to the other geometries mentioned above, where 
the Poincar\'e inequality does not hold. We also emphasize that the proof of the results are based on apriori estimates obtained by working with mass and energy functionals. These estimates do not depend on the geometry of the domain, in contrast to the Strichartz estimates frequently used in the literature, which instead strongly depend on the geometry of the domain and the boundary conditions considered.

\smallskip
To conclude, although it is never used in the proofs, for completeness, we give below a characterization of the domain of the operator $A$ , when $A=(-\Delta_{\text{Dir}})^\beta$ and $A=(-\Delta_{\text{Neu}})^\beta$. % and $A=(-\Delta_g)^\beta$. 
Further details can be found in \cite{Grubb}.
\\
Let $-\Delta_B$ be the realization of $-\Delta$ in $H$ with domain $\{u \in H^2(D)  :  Bu=0\}$; here $Bu$ stands for either the Dirichet boundary condition or the Neumann boundary condition. In details, 
\[
Bu=B_ju, \qquad \text{for} \quad j=0 \quad \text{or} \quad j=1,
\]
where $B_0=\gamma_{|\partial D}$ and $B_1=\gamma_{|\partial D}\partial_\nu$. Here by $\gamma_{|\partial D}$ we denote the trace operator and by $\nu$ the outward normal unit vector to $\partial D$.
In \cite[Theorem 2.2]{Grubb}, for any real $b>0$, the author provides the following characterization 
\small{
\[
\mathcal D((-\Delta_B)^{\frac b 2})
=\begin{cases}   
  H^b(D), &0\le b<j+\frac 12,
%  \\\{u \in H^\beta(D): d^{-\frac 12}B u\in H  \} ,& \beta=\frac 12
% \\ \{u \in H^\beta(D): Bu=0 \}, & \frac12<\beta<\frac52,
  \\
 \{u \in H^b(D): Bu=B(-\Delta_B) u =\dots= B (-\Delta_B)^k u=0 \}, & \frac 12+j+2k<b<\frac 52+j+2k,
 \\
  \{u \in H^b(D): %d^{-\frac 12}B u, {\color{red}\ldots} \in H  \}
  B(-\Delta_B)^l u=0\ \text{for} \ l<k, B(-\Delta_B)^ku \in \bar{H}^{\frac 12}(\overline D)\},& b=\frac 12+j+2k,
 \end{cases}
\]
}
where $k \in \mathbb{N}_0$.
Here $\bar{H}^\frac12(\overline D)$ stands for the space of functions in $H^\frac12(\mathbb{R}^n)$ with support in $\overline{D}$.

%Theorem \ref{dom_D_N} and \cite[Theorem 4.3.2 (1)]{Triebel} yield, in particular,\begin{equation}\label{realization_sD}\mathcal{D}(A_D^{\frac s2})=\begin{cases}H^{s} \ \text{for} \ s \in \left(0,\frac 12\right),\\H^{s}_0\ \text{for} \ s \in \left(\frac 12, 1\right].\end{cases}\end{equation}

%RICHIAMI A VARI RISULTATI.
%Set $H_0^b(D)$ be the closure of $C^\infty_0(D)$ in $H^b(D)$ for $b>0$.\\
%Theorem 11.1 by Lions - Magenes \cite{LionsMag}\\Let $D$ be a bounded smooth domain. 
%The space $C^\infty_0(D)$ is dense in $H^b(D)$  if and only if
%\[b\le \frac 12\]
%(then $H^b_0(D)=H^b(D)$). If $b>\frac 12$, we have $H^b_0(D)$ is striclty contained in $H^b(D)$.\bigskipTheorem 11.5 by Lions - Magenes \cite{LionsMag}\\
%Let $D$ be a bounded smooth domain. Let $b>\frac 12$. The the following two conditions are equivalent
%\begin{enumerate}\item $u \in H^b_0(D)$\item  $u  \in H^b(D): \frac{\partial^j u}{\partial \nu^j}=0 \text{ for } 0\le j<b-\frac 12.$\end{enumerate}where $\frac{\partial^j u}{\partial \nu^j}$ denotes the normal $j$-order derivative on $\partial D$, oriented toward the interior of $D$.\\Riguarda Th 11.7 , Th 11.6\\e TH 11.2 e 11.3per il caso $\frac 12< b\le 1$ e $0<b<\frac 12$.\\The dual space:by definition we set $H^{-b}=\left( H^b_0(D)\right)^\prime$ for $b>0$.\\\bigskip$\EA$ the completition of $C^{\infty}_0(D)$  in $H^{\beta}(D)$, see \cite[Definition 4.2.1(2)]{Triebel}.

%%%%%%%%%%%%%%%%%%%%%%%%%%%%%%%%%%%%
\section{Regularity properties of the law of the mass functional}
\label{section-dense}

The proof of  Theorem \ref{mass_density} is a minimal modification of that in  \cite{Sh11}. 
For reader's convenience we write it, also because 
our noise  \eqref{noise-nell-assumption}
is a bit more general than that of Shirikyan, who assumes $\varphi_j=\varphi_{-j}$ for any $j$; however we allow to have a real noise or a pure imaginary one to get some partial results (e.g., to obtain that the random variable $U_\gamma(0)$ has no atom in $0$).

The results  of this section exploit some properties of local time of  a process  associated to the perturbed and damped equation 
\eqref{eqS-sqrt-gamma}. 
Then for vanishing $\gamma$ the limit  enjoys the same property.

We start with a general result on  local time. This is   Theorem 2.2 in \cite{Sh11}, which 
 provides a useful relation \eqref{proprieta-local-time} involving the local time defined by the Tanaka-Meyer formula 
 \eqref{def-local-time}.
 
 \begin{theorem}
 Consider the real valued semimartingale
 \begin{equation}\label{semimartingale}
 y(t)=y(0)+\int_0^t b(s) {\rm d}s +\sum_{j \in{\mathbb Z}_0} \int_0^t \theta_j(s) {\rm d}\beta_j(s)
 \end{equation}
 where $\{\beta_j\}_j$ is a sequence of independent real Brownian motions defined on 
 $(\Omega,\mathcal F, \mathbb P)$ with a standard filtration, and the processes $b$ and $\theta_j$ 
 are adapted to this filtration  and such that
 \[
 \mathbb E \int_0^t\left( |b(s)|+ \sum_{j \in{\mathbb Z}_0} | \theta_j(s)|^2\right) {\rm d}s<\infty \quad \forall t>0.
 \]
 Then for any $a\in \mathbb R$, the local time $\Lambda_t(a)$ defined by
 \begin{equation}\label{def-local-time}
 \Lambda_t(a)=|y(t)-a| - |y(0)-a| - \sum_{j \in{\mathbb Z}_0}
 \int_0^t \pmb{1}_{(a,+\infty)}(y(s))\theta_j(s) {\rm d}\beta_j(s)
 - \int_0^t \pmb{1}_{(a,+\infty)}(y(s)) b(s) {\rm d}s,\qquad t>0
 \end{equation}
 possesses the following properties:
 \begin{enumerate}
 \item
 The mapping $(t,a,\omega)\mapsto \Lambda_t(a)(\omega)$ is measurable and 
 for any $a\in \mathbb R$ the process $\{\Lambda_t(a)\}_t$ is adapted, continuous and non-decreasing.
 Moreover, for every $t\ge 0$ and $\mathbb P$-a.s. the function 
 $a \mapsto \Lambda_t(a)(\omega)$ is right-continuous.
 \item
For any non-negative Borel-measurable function $h:\mathbb R\to \mathbb R$, we have $\mathbb P$-a.s.
  \begin{equation}\label{proprieta-local-time}
  \int_{-\infty}^{+\infty} h(a) \Lambda_t(a){\rm d}a=\frac 12 
  \sum_{j \in{\mathbb Z}_0} \int_0^t h(y(s))  | \theta_j(s)|^2  {\rm d}s,\qquad t>0.
  \end{equation}
  \end{enumerate}
 \end{theorem}

Using different functionals $y$  of the stationary solution $U_\gamma$  of the stochastic and damped NLS equation,
Shirikyan proves three results (see Theorem 3.1 in  \cite{Sh11}). 
Now we state them and then we show  that with minimal changes they can be obtained in our setting, i.e. when the stochastic and damped NLS equation is
\eqref{eqS-sqrt-gamma}. 
We denote by $\mu_\gamma$ the   marginal law of the stationary process  $\hat U_\gamma$
given in Proposition \ref{Prop-no-gamma} and by  $l$ the Lebesgue measure on the real line.

%% TEO
\begin{theorem}
Let Assumptions \ref{ass-sigma} and  \ref{ass_G} be in force and fix  $\gamma>0$.
Let $\mu_\gamma$ be  the law of $U_\gamma(t)$ (at fixed time), where $U_\gamma$ is any stationary martingale solution of equation \eqref{eqS-sqrt-gamma}  as given in Proposition \ref{Prop-no-gamma}.
\begin{enumerate}
\item
If $\varphi_j\neq 0$ for some $j \in \mathbb{Z}_0$, then $\mu_\gamma$ has no atoms in zero.
\item
If   $\varphi_{j_1}\varphi_{j_2}\neq 0$ for some $j_1\neq j_2$,
then there exists a positive constant $c_\Phi$ such that 
\begin{equation}\label{stima-su-palla}
\mu_\gamma\left(\|u\|_H \le \delta \right)  \le\frac {1+ \|\Phi\|_{L_{HS}(Y,H)}^2}{c_\Phi} \ \delta \qquad \forall \delta \ge 0.
\end{equation}
\item  
if $\varphi_j \ne 0$ for all $j \in \mathbb{Z}_0$, 
then there exists a continuous increasing function $p(r)$ going to zero with $r$ such that 
\begin{equation*}
\mu_\gamma\left(\mathcal{M}(u) \in \Gamma \right) \le p(l(\Gamma))
\end{equation*}
for any Borel subset $\Gamma$ of $\mathbb{R}$.
\end{enumerate}
\end{theorem}
\begin{proof}
Following \cite{Sh11}, the strategy consists in using the identities \eqref{def-local-time} and \eqref{proprieta-local-time} for  stationary processes $y, b, \theta_j$ and  choosing $h=1_\Gamma$ so that taking the mathematical expectation in  \eqref{proprieta-local-time} we get 
\[
\int_\Gamma \mathbb E  \Lambda_t(a){\rm d}a=
\mathbb E  \int_{-\infty}^{+\infty} h(a) \Lambda_t(a){\rm d}a= 
\frac 12  t\mathbb E \left( 1_\Gamma(y(0)) \sum_{j \in{\mathbb Z}_0}  | \theta_j(0) |^2  \right)
\]
and from \eqref{def-local-time} we get
\[
\mathbb E \Lambda_t(a)= - \mathbb E\int_0^t \pmb{1}_{(a,+\infty)}(y(s)) b(s) {\rm d}s
   =-t \mathbb E\left( \pmb{1}_{(a,+\infty)}(y(0)) b(0) \right).
\]
Hence, integrating the latter identity on $\Gamma$ so to get 
$\int_\Gamma \mathbb E  \Lambda_t(a){\rm d}a$ in both left-hand sides, 
we obtain
\begin{equation}\label{identita-L-T}
\frac 12  \mathbb E \big[\pmb{1}_\Gamma(y(0))  \sum_{j \in{\mathbb Z}_0}   | \theta_j(0) |^2  \big]
=
-\int_\Gamma \mathbb E\big[ \pmb{1}_{(a,+\infty)}(y(0)) b(0) \big]{\rm d}a.
\end{equation}
{\bf Proof of 1.}
\\
Let $\varphi_j\neq 0$ and choose an element $v\in \EA$  such that $\text{Re} (e_j,v)_H\neq 0$ if $j>0$ or
 $\text{Im} (e_{-j},v)_H\neq 0$ if $j<0$.
 Set $y(t)=\text{Re} (U_\gamma(t),v)_H$.
The differential equation fulfilled by $y$ is of the form \eqref{semimartingale} with
\[
b(t)= \text{Re} \langle \im AU_\gamma(t)-\gamma U_\gamma(t)-i \alpha F(U_\gamma(t)), v \rangle
\]
and $\theta_j $, independent of time, given by
\[
\begin{split}
\theta_j=\sqrt \gamma \varphi_j Re (e_j,v)_H\qquad & \text{ for } j>0
\\
\theta_j=-\sqrt \gamma \varphi_j Im (e_{-j},v)_H\;& \text{ for } j<0
\end{split}
\]
so
\[
 \sum_{j \in{\mathbb Z}_0}   | \theta_j |^2 
 =\gamma\sum_{j=1}^\infty \left(\varphi_j^2[ \text{Re}  (e_j,v)_H]^2 + \varphi_{-j}^2[ \text{Im}  (e_j,v)_H]^2\right)>0.
\]
Therefore from \eqref{identita-L-T} we obtain
\[\begin{split}
\mathbb P(y(0)\in \Gamma)
&=-\frac 2 {\gamma\sum_{j=1}^\infty \left(\varphi_j^2[ \text{Re}  (e_j,v)_H]^2 + \varphi_{-j}^2[ \text{Im}  (e_j,v)_H]^2\right)}
  \int_\Gamma \mathbb E\big[ \pmb{1}_{(a,+\infty)}(y(0)) b(0) \big]{\rm d}a
\\
&\le \frac 2 {\gamma\sum_{j=1}^\infty \left(\varphi_j^2[ \text{Re}  (e_j,v)_H]^2 + \varphi_{-j}^2[ \text{Im}  (e_j,v)_H]^2\right)} \mathbb E(|b(0)|) l(\Gamma).
\end{split}
\]
Now,  exploiting estimates \eqref{tildeH-domina},\eqref{stimaF},  \eqref{finaleMp} and \eqref{finale-energy}, we deal with the term depending on $b$
\[\begin{split}
\mathbb{E}_\gamma \left[|b(0)|\right]
&=\mathbb{E}_\gamma\left[|\text{Re} \langle \im A^{\frac 12}U_\gamma(0), A^{\frac 12}v\rangle-\text{Re}\langle \gamma U_\gamma(0), v \rangle-\text{Re}\langle i \alpha F(U_\gamma(0)), v \rangle|\right]
\\
&\le \mathbb{E}_\gamma\left[\|U_\gamma(0)\|_V\|v\|_V+\gamma \|U_\gamma(0)\|_H\|v\|_H+\|F(U_\gamma(0)) \|_{L^{\frac{2\sigma+2}{2\sigma+1}}} \|v\|_{L^{2\sigma+2}}\right]
\\
&\lesssim \|v\|_\EA \left( \mathbb{E}_\gamma\left[\|U_\gamma(0)\|_V\right]
+\gamma\mathbb{E}_\gamma\left[ \|U_\gamma(0)\|_H\right] + \mathbb{E}_\gamma\left[\|U_\gamma(0)\|_{L^{2\sigma+2}}^{2\sigma+1}\right]\right)
\end{split}\]
which is bounded by $ \|v\|_\EA(C_1+\gamma C_2)$ according to Proposition \ref{Prop-no-gamma}.
Therefore
\[
\mathbb P(\text{Re} (U_\gamma(0),v)_H \in \Gamma)\le
\frac  { \|v\|_\EA(C_1+\gamma C_2)} {\gamma\sum_{j=1}^\infty \left(\varphi_j^2[ \text{Re}  (e_j,v)_H]^2 + \varphi_{-j}^2[ \text{Im}  (e_j,v)_H]^2\right)} l(\Gamma)
\]
and we conclude that the law of the  random variable
$\text{Re} (U_\gamma(t),v)_H$
 is absolutely continuous with respect to the Lebesgue measure. 
% {\color{blue}Similarly one argues for $y(t)=\text{Im} (U_\gamma(t),v)_H$.}
 In particular  $\mu_\gamma$  has no atom in $0$; indeed, if $\varphi_j\neq 0$ for $j>0$, by choosing $v=e_j$ 
 we obtain that  
 $\mathbb P_\gamma(U_\gamma(0)=0)\le \mathbb P_\gamma( (\text{Re }U_\gamma(t),e_j)_H=0)=0$ (and similarly we proceed if $\varphi_j\neq 0$ for $j<0$ by choosing $v=\im e_{-j}$ getting the result for $(\text{Im }U_\gamma(t),e_j)_H$).
  \\
{\bf Proof of 2.}\\
We deal with $y(t)=\|U_\gamma(t)\|_H$; 
therefore  equation \eqref{semimartingale}, under the condition that $y$ does not vanish,  is now
\begin{multline}
 y(t)=y(0)+\int_0^t[-\gamma  \|U_\gamma(s)\|_H+\gamma\frac{\|\Phi\|_{L_{HS}(Y,H)}^2}{2 \|U_\gamma(s)\|_H}] {\rm d}s 
 +\sqrt \gamma\int_0^t\frac{Re \langle  U_\gamma(s) ,\Phi dW(s)\rangle}{ \|U_\gamma(s)\|_H}
\\ 
-\frac \gamma{2}   \sum_{j=1}^\infty \int_0^t   \frac{
[ (\text{Re } U_\gamma(s),e_j)_H^2\varphi_j ^2+  (\text{Im } U_\gamma(s),e_j)_H^2\varphi_{-j} ^2]}{ \|U_\gamma(s)\|_H^3}{\rm d}s
\end{multline}
Choosing $\Gamma=(\alpha,\beta)$ with $0<\alpha<\beta$, relation \eqref{identita-L-T} becomes
\begin{multline}
- {\mathbb E}_\gamma \int_\alpha^\beta \pmb{1}_{(a,+\infty)}(\|U_\gamma(0)\|_H) 
 \Big( - \|U_\gamma(0)\|_H+\frac{\|\Phi\|_{L_{HS}(Y,H)}^2}{2 \|U_\gamma(0)\|_H} 
\\ - \frac 1{2 \|U_\gamma(0)\|_H^3}  \sum_{j=1}^\infty  [ (\text{Re } U_\gamma(0),e_j)_H^2\varphi_j ^2+  (\text{Im } U_\gamma(0),e_j)_H^2\varphi_{-j} ^2] \Big)
 {\rm d}a
\\=
\frac 12  \mathbb E \left[\pmb{1}_{(\alpha,\beta)}(\|U_\gamma(0)\|_H)
   \sum_{j=1}^\infty \frac{ (\text{Re } U_\gamma(0),e_j)_H^2\varphi_j ^2+  (\text{Im } U_\gamma(0),e_j)_H^2\varphi_{-j} ^2}{ \|U_\gamma(0)\|^2_H}\right]
\end{multline}
Hence, neglecting  the r.h.s. which is not negative, we get
\begin{multline*}
\mathbb E_\gamma \int_\alpha^\beta \frac{\pmb{1}_{(a,+\infty)}(\|U_\gamma(0)\|_H) }{2\|U_\gamma(0)\|_H^3}
 \left(\|\Phi\|_{L_{HS}(Y,H)}^2\|U_\gamma(0)\|_H^2- \sum_{j=1}^\infty   [ (\text{Re } U_\gamma(0),e_j)_H^2\varphi_j ^2+  (\text{Im } U_\gamma(0),e_j)_H^2\varphi_{-j} ^2] \right){\rm d}a
\\
 \le 
 (\beta-\alpha)\mathbb E_\gamma \|U_\gamma(0)\|_H .
\end{multline*}
We estimate the r.h.s. from above by using  \eqref{uniform-mass}
\[
\mathbb E_\gamma \|U_\gamma(0)\|_H\le \frac 12 +\frac 12 \mathbb E \|U_\gamma(0)\|_H^2=\frac 12 
+\frac 12 \|\Phi\|_{L_{HS}(Y,H)}^2,
\]
and the  l.h.s. from below by using
the assumption that $\varphi_{j_1}\ne 0 $ and $\varphi_{j_2}\ne 0$.
Indeed, recalling that $\|\Phi\|_{L_{HS}(Y,H)}^2=\sum_{j=1}^\infty [\varphi_j ^2+\varphi_{-j} ^2]$ and 
$\|U\|_H^2=\sum_{j=1}^\infty [(\text{Re }U,e_j)_H^2+(\text{Im }U,e_j)_H^2]$, we get
\begin{equation}\label{stima-Phi_U}\begin{split}
\|\Phi\|_{L_{HS}(Y,H)}^2 & \|U_\gamma(0)\|_H^2- \sum_{j=1}^\infty 
[ (\text{Re } U_\gamma(0),e_j)_H^2\varphi_j ^2+  (\text{Im } U_\gamma(0),e_j)_H^2\varphi_{-j} ^2]
\\ &=
(\sum_{j=1}^\infty \varphi_{j} ^2 ) (\sum_{j=1}^\infty  (\text{Im } U_\gamma(0),e_j)_H^2 )+
(\sum_{j=1}^\infty \varphi_{-j} ^2 ) (\sum_{j=1}^\infty  (\text{Re } U_\gamma(0),e_j)_H^2 )
\\&\quad 
+\left[(\sum_{j=1}^\infty \varphi_{j} ^2 ) (\sum_{j=1}^\infty  (\text{Re } U_\gamma(0),e_j)_H^2 ) 
-\sum_{j=1}^\infty 
 (\text{Re } U_\gamma(0),e_j)_H^2\varphi_j ^2\right]
 \\&\quad 
 +\left[(\sum_{j=1}^\infty \varphi_{-j} ^2 ) (\sum_{j=1}^\infty  (\text{Im } U_\gamma(0),e_j)_H^2 ) 
-\sum_{j=1}^\infty 
 (\text{Im } U_\gamma(0),e_j)_H^2\varphi_{-j}^2\right]
\end{split}\end{equation}

Both terms in square brackets are not negative. Moreover, 
 if $j_1>0$ and $j_2>0$ we get
\[\begin{split}
(\sum_{j=1}^\infty \varphi_{j} ^2 ) (\sum_{j=1}^\infty  (\text{Re } U_\gamma(0),e_j)_H^2 ) 
-\sum_{j=1}^\infty 
 (\text{Re } U_\gamma(0),e_j)_H^2\varphi_j ^2
&=\sum_{j=1}^\infty 
 (\text{Re } U_\gamma(0),e_j)_H^2(\sum_{h\neq j} \varphi_h ^2)
\\
&\ge( \min (\varphi^2_{j_1},\varphi^2_{j_2})) \sum_{j=1}^\infty 
 (\text{Re } U_\gamma(0),e_j)_H^2
\end{split}\]
Hence, neglecting the second and the forth addend in  the r.h.s of \eqref{stima-Phi_U} we get
\[\begin{split}
\|\Phi\|_{L_{HS}(Y,H)}^2 & \|U_\gamma(0)\|_H^2- \sum_{j=1}^\infty 
[ (\text{Re } U_\gamma(0),e_j)_H^2\varphi_j ^2+  (\text{Im } U_\gamma(0),e_j)_H^2\varphi_{-j} ^2]
\\ &\ge
( \min (\varphi^2_{j_1},\varphi^2_{j_2}))\|U_\gamma(0)\|_H^2.
\end{split}\]
Similarly if  $j_1<0$ and $j_2<0$.
However, if $j_1j_2<0$ we get the same estimate by neglecting the two latter addends in  the r.h.s of \eqref{stima-Phi_U} and estimating the two first ones in an elementary  way.

We set $c_\Phi= \min (\varphi^2_{j_1},\varphi^2_{j_2})>0$.
Summing up, we have obtained that
\[
\mathbb E_\gamma \int_\alpha^\beta \frac{1_{(a,+\infty)}(\|U_\gamma(0)\|_H) }{\|U_\gamma(0)\|_H}{\rm d}a
\le \frac{1+ \|\Phi\|_{L_{HS}(Y,H)}^2}{c_\Phi} (\beta-\alpha)
\]
 and letting  $\alpha$ vanish
\begin{equation} \label{stima-in-ii}
\mathbb{E}_\gamma\left[ \int_0^\beta \frac{\pmb{1}_{(a, + \infty)}(\|U_\gamma(0)\|_H)}{\|U_\gamma(0)\|_H}\, {\rm d}a\right] \le \frac{\left(1+ \|\Phi\|_{L_{HS}(Y,H)}^2\right)}{c_\Phi}  \beta.
\end{equation}
Now, for a fixed $\delta>0$   we have
\[
\int_0^\beta \mathbb{P}(a< \|U_\gamma\|_H\le \delta)\, {\rm d}a
=
\mathbb{E}_\gamma\left[ \int_0^\beta \pmb{1}_{(a, \delta]}(\|U_\gamma\|_H)\, {\rm d}a\right]
\le \delta
 \mathbb{E}_\gamma\left[ \int_0^\beta \frac{\pmb{1}_{(a, \delta]}(\|U_\gamma\|_H)}{\|U_\gamma\|_H}\, {\rm d}a\right]. 
\]
Estimating the r.h.s. by means of \eqref{stima-in-ii}
we get
\[
\frac 1\beta \int_0^\beta \mathbb{P}(a< \|U_\gamma\|_H\le \delta)\, {\rm d}a
\le \frac{1+ \|\Phi\|_{L_{HS}(Y,H)}^2}{c_\Phi} \ \delta .
\] 
Passing to the limit as $\beta \rightarrow 0^+$ and recalling from the previous point (1) that $\mu_\gamma$ has no atom  in $0$, we obtain \eqref{stima-su-palla}.
\\
{\bf Proof of 3.}
\\
We deal with $y(t)=\|U_\gamma(t)\|_H^2=\mathcal M(U_\gamma(t))$;
therefore  equation \eqref{semimartingale} is now
\begin{multline*}
{\rm d}y(t)=\gamma(- 2  y(t)+
\|\Phi \|^2_{L_{HS}(Y,H)})\, {\rm d}t
\\+ 2\sqrt \gamma\sum_j [ \varphi_j (\text{Re }U_\gamma(t), e_j)_H\,{\rm d}W_j(t)
+\varphi_{-j} (\text{Im }U_\gamma(t), e_j)_H\,{\rm d}W_{-j}(t)]
\end{multline*}
and relation \eqref{identita-L-T} becomes
\begin{multline}
 \frac 12   \mathbb{E}_\gamma \left[ \pmb{1}_\Gamma (\|U_\gamma(0)\|_H^2) 4 \sum_{j=1}^\infty 
    [ (\text{Re } U_\gamma(0),e_j)_H^2\varphi_j ^2+  (\text{Im } U_\gamma(0),e_j)_H^2\varphi_{-j} ^2]  \right] 
\\=-
\mathbb{E}_\gamma \left[\int_\Gamma\pmb{1}_{(a, +\infty)}(\|U_\gamma(0)\|_H^2)
  [- 2  \|U_\gamma(0)\|_H^2+
\|\Phi \|^2_{L_{HS}(Y,H)}]  {\rm d}a \right] .
\end{multline}
Hence 
\begin{equation}
\label{4}
\begin{split}
 \mathbb{E}_\gamma \left[ \pmb{1}_\Gamma (\|U_\gamma(0)\|_H^2) \sum_{j=1}^\infty 
      [ (\text{Re } U_\gamma(0),e_j)_H^2\varphi_j ^2+  (\text{Im } U_\gamma(0),e_j)_H^2\varphi_{-j} ^2] \right]
    &\le  l(\Gamma)\mathbb{E}_\gamma \left[\|U_\gamma(0)\|^2_H\right] 
\\&= \frac{l(\Gamma)\|\Phi\|^2_{L_{HS}(Y,H)}}{2}.
\end{split}\end{equation}

We estimate the l.h.s. of the above inequality from below. 
Recall from \eqref{productHb} that
\[
\|u\|_\EA^2= \|u\|_H^2+\|A^{\frac 12} u\|_H^2= \sum_{j=1}^\infty (1+\lambda_j^\beta) |(u,e_j)_H|^2.
\]
Hence
\begin{equation}\label{inverse-Poincare}
\sum_{j=1}^N |(u,e_j)_H|^2= \sum_{j=1}^\infty |(u,e_j)_H|^2 -  \sum_{j=N+1}^\infty|(u,e_j)_H|^2
\ge \|u\|_H^2 - \frac 1{1+\lambda_{N+1}^\beta} \|u\|_\EA^2.
\end{equation}

Therefore, 
given $\delta > 0$, if $\|U_\gamma\|_H \ge \delta$ and $\|U_\gamma\|_V\le \frac{1}{\sqrt \delta}$, 
by mean of \eqref{inverse-Poincare} we infer 
\begin{equation*}\begin{split}
 \sum_{j=1}^\infty &
      [ (\text{Re } U_\gamma(0),e_j)_H^2\varphi_j ^2+  (\text{Im } U_\gamma(0),e_j)_H^2\varphi_{-j} ^2]
\\&
\ge c_N^\Phi\sum_{k=1}^N|( U_\gamma, e_k)_H| ^2
%=b^2_{min}\left(\|U_\gamma\|^2_H-\sum_{ k > N}|\text{Re}\langle U_\gamma, e_k\rangle| ^2\right)
\ge c_N^\Phi \left(\|U_\gamma\|^2_H-\frac{1}{1+\lambda_{N+1}^\beta}\| U_\gamma \|^2_V\right) 
\ge c_N^\Phi\left(\delta^2-\frac{1}{1+\lambda_{N+1}^\beta}\frac 1 {\delta} \right),
\end{split}
\end{equation*}
where $c_N^\Phi:=\left(\min_{|k|\le N}\varphi_k^2 \right)$.  Notice that $c_N^\Phi\to 0 $ as $N\to+\infty$, since
from Assumption \ref{ass_G} we know that $(1+ \lambda_j^\beta) \varphi_j^2 \le \|\Phi\|^2_{L_{HS}(Y,\EA)}$ for any $j$ and therefore $c_N^\Phi \le  \varphi_N^2\le  \|\Phi\|_{L_{HS}(Y,\EA)}^2 \lambda_N^{-\beta}$ with $\lambda_N \to +\infty$ as $N \to +\infty$.

Choosing $N=N(\delta)$  such that $1+\lambda_{N(\delta)+1}^\beta> \frac 2{\delta^3}$ 
we find
\[
 \sum_{j=1}^\infty 
      [ (\text{Re } U_\gamma(0),e_j)_H^2\varphi_j ^2+  (\text{Im } U_\gamma(0),e_j)_H^2\varphi_{-j} ^2]
\ge c_{N(\delta)}^\Phi  \frac{\delta^2}2.
\]
Defining
$\eta(\delta)=\frac 12 c_{N(\delta)}^\Phi\delta^2$, we have that $\eta(\delta)$ goes to zero with $\delta$, and
when  $\|U_\gamma(0)\|_H \ge \delta$ and $ \|U_\gamma(0)\|_V \le \frac{1}{\sqrt \delta}$ we get
\begin{equation}
\label{5}
 \sum_{j=1}^\infty 
      [ (\text{Re } U_\gamma(0),e_j)_H^2\varphi_j ^2+  (\text{Im } U_\gamma(0),e_j)_H^2\varphi_{-j} ^2]
      \ge \eta(\delta).
\end{equation}

We now define the event $G_\delta:= \{\|U_\gamma\|_H \le \delta \ \text{or} \ \|U_\gamma\|_V \ge \frac{1}{\sqrt \delta}\}$. From the Chebychev inequality,  the previous point (2) and estimate \eqref{sti_tight-gamma}, we infer
\begin{align*}
\mathbb{P}(G_\delta)\le \mathbb{P}(\|U_\gamma\|_H \le \delta)+\mathbb{P}\left(\|U_\gamma\|_V\ge \frac{1}{\sqrt \delta}\right)
 \le \delta C_\phi+ \delta \mathbb{E}\left[\|U_\gamma\|^2_V\right]
 \\
  \lesssim \delta (C_\phi +1+ \phi_{1}(d,\beta, \sigma,1, \Phi)) =:\delta C.
\end{align*}
Combining the above estimate with \eqref{4} and \eqref{5}, by means of the Chebychev inequality we infer that, for any borelian set $\Gamma$ of $\mathbb{R}$,
\begin{align*}
\mathbb{P}(\mathcal{M}(U_\gamma) \in \Gamma)
&=\mathbb{P}(\{\mathcal{M}(U_\gamma) \in \Gamma\} \cap G_\delta) + \mathbb{P}(\{\mathcal{M}(U_\gamma) \in \Gamma\} \cap G_\delta^c)
\\
&\le C\delta +\frac{1}{\eta(\delta)}\mathbb{E}_\gamma \left[\pmb{1}_\Gamma(\mathcal{M}(U_\gamma))
 \sum_{j=1}^\infty 
      [ (\text{Re } U_\gamma(0),e_j)_H^2\varphi_j ^2+  (\text{Im } U_\gamma(0),e_j)_H^2\varphi_{-j} ^2]\right]
\\
&\le C\delta + \frac{l(\Gamma)\|\Phi\|^2_{L_{HS}(Y,H)}}{2 \eta(\delta)},
\end{align*}
and the thesis follows by setting 
\[
p(r)=\inf_{\delta>0} \left(  C\delta+\frac{\|\Phi\|^2_{L_{HS}(Y,H)}}{2 \eta(\delta)}  r\right).
\]
\end{proof}

\begin{remark}
The difference of our point (2)  with respect to that in \cite{Sh11} is that  our result  (2) holds true 
if  at least two coefficients do not vanish without any relation between $j_1$ and $j_2$, whereas Shirikyan \cite{Sh11} considered $j_2=-j_1$.
\end{remark}

%%%%%%%%%%%%%%%%%%%%%%%%%%%%%%%%%%%%
\subsection*{Acknowledgements}

The authors are members of Gruppo Nazionale per l’Analisi Matematica, la Probabilità e le loro Applicazioni (GNAMPA) of the Istituto Nazionale di Alta Matematica (INdAM), and of the Center CAMRisk at the Department of Economics and Management of the University of Pavia.
 Moreover they  gratefully acknowledge financial support through the projects CUP$-$E55F22000270001 %(Scarpa) 
and 
CUP$-$E53C22001930001. %(Zanella)

\subsection*{Declarations}
There is no conflict of interest. The authors have no relevant financial or non-financial interests to disclose.

%%%%%%%%%%%%%


\begin{thebibliography}{MMT10}


\bibitem[Bou96]{Bourgain96}
J. Bourgain.
   \newblock  Invariant measures for the 2D-defocusing nonlinear Schr\"odinger equation.
   \newblock {\it Comm. Math. Phys.} 176: 421-445, 1996.

\bibitem[BM19]{brezis}
H. Br\'ezis and  P. Minorescu.
    \newblock Where Sobolev interacts with Gagliardo-Nirenberg.
    \newblock {\it J. Funct. Anal.} 277: 2839--2864, 2019.

\bibitem[BFZ24]{noi2d}
      Z.~Brze{\'{z}}niak, B.~Ferrario and M.~Zanella.
        \newblock Invariant measures for a stochastic nonlinear and damped 2D Schr\"odinger equation.
        \newblock {\it Nonlinearity}  37: 1, 2024.

\bibitem[BFZ23]{large-damping}
      Z.~Brze{\'{z}}niak, B.~Ferrario and M.~Zanella.
        \newblock Ergodic results for the stochastic nonlinear Schr\"odinger equation with large damping.
         \newblock {\it J. Evol. Equ.} 23: 19, 2023.
          
\bibitem[BHW19]{BHW-2019}
       Z. Brze{\'{z}}niak, F. Hornung and L. Weis.
        \newblock Martingale solutions for the stochastic nonlinear Schr\"odinger equation in the energy space.
        \newblock {\it  Probab. Theory Related Fields}, 174(3-4):1273--1338, 2019.

\bibitem[BGT04]{Burq+G+T_2004}
       N.  Burq, P. G\'erard and  N. Tzvetkov.
       \newblock Strichartz inequalities and the nonlinear Schr\"odinger equation on compact manifolds.
       \newblock {\it Amer. J. Math.}, 126: 569--605, 2004.

\bibitem[CM21]{CM}
J.-B. Casteras and L. Monsaingeon.
        \newblock Invariant measures and global well-posedness for a fractional Schr\"odinger equation 
with Moser-Trudinger type nonlinearity.
        \newblock {\it Stoch. Partial Differ. Equ. Anal. Comput.}, 12: 416--465, 2021.
        
\bibitem[Caz03]{Cazenave}
     T. Cazenave.
     \newblock {\it Semilinear Schr\"odinger Equations}, volume~10.
      \newblock American Mathematical Soc., 2003.	

\bibitem[CK97]{ChowK}
P.-L. Chow and R.Z. Khasminskii.
        \newblock Stationary solutions of nonlinear stochastic evolution equations.
                \newblock {\it Stoch. Anal. Appl. }15(5):  671--699, 1997.


\bibitem[DPZ14]{dpz}
      G. Da Prato and J. Zabczyk.
       \newblock  {\it Stochastic Equations in Infinite Dimensions.}
       \newblock Encyclopedia of Mathematics and Its Applications. Cambridge University Press, Cambridge (2014).

\bibitem[DO05]{DebO}
      A. Debussche and C. Odasso
       \newblock Ergodicity for a weakly damped stochastic non-linear Schr\"odinger equation.
           \newblock {\it J. Evol. Equ.} 5(3):317--356, 2005.
           
\bibitem[EKZ17]{Ekren_2017}	% R^d sia focusing che defocusing. Poco chiare le ipotesi. additive noise.
         I. Ekren, I. Kukavica and  M. Ziane.
        \newblock Existence of invariant measures for the stochastic damped  Schr\"odinger equation.
        \newblock {\it  Stoch. Partial Differ. Equ. Anal. Comput.} 5(3): 343--367, 2017.

\bibitem[Fe23]{Fer}
B. ~Ferrario.
    \newblock On 2D Eulerian limits \`a la Kuksin.
    \newblock {\it  J. Differential Equations} 342:   1--20, 2023.


\bibitem[FG95]{FG}
F. ~Flandoli  and D.~ G\c{a}tarek.
        \newblock Martingale and stationary solutions for stochastic Navier-Stokes equations.
         \newblock {\it Probab. Theory Related Fields}  102(3): 367--391, 1995. 

\bibitem[Gr16]{Grubb}
G.~Grubb.
     \newblock Regularity of spectral fractional Dirichlet and Neumann problems
     \newblock  {\it Math. Nachr}. 289(7):   831--844,  2016.

\bibitem[Jak86]{Jak86}
     A.~Jakubowski.
      \newblock On the Skorokhod topology.
     \newblock {\it Ann. Inst. H. Poincar\'e Probab. Stat.}, 22: 263--285, 1986.

\bibitem[Jak98]{Jak98}
	A.~Jakubowski.
	\newblock The almost sure {Skorokhod} representation for subsequences in
	nonmetric spaces.
	\newblock {\it Theory of Probability \& Its Applications}, 42(1):167--174,
	1998.

\bibitem[K06]{Kim}
        J.U. Kim.
        \newblock Invariant Measures for a Stochastic Nonlinear Schr\"odinger Equation.
        \newblock  {\it Indiana University Mathematics Journal}, 55(2): 687--717,
        2006.


\bibitem[K99]{Kuksin99}
        S.B. Kuksin. 
        \newblock   A stochastic nonlinear Schr\"odinger equation I. A priori estimates.
        \newblock {\it Tr. Mat. Inst. Stekl.} 225, 232--256, 1999.

\bibitem[K04]{K04}
       S.B.  Kuksin.
        \newblock The Eulerian limit for 2D statistical hydrodynamics. 
        \newblock {\it J. Statist. Phys.} 115(1-2): 469--492, 2004. 
  
\bibitem[K10]{Kuksin-KdV}
         S.B. Kuksin. 
        \newblock  Damped-driven KdV and effective equations for long-time behaviour of its solutions. 
        \newblock  {\it Geom. Funct. Anal.} 20: 1431--1463,   2010.

\bibitem[K13]{Kuksin}
S.B. Kuksin.
      \newblock Weakly nonlinear stochastic CGL equations.
      \newblock {\it Annales de l’Institut Henri Poincar\'e - Probabilit\'es et Statistiques}
49(4):1033--1056, 2013.

\bibitem[KP08]{KP-KdV}
         S.B. Kuksin and A. L. Piatnitski. 
      \newblock  Khasminskii–Whitham averaging for randomly perturbed KdV equation. 
            \newblock {\it J. Math. Pures Appl.} 89: 400--428, 2008.
      
\bibitem[KS04]{KuksinS}
S.B. Kuksin and A. Shirikyan. 
      \newblock Randomly forced CGL equation: stationary measures and the inviscid limit. 
      \newblock   {\it J. Phys. A Math. Gen.} 37:  1--18, 2004.

\bibitem[Kun13]{Kunze_2013_Yamada} M.~Kunze.
\newblock On a class of martingale problems on {Banach} spaces.
\newblock {\it Electronic Journal of Probability}, 18(104):1--30, 2013.

\bibitem[L02]{La}
N. Laskin.
      \newblock Fractional Schr\"odinger equation.
       \newblock {\it Phys. Review E} 66: 056108, 2002.

\bibitem[LM72]{LionsMag}       
    J.L. Lions  and E. Magenes.
      \newblock {\it Non-Homogeneous Boundary Value Problems and Applications}.
      \newblock Vol 1, Springer, 1972.
            
\bibitem[On04]{Ondrejat_2004_Uniqueness}       M. Ondrej\'at.        \newblock  Uniqueness for stochastic evolution equations in Banach spaces.     \newblock {\it Dissertationes Math.}, 426:1--63, 2004.
                  
        
\bibitem[S11]{Sh11}     
        A.  Shirikyan.  
      \newblock Local times for solutions of the complex Ginzburg–Landau equation and the inviscid limit.
            \newblock {\it J. Math. Anal. Appl.} 384: 130--137, 2011.
        
 \bibitem[SuSu99]{Sulem}       
        C.  Sulem and P.-L. Sulem.
       \newblock {\it  The nonlinear Schr\"odinger equation. Self-focusing and wave collapse.}
       \newblock  Applied Mathematical Sciences, 139. Springer--Verlag, New York, 1999.


\bibitem[S21]{S}
        M. Sy, 
        \newblock Almost sure global well-posedness for the energy supercritical Schr\"odinger equations,
        \newblock {\it J. Math. Pures Appl.} (9) 154 (2021), 108-145.

\bibitem[SY22]{SY}
        \newblock M. Sy and X. Yu, Global well-posedness and long-time behavior of the fractional NLS,
        \newblock {\it Stoch. Partial Differ. Equ. Anal. Comput.} 10 (2022), no. 4, 1261-1317.

        
% \bibitem[Te77]{temam}  {\color{purple} TOGLIERE?}       R. Temam.          \newblock  {\it Navier-Stokes Equations. Theory and Numerica lAnalysis.}             Studies in Mathematics and Its Applications, vol. 2. North-Holland, Amsterdam (1977)
        
\bibitem[Tr78]{Triebel}
            H. Triebel.
            \newblock {\it  Interpolation Theory, Function Spaces, Differential Operators}.  North-Holland Mathematical Library, 18. North-Holland Publishing Co., Amsterdam-New York,  1978.

\bibitem[Tz06]{Tzvetkov06}
N. Tzvetkov.
    \newblock Invariant measures for the nonlinear Schr\"odinger equation on the disc.
     \newblock {\it Dyn. Partial Differ. Equ.} 3: 111-160, 2006.

\bibitem[Tz08]{Tzvetkov}
N. Tzvetkov.
    \newblock Invariant measures for the defocusing nonlinear Schr\"odinger equation,.
    \newblock {\it Ann. Inst. Fourier} 58(7): 2543--2604,  2008.

\bibitem[VF88]{vishik}
      M.I. Vishik and A. V. Fursikov.
       \newblock {\it  Mathematical Problems in Statistical Hydromechanics}.
        Kluwer, Dordrecht, 1988.

       \end{thebibliography}
\end{document}